\documentclass[11pt,a4paper]{article}

\usepackage{cmap}
\usepackage[T1]{fontenc}

\usepackage{amsmath,amsthm,amssymb,latexsym,graphicx,mathrsfs}
\usepackage{secdot}
\sectiondot{subsection}
\usepackage{a4wide}
\usepackage{float}

\usepackage{enumitem}

\usepackage{marginnote}

\usepackage{xcolor}
\usepackage{url}
\usepackage{hyperref}
\definecolor{ForestGreen}{rgb}{0.1,0.6,0.05}
\definecolor{EgyptBlue}{rgb}{0.063,0.1,0.6}
\definecolor{RipeOlive}{HTML}{556B2F}
\hypersetup{
	colorlinks=true,
	linkcolor=EgyptBlue,         
	citecolor=ForestGreen,
	urlcolor=RipeOlive
}

\usepackage[hyperpageref]{backref}
\usepackage{epstopdf}
\newtheorem{theorem}{Theorem}
\newtheorem{proposition}[theorem]{Proposition}
\newtheorem{lemma}[theorem]{Lemma}

\theoremstyle{definition}
\newtheorem{definition}[theorem]{Definition}
\newtheorem{remark}[theorem]{Remark}

\numberwithin{equation}{section}
\numberwithin{theorem}{section}

\numberwithin{equation}{section}
\numberwithin{theorem}{section}

\newenvironment{proof*}[1]{\begin{trivlist}\item[\hskip%
		\labelsep{{\bf Proof of \/{\rm\bf #1.}}~}]\rm}%
	{\hfill\qed\rm\end{trivlist}}

\newcommand{\W}{W_0^{1,p}}

\newcommand{\intO}{\int_\Omega}

\newcommand{\I}{I_{\lambda}}
\newcommand{\In}{I_{\lambda_n}}
\newcommand{\wI}{\widetilde{I}_{\lambda}}

\newcommand{\wIm}{\widetilde{I}_{\lambda}^\mu}
\newcommand{\E}{E_\lambda}
\newcommand{\wE}{\widetilde{E}_{\lambda}}

\newcommand{\A}{\mathcal{A}}

\newcommand{\J}{J_{\lambda}}

\newcommand{\N}{\mathcal{N}_{\lambda}}
\newcommand{\wN}{\widetilde{\mathcal{N}}_{\lambda}}

\DeclareRobustCommand\lambdaM{\lambda}

\title{
	\vspace*{-1cm}
	On subhomogeneous indefinite $p$-Laplace equations in supercritical spectral interval} 
\author{ 
	\normalsize Vladimir Bobkov\\ 
	{\small  Institute of Mathematics, Ufa Federal Research Centre, RAS}\\ 
	{\small Chernyshevsky str.\ 112, Ufa 450008, Russia}\\
	{\small e-mail: bobkov@matem.anrb.ru}\\[0.5em] 	
	\normalsize Mieko Tanaka\\
	{\small Department of  Mathematics, 
		Tokyo University of Science}\\
	{\small Kagurazaka 1-3, Shinjyuku-ku, Tokyo 162-8601, Japan}\\
	{\small e-mail: miekotanaka@rs.tus.ac.jp} 
}

\date{}

\begin{document}
	\maketitle 
	
	\begin{abstract}
		We study the existence, multiplicity, and certain qualitative properties of solutions to the zero Dirichlet problem for the equation 
		$-\Delta_p u = \lambda |u|^{p-2}u + a(x)|u|^{q-2}u$ in a bounded domain $\Omega \subset \mathbb{R}^N$, where $1<q<p$, $\lambda\in\mathbb{R}$, and $a$ is a continuous sign-changing weight function.
		Our primary interest concerns ground states and nonnegative solutions which are positive in $\{x\in \Omega: a(x)>0\}$, when the parameter $\lambda$ lies in a neighborhood of the critical value $\lambda^* := \inf\left\{\int_\Omega |\nabla u|^p \, dx/\int_\Omega |u|^p \, dx: u\in\W(\Omega) \setminus \{0\},\ \intO a|u|^q\,dx \geq 0\,\right\}$. 
		Among main results, we show that if $p>2q$ and either $\intO a\varphi_p^q\,dx=0$ or $\intO a\varphi_p^q\,dx>0$ is sufficiently small, then such solutions do exist in a \textit{right} neighborhood of $\lambda^*$. 
		Here $\varphi_p$ is the first eigenfunction of the Dirichlet $p$-Laplacian in $\Omega$.
		This existence phenomenon is of a purely subhomogeneous and nonlinear nature, since either in the superhomogeneous case $q>p$ or in the sublinear case $q<p=2$ the nonexistence takes place for any $\lambda \geq \lambda^*$. 
		Moreover, we prove that if $p>2q$ and $\intO a\varphi_p^q\,dx>0$ is sufficiently small, then there exist \textit{three} nonzero nonnegative solutions in a \textit{left} neighborhood of $\lambda^*$, two of which are strictly positive in $\{x\in \Omega: a(x)>0\}$.
		
		
	\end{abstract}

	\section{Introduction}\label{sec:intro}
	In the present work, we study the boundary value problem
	\begin{equation}\label{eq:P}
		\tag{$P_{\lambdaM}$}
		\left\{
		\begin{aligned}
			-\Delta_p u &= \lambda |u|^{p-2}u+ a(x)|u|^{q-2}u 
			&&\text{in}\ \Omega, \\
			u&=0 &&\text{on}\ \partial \Omega,
		\end{aligned}
		\right.
	\end{equation}
	where the $p$-Laplace operator $\Delta_p$ with $p>1$ acts formally as $\Delta_p u = \text{div}\left(|\nabla u|^{p-2} \nabla u \right)$, 
	$\lambda\in \mathbb{R}$ is a parameter, and $\Omega \subset \mathbb{R}^N$ is a bounded domain, $N \geq 1$. In the case $N \geq 2$, we ask the boundary $\partial\Omega$ of $\Omega$ to be at least $C^2$-smooth.
	Unless explicitly stated otherwise, we always assume that $1<q < p$ and 
	the weight function $a \in C(\overline{\Omega}) \setminus \{0\}$ is sign-changing, i.e., $\Omega_a^\pm \neq \emptyset$, where
	$$
	\Omega_a^+ := \{x \in \Omega:~ a(x)>0\}
	\quad \text{and} \quad
	\Omega_a^- := \{x \in \Omega:~ a(x)<0\}. 
	$$
	We also denote
	$$
	\Omega_a^0 := \{x \in \Omega:~ a(x)=0\}. 
	$$
	Because of the two latter assumptions, 
	the problem \eqref{eq:P} is called \textit{subhomogeneous} ($q<p$) and \textit{indefinite} ($\Omega_a^\pm \neq \emptyset$). 
	Nevertheless, a few of our results remain valid when $a$ is sign-constant. 
	
	Although in the case $p=2$ the problem \eqref{eq:P} can be considered pointwisely, in the general case $p>1$ this problem is understood in the weak sense. 
	Namely, we say that $u\in\W(\Omega)$ is a (weak) solution of \eqref{eq:P} if the equality
	\begin{equation*}
		\intO |\nabla u|^{p-2}\nabla u\nabla\varphi \,dx 
		= \lambda \intO |u|^{p-2}u\varphi\,dx
		+\intO a|u|^{q-2}u\varphi\,dx 
	\end{equation*}
	holds for all $\varphi\in \W(\Omega)$. 
	It is not hard to see that such solutions are precisely critical points of the energy functional $ \I \in C^1(\W(\Omega),\mathbb{R})$ defined as
	\begin{equation}\label{def:functionals} 
		\I (u):=\dfrac{1}{p}\,E_\lambda(u)-\dfrac{1}{q}\intO a|u|^q\,dx,
		\quad \text{where} \quad 
		E_\lambda(u):=\intO |\nabla u|^{p}\, dx-\lambda \intO |u|^{p} \, dx. 
	\end{equation}
	
	\begin{remark}\label{rem:reg} 
		Any solution of \eqref{eq:P} belongs to $C^{1,\beta}_0(\overline{\Omega})$ with some $\beta\in(0,1)$. 
		In fact, if $u$ is a solution of \eqref{eq:P}, 
		then $u\in
		L^\infty(\Omega)$, which can be shown by the standard Moser iteration process (see, e.g., 
		\cite[Appendix A]{MMT}). 
		Hence, the regularity up to the boundary given by \cite[Theorem~1]{Lieberman} 
		ensures that $u\in
		C^{1,\beta}_0(\overline{\Omega})$.
	\end{remark} 
	
	The distinguishing feature of \eqref{eq:P} to comprise a mixture of the subhomogeneous (even non-Lipschitz when $q < 2$) and indefinite natures made this simply looking problem a subject of considerable interest over the last thirty years, and several important contributions to its understanding have been provided only recently, even in the case $p=2$ and $\lambda=0$.
	In particular, thanks to the assumption $q<p$ and the presence of the sign-changing weight $a$, 
	{\it nonnegative solutions of \eqref{eq:P} do not obey, in general, the strong maximum principle}.
	(See, for instance, \cite{PS2} for a comprehensive summary of maximum principles.)
	As a result, an emergence of the so-called {\it dead core} solutions, i.e., nonzero solutions vanishing in a subdomain of $\Omega$, can be observed.
	We refer the reader to, e.g., \cite{BPT1,diaz,KQU} for a deeper discussion which includes explicit constructions of such solutions, their qualitative properties, and description of earlier results in this direction.
	
	Apart from the dead core solutions, two other classes of solutions of \eqref{eq:P} are of significant importance:
	\begin{enumerate}
		\item Ground states and nonnegative ground states.
		\begin{definition}\label{def:ground}
			We say that a nonzero critical point $u$ of $\I$ is a \textit{ground state} (or, equivalently, a \textit{least energy solution}) of \eqref{eq:P} if $\I(u) \leq \I(v)$ for any nonzero critical point $v$ of $\I$.
		\end{definition}
		\item Solutions which are (strictly)
		positive\footnote{Throughout this work, the words ``positive'' and ``negative'' mean ``$>0$\,'' and ``$<0$\,'', respectively. The word ``strictly'' will be  used occasionally for accentuation and clarification.} 
		in $\Omega_a^+$.
	\end{enumerate}

	In general, ground states may not always have a strictly constant sign in $\Omega_a^+$. 
	This is indicated by \cite{BDO,DH,DHI1,FLS}, in the context of study of compact support solutions of \eqref{eq:P} (with sign-constant weight) for sufficiently large $\lambda$, see also Proposition \ref{prop:nonlipschitz} below.
	Furthermore, \eqref{eq:P} might possess \textit{sign-changing} ground states, as it follows, e.g., from \cite[Theorem 1.8]{KQU} (see, more precisely, \cite[Remark 1.9]{KQU} for the construction of a two-bumps solution of \eqref{eq:P} in 1D case, one bump of which might be reflected over the $x$-axis to make this solution sign-changing).
	On the other hand, under certain assumptions, there exist nonnegative mountain pass solutions of \eqref{eq:P} which are positive in $\Omega_a^+$, but they are not ground states, see, e.g., \cite[Theorem 1.4]{KQU}.
	Thus, the classes of ground states and solutions which are positive in $\Omega_a^+$ are independent.
	These two classes of solutions are the main objects of our study. Several known results on the existence of such solutions with respect to the parameter $\lambda$ are collected in Section~\ref{sec:subcritical}.
	
	The combination of subhomogeneous and indefinite natures suggests the consideration of nonnegative solutions of the problem \eqref{eq:P} in the following three ranges of the parameter:
	\begin{enumerate}
		\item Nonpositive values of $\lambda$, with the special emphasis on the case $\lambda=0$.
		Such values of $\lambda$ allow, in particular, a deeper investigation of the formation of dead cores or positivity, and the study of uniqueness issues. 
		We refer to the series of articles \cite{KQU1,KQU,KQU2} and references therein.
		\item ``Middle'' values of $\lambda$. 
		In this range, the primary interest is to consider the existence of negative energy ground states and the multiplicity of solutions which are positive in $\Omega_a^+$. 
		The results of the present work correspond to such values of $\lambda$, see Sections \ref{sec:subcritical} and \ref{sec:mainresults} for an overview.
		\item ``Large'' values of $\lambda$. 
		In this range, there are no solutions positive in $\Omega_a^+$ (see Proposition~\ref{prop:bounded} below). 
		However, nonzero solutions which vanish in $\Omega_a^+$ might exist.
		The existence and properties of such solutions and positive energy ground states are of the main importance.
		We refer to \cite{DH,DHI1,FLS} for the consideration of these and related issues.
	\end{enumerate}
	
	Let us mention that the problem \eqref{eq:P} in the superhomogeneous regime $q>p$ is somewhat more developed.\footnote{When commenting on the superhomogeneous case $q>p$, we always assume that $q<p^*$, where $p^*$ is the critical Sobolev exponent for $N \geq 3$, and $p^*=+\infty$ for $N=1,2$.} 
	The strong maximum principle holds in this case, which yields the positivity of any nonzero nonnegative solution. 
	In particular, there are no nonnegative dead core solutions, and there are no positive solutions for sufficiently large $\lambda$.
	Although the base-level multiplicity information for $q>p$ and $q<p$ for the ``middle'' values of $\lambda$ looks similar, the structure and properties of the corresponding solution sets are completely different. 
	The difference is amplified by our main results stated in Section \ref{sec:mainresults}.
	We refer the reader to the list of classical works \cite{alama-del,alama-tar,berest,drabek-poh,ilyasov1,ouyang} and to a more resent article \cite{IS} on the problem \eqref{eq:P} in the superhomogeneous case $q>p$.
	In the following subsections, we will comment on several known results from these works in more detail.
	
	Finally, we mention that even when the weight $a$ is sign-constant, various recent contributions on the subhomogeneous problem \eqref{eq:P} have been made.
	In the case $a \geq 0$, in which the strong maximum principle holds, we refer to \cite{BDSWPT} for the existence of least energy nodal solutions, to \cite{brasco} for an overview on the corresponding eigenvalue problem ($\lambda=0$ and $a(x)=\mu$), and to \cite{kaji1} for the existence of infinitely many solutions with negative energy converging to zero. 
	On the other hand, in the case $a \leq 0$, the problem \eqref{eq:P} has been investigated in, e.g.,  \cite{BDO,DH,DHI1,FLS}, in the context of study of compact support and ``flat'' solutions.

	\medskip
	In the following subsection, we briefly overview several known results, as well as new ones,  on the problem \eqref{eq:P} which are the basis of our further analysis provided in Section \ref{sec:mainresults}.

	\subsection{Behavior in subcritical spectral interval}\label{sec:subcritical}
	
	Hereinafter, for brevity, we denote by $\W:= W_0^{1,p}(\Omega)$ the standard Sobolev space and by $\| \cdot \|_r$ the standard norm in $L^r(\Omega)$, $r \in [1,+\infty]$. 
	The first eigenvalue of the Dirichlet $p$-Laplacian is denoted by
	$$
	\lambda_1(p) 
	:= 
	\inf
	\left\{
	\frac{\|\nabla u\|_p^p}{\|u\|_p^p}:~ u\in\W \setminus \{0\}	
	\right\}.
	$$
	This eigenvalue is simple and isolated, the corresponding first eigenfunction $\varphi_p$ belongs to $C^{1,\beta}_0(\overline{\Omega})$, $\beta\in(0,1)$, and $\varphi_p$ is of a constant sign in $\Omega$, see, e.g., \cite{anane1987}.
	We assume, without loss of generality, that $\varphi_p>0$ in $\Omega$ and $\|\nabla\varphi_p\|_p=1$.

	Noting that the energy functional $\I$ defined by \eqref{def:functionals} is bounded from below and coercive when $\lambda<\lambda_1(p)$, it is not hard to show that in this case $\I$ has a global minimizer, this minimizer has negative energy, it is a ground state of \eqref{eq:P}, and any nonnegative minimizer is positive in $\Omega_a^+$, see, e.g., \cite[Section 1.2]{KQU}.
	On the other hand, by taking $u=t\varphi_p$ and letting $t \to +\infty$, we see that $\inf\{ \I(u): u\in\W\} = -\infty$ when $\lambda>\lambda_1(p)$, and hence no global minimizer exists. 
	That is, ground states of \eqref{eq:P} cannot be characterized as global minimizers of $\I$ for $\lambda>\lambda_1(p)$.
	In the borderline case $\lambda=\lambda_1(p)$, the existence of global minimizers is a more subtle issue which depends on the settings of the problem. 
	We provide a corresponding result in Section \ref{sec:mainresults}. 
	
	In order to find a possibly larger spectral interval of the existence of solutions to \eqref{eq:P}, it is natural to investigate minimizers of $\I$ over a subset of $\W$ described by the
	{\it Nehari manifold}
	\begin{equation}\label{eq:nehari}
		\N :=\left\{u \in \W \setminus\{0\}:\, 
		\langle \I^\prime(u),u \rangle=0\,\right\} 
		=\left\{u \in \W \setminus\{0\}:\, 
		\E(u)=\intO a |u|^q\,dx\,\right\}.
	\end{equation}
	Consider the corresponding minimal level of $\I$:
	\begin{equation}\label{eq:M}
		{M}(\lambda)
		:=
		\inf\{ \I (u):~ u\in \N \}.
	\end{equation}
	The following result asserts that ground states of \eqref{eq:P} can be characterized as minimizers of $M(\lambda)$ whenever $M(\lambda)$ is attained, cf.\ \cite[Theorem 1.4 (1)]{KQU}.
	\begin{proposition}\label{prop:charact}
	\marginnote{{\scriptsize Proof on\\ p.\pageref{page:prop:charact}}}
		Let $\lambda \in \mathbb{R}$ be such that $M(\lambda)$ is attained. 
		Then any corresponding minimizer $u$ is a ground state of \eqref{eq:P}, $\I(u)<0$, $u$ is a local minimum point of $\I$, and either $u>0$ or $u<0$ in $\Omega_a^+$. 
		Moreover, $u$ is a global minimum point of $\I$ provided $\lambda \leq \lambda_1(p)$.
	\end{proposition}
	
	\begin{remark}
		Evidently, if $u$ is a minimizer of $M(\lambda)$, then so is $|u|$.
		That is, the attainability of $M(\lambda)$ implies the existence of a nonnegative minimizer which is positive in $\Omega_a^+$.
		On the other hand, $M(\lambda)$ might possess sign-changing minimizers, as we indicated above.
	\end{remark}
	
	Proposition \ref{prop:charact} motivates the investigation of assumptions under which $M(\lambda)$ is attained. 
	For this purpose, we introduce the following critical value of the parameter $\lambda$:
	\begin{equation}\label{eq:lambda+*}
		\lambda^* := \inf\left\{\frac{\|\nabla u\|_p^p}{\|u\|_p^p}:~ u\in\W \setminus \{0\},\ \intO a |u|^q\,dx \geq 0\,\right\},
	\end{equation}
	which will play a significant role throughout the work.
	We readily have $\lambda_1(p) \leq \lambda^*$, and the simplicity of $\lambda_1(p)$ yields $\lambda_1(p)<\lambda^*$ if and only if $\intO a \varphi_p^q \,dx < 0$, see Proposition \ref{prop:+} below for details.
	The critical value $\lambda^*$ has the property that the set $\N$ is a $C^1$-manifold of codimension~$1$ for any $\lambda<\lambda^*$, see, e.g., \cite{brown,QS}. 
	In particular, the following information can be obtained (see, e.g., \cite[Lemma 2.3 (1) and Remark 2.4]{KQU}).
	\begin{theorem}\label{thm:GS} 
		The following assertions hold:
		\begin{enumerate}[label={\rm(\roman*)}]
			\item\label{thm:GS:1}  If $\lambda<\lambda^*$, then 
			$M(\lambda) \in (-\infty,0)$ and it is attained. 
			\item\label{thm:GS:2}  If $\lambda>\lambda^*$, then $M(\lambda)=-\infty$. 
		\end{enumerate}
	\end{theorem}

	Notice that $M(\lambda)$ can be also characterized as 
	$$
	{M}(\lambda)
	=
	{M}^+(\lambda)
	:=
	\inf\{ \I (u):~ u\in \N \cap\A^+ \},
	$$
	where 
	\begin{equation}\label{eq:A+}
		\A^+
		:=\left\{u \in \W:~ 
		\intO a |u|^q\,dx>0\,\right\},
	\end{equation}
	see Remark \ref{rem:m<0} below.
	In the following proposition, we collect a few qualitative results on the behavior of $M(\lambda)$.
	We refer the reader to Section \ref{sec:qualitative} for additional information on the behavior of $M(\lambda)$ and the corresponding minimizers.
	\begin{proposition}\label{prop:properties-m}
	\marginnote{{\scriptsize Proof on\\ p.\pageref{page:proof:properties-m}}}
		The extended function ${M}: \mathbb{R} \mapsto \mathbb{R} \cup \{-\infty\}$ has the following properties:
		\begin{enumerate}[label={\rm(\roman*)}]
			\item\label{prop:properties-m:1} ${M}$ is nonincreasing on $\mathbb{R}$ and decreasing on $(-\infty,\lambda^*]$.
			\item\label{prop:properties-m:3} ${M}$ is continuous on $(-\infty,\lambda^*)$.  
			\item\label{prop:properties-m:4}  $M(\lambda) \to M(\lambda^*)$ as $\lambda \to \lambda^*-0$.
		\end{enumerate}
	\end{proposition}

	In view of the discussion on the global minimizers above, Theorem \ref{thm:GS} \ref{thm:GS:1} provides a nontrivial information if $\lambda_1(p) < \lambda^*$, i.e., when $\intO a\varphi_p^q\,dx< 0$.
	Moreover, if $\intO a\varphi_p^q\,dx< 0$, then for any $\lambda \in (\lambda_1(p),\lambda^*)$ there exists another nonnegative solution of \eqref{eq:P} which is positive in $\Omega_a^+$. 
	This solution has the least energy among all positive energy solutions and can be characterized as a minimizer of 
	\begin{equation}\label{eq:M-}
		{M}^-(\lambda)
		:=
		\inf\{ \I (u):~ u\in \N \cap\A^- \},
	\end{equation}
	where 
	\begin{equation}\label{eq:A-}
		\A^-
		:=\left\{u \in \W:~ 
		\intO a |u|^q\,dx<0\,\right\}.
	\end{equation}
	More precisely, the following result is given by \cite[Theorem 1.4 (2)]{KQU}.
	\begin{theorem}\label{prop:m-minus}
		Let $\intO a\varphi_p^q\,dx< 0$ and $\lambda \in (\lambda_1(p), \lambda^*)$. 
		Then ${M}^-(\lambda)>0$, it is attained, and there exists a nonnegative minimizer of ${M}^-(\lambda)$ which is positive in $\Omega_a^+$. 
		Moreover, 
		${M}^-(\lambda)$ coincides with a mountain pass value of $\I$ such that 
		$$
		{M}^-(\lambda)
		=\inf_{u\in \W\setminus\{0\}}\sup_{t>0}\, \I (tu). 
		$$
	\end{theorem}
	\begin{remark}
		Under the assumptions of Theorem \ref{prop:m-minus}, any minimizer of ${M}^-(\lambda)$ is a saddle point  of $\I$ (i.e., neither a local minimum nor local maximum). 
		The proof of this fact can be obtained along the same lines as the proof of \cite[Proposition 2.9 (ii)]{BobkovTanaka2017}.
	\end{remark}
	
	Let us collect a few qualitative results on the behavior of $M^-(\lambda)$, see also Section \ref{sec:qualitative} for further properties of $M^-(\lambda)$.
	\begin{proposition}\label{prop:PCV} 
	\marginnote{{\scriptsize Proof on\\ p.\pageref{page:proof:PCV}}}	
		Let $\intO a\varphi_p^q\,dx< 0$. 
		Then the following assertions hold:
		\begin{enumerate}[label={\rm(\roman*)}]	
			\item\label{prop:PCV:1} $M^-$ 
			is nonincreasing on $(\lambda_1(p),+\infty)$ and 
			decreasing on $(\lambda_1(p),\lambda^*]$. 
			\item\label{prop:PCV:3} $M^-$ is continuous on $(\lambda_1(p),\lambda^*)$.
			\item\label{prop:PCV:2} $M^-(\lambda) \to +\infty$ as $\lambda\to \lambda_1(p)+0$.  
			\item\label{prop:PCV:5} $M^-(\lambda) \to M^-(\lambda^*)$ as $\lambda\to \lambda^*-0$.
			\item\label{prop:PCV:4} $M^-(\lambda)=0$ for any $\lambda > \lambda^*$. 
		\end{enumerate}
	\end{proposition}
	
	\begin{remark}\label{rem:positive-least-energy} 
		In the case $\intO a\varphi_p^q\,dx \ge 0$, it can be shown that for any $\lambda \in \mathbb{R}$ one has either $\N \cap\A^-=\emptyset$ (and hence $M^-(\lambda)$ is undefined) or $M^-(\lambda)=0$, see Lemma \ref{lem:NehariEmpty} and Remark \ref{rem:positive-least-energy_2}. 
	\end{remark}

	Propositions \ref{prop:properties-m} and \ref{prop:PCV}, together with other results from Section \ref{sec:qualitative}, supplement the information on $M^\pm$ from \cite[Section 4]{brown} and \cite[Section 4]{KQU}, and we refer the reader to these works for additional results on the attainability and qualitative properties of $M^\pm$ and their minimizers.
	Let us mention that there are several other definitions of the minimization problems ${M}^\pm$ equivalent to  \eqref{eq:M} and \eqref{eq:M-}.
	For instance, the authors of \cite{KQU} use the minimization over $\A^\pm$ (without explicit reference to $\N$), while the authors of \cite{QS} deal with definitions like \eqref{eq:M} and \eqref{eq:M-}, but containing the truncated integrals $\intO u_+^p \, dx$ and $\intO a u_+^q \, dx$ instead of their untruncated versions $\intO |u|^p \, dx$ and $\intO a |u|^q \, dx$.
	Here, $u_+ := \max\{u,0\}$.
	It is not hard to show that all such definitions coincide in the sense that they describe the same critical levels. 
	We give a few rigorous results in this direction in Section \ref{sec:relation}.

	\medskip
	It is important to remark that Theorem \ref{thm:GS} does not provide an answer to the attainability of $M(\lambda^*)$.
	In general, the information on existence of nonnegative solutions to \eqref{eq:P} in the supercritical spectral interval $[\lambda^*,+\infty)$ is very limited. 
	We refer the reader to \cite{kaji1,moroz} for the existence of abstract solutions to \eqref{eq:P} (without information on the sign) for all $\lambda$, and to \cite{DHI1} for the existence of nonnegative solutions with compact support in $\Omega_a^-$ for sufficiently large $\lambda$ (see Proposition \ref{prop:nonlipschitz} below).
	Several nontrivial results on the existence of nonnegative solutions in a right neighborhood of $\lambda^*$ have been obtained recently in \cite{QS}.
	The authors of \cite{QS} develop a theory applicable to general variational functionals consisted of two homogeneous parts obeying certain assumptions. 
	In the case of the problem \eqref{eq:P}, these assumptions are satisfied only when $\intO a \varphi_p^q \,dx <0$.
	The main aim of the present work is to contribute to the available theory of existence and multiplicity of nonnegative solutions to \eqref{eq:P} by studying in more detail the supercritical spectral interval $[\lambda^*,+\infty)$ for all signs of $\intO a \varphi_p^q \,dx$.
	In particular, we show that if $p>2q$ and either $\intO a \varphi_p^q \,dx =0$ or $\intO a \varphi_p^q \,dx >0$ is ``sufficiently small'', then the problem \eqref{eq:P} exhibits nontrivial existence and multiplicity phenomena which are impossible in the cases $q>p$ and $q<p=2$. 
	
	Precise statements of our main results are given in the following subsection.

	\subsection{Behavior in supercritical spectral interval. Statements of main results}\label{sec:mainresults}
	
	Let us recall that the critical value $\lambda^*$ defined by \eqref{eq:lambda+*} is the threshold dividing the existence and nonexistence of minimizers of $M(\lambda)$, see Theorem \ref{thm:GS}.
	Our first main result describes the attainability of $M(\lambda^*)$ with respect to the sign of $\intO a\varphi_p^q\,dx$. 
	The second assertion of this theorem gives the most essential contribution.
	
	\begin{theorem}\label{NoGroundStates}
		\marginnote{{\scriptsize Proof on\\ p.\pageref{sec:proof:NoGroundS:1}}}
		The following assertions hold: 
		\begin{enumerate}[label={\rm(\roman*)}]
			\item\label{NoGroundStates:1} Let $\intO a\varphi_p^q\,dx< 0$. 
			Then $M(\lambda^*) \in (-\infty,0)$ and it is attained. 
			\item\label{NoGroundStates:2} Let $\intO a\varphi_p^q\,dx= 0$. 
			Assume that $\partial \Omega$ is connected provided $p \geq 2q$ and $N\ge 2$. 
			Then $M(\lambda^*) \in (-\infty,0)$ if and only if $p \geq 2q$. 
			Furthermore, if $p > 2q$, 
			then $M(\lambda^*)$ is attained. 
			\item\label{NoGroundStates:3} Let $\intO a\varphi_p^q\,dx> 0$.
			Then $M(\lambda^*)=-\infty$. 
		\end{enumerate} 
	\end{theorem}
	\begin{remark}\label{rem:NoGS}
		The assertion \ref{NoGroundStates:1} of Theorem \ref{NoGroundStates} can be found in \cite[Corollary 3.5]{QS} under slightly different, but same in essence, assumptions on \eqref{eq:P}. 
		The assertion \ref{NoGroundStates:3} of Theorem~\ref{NoGroundStates} is in agreement with \cite[Theorem 4]{brown}.
	\end{remark}
	
	\begin{remark} 
		Recall that $\lambda^*=\lambda_1(p)$ whenever $\intO a \varphi_p^q \,dx \geq 0$, see Proposition \ref{prop:+}.
		Thus, a ground state at $\lambda=\lambda_1(p)$ provided by
		Theorem \ref{NoGroundStates} \ref{NoGroundStates:2} for $p>2q$ is the global minimum point of $ \I$, 
		see Proposition \ref{prop:charact}. 
		The attainability of  $M(\lambda^*)$ in the case $\intO a\varphi_p^q\,dx= 0$ and $p=2q$ remains unknown to us.
	\end{remark} 
	
	The main question arising from Theorems \ref{thm:GS} and \ref{NoGroundStates} is whether $\lambda^*$ is a terminal point for the existence of nonnegative ground states and nonnegative solutions which are positive in $\Omega_a^+$, or these solutions can be obtained in a right neighborhood of $\lambda^*$.
	Let us mention, for comparison, that in the superhomogeneous regime $q>p$ the answer depends solely on the sign of $\intO a \varphi_p^q \, dx$.
	Namely, $\lambda^*$ \textit{is} the terminal point if and only if $\intO a \varphi_p^q \, dx \geq 0$, see, e.g., \cite{ilyasov2}.
	The situation in the subhomogeneous regime $q<p$ appears to be more delicate since not only the sign of $\intO a \varphi_p^q \, dx$ matters, but also a nontrivial relation between the exponents $p$ and $q$ starts to play a significant role.
	In Theorems \ref{thm:1}, \ref{thm:3} and Theorem \ref{NoSol} below, we provide two groups of opposite results in this direction, for different relations between $p$ and $q$. These theorems are among our main results.

	Since our primary interest concerns nonnegative solutions of \eqref{eq:P}, it will be convenient to work with the following \textit{truncated} energy functional:
	\begin{equation}\label{def:functionals-t} 
		\wI (u):=\frac{1}{p}\,\widetilde{E}_\lambda(u)-\frac{1}{q}\intO a u_+^q\,dx,
		\quad \text{where} \quad 
		\widetilde{E}_\lambda(u):=\|\nabla u\|_p^p-\lambda \|u_+\|_p^p.
	\end{equation} 
	Here, we denote $u_\pm =\max\{\pm u, 0\}$, i.e., $u = u_+-u_-$.
	Notice that $ \wI \in C^1(\W,\mathbb{R})$ and $\wI$ is weakly lower semicontinuous.
	Any critical point $u \in \W$ of $ \wI$ is a \textit{nonnegative} solution to \eqref{eq:P}, i.e., $u_- \equiv 0$ in $\Omega$, which follows from the fact that 
	$
	0 
	= 
	\langle \wI'(u),u_-\rangle
	=
	-\|\nabla u_-\|_p^p.
	$
	We consider the corresponding truncated Nehari manifold: 
	\begin{equation*}\label{eq:nehari-t}
		\wN :=\left\{u \in \W \setminus\{0\}:~ 
		\wE(u)=\intO a u_+^q\,dx\,\right\},
	\end{equation*}
	and denote by $\widetilde{M}(\lambda)$ the minimal level of $\wI$ over $\wN$, i.e.,\footnote{Throughout this work, the diacritic ``tilde'' over a capital letter always corresponds to the presence of the truncated integrals $\intO u_+^p\,dx$ and $\intO a u_+^q\,dx$ instead of their untruncated counterparts $\intO |u|^p\,dx$ and $\intO a |u|^q\,dx$.}
	\begin{equation*}\label{eq:wM}
		\widetilde{M}(\lambda)
		:=
		\inf\{ \wI (u):~ u\in \wN \}.
	\end{equation*}
	Hereinafter, we will use the following notion, cf.\ Definition \ref{def:ground}.
	\begin{definition}\label{def:Ileast}
		We say that a nonzero critical point $u$ of $\wI$ is a \textit{least $\wI$-energy solution} of \eqref{eq:P} if $\wI(u) \leq  \wI(v)$ for any nonzero critical point $v$ of $\wI$. 
	\end{definition}
	\vspace{-7pt}
	\noindent
	In other words, a least $\wI$-energy solution is 
	a {\it nonnegative} solution with the least energy 
	{\it among all nonnegative solutions}, since 
	any nonnegative solution is a critical point of $\wI$. 
	
	\medskip
	The difference between this notion and the notion of ground state of \eqref{eq:P} (given by Definition \ref{def:ground} via the untruncated functional $\I$) is subtle. 
	Evidently, $\I$ has more critical points than $\wI$ since it also includes sign-changing solutions to \eqref{eq:P}. 
	In particular, while least $\wI$-energy solutions are always nonnegative, there might exist sign-changing ground states of \eqref{eq:P}, see the discussion at the beginning of the paper. 
	Nevertheless, we show in Proposition~\ref{prop:wm=m} below that these two notions correspond to the same critical level at least when either $M(\lambda)$ or $\widetilde{M}(\lambda)$ is attained.
	We provide several other results on the relation between $\I$ and $\wI$ in Section~\ref{sec:relation}.

	The attainability of $M(\lambda^*)$ described in Theorem \ref{NoGroundStates} \ref{NoGroundStates:1}, \ref{NoGroundStates:2} is a cornerstone for the following two main theorems which contain a nontrivial multiplicity information.
	
	\begin{theorem}\label{thm:1}
		\marginnote{{\scriptsize Proof on\\ p.\pageref{sec:thm1}}}
		Let one of the following assumptions be satisfied:
		\begin{enumerate}[label={\rm(\Roman*)}]
			\item\label{thm:1:1} $\intO a\varphi_p^q\,dx<0$.
			\item\label{thm:1:2} $\intO a\varphi_p^q\,dx=0$, $p>2q$, and $\partial\Omega$ is connected provided $N \geq 2$.
		\end{enumerate}
		Then there exist 
		$\Lambda$ and $\Lambda^*$ such that $\lambda^*<\Lambda^* \leq \Lambda < +\infty$ and 
		the following assertions hold:
		\begin{enumerate}[label={\rm(\roman*)}]
			\item\label{thm:1:r1}
			Let $\lambda \in (\lambda^*, \Lambda)$.
			Then \eqref{eq:P} possesses a nonnegative solution $u_\lambda$ such that $u_\lambda>0$ in $\Omega_a^+$ and $\wI(u_\lambda) < 0$.
			Moreover, \eqref{eq:P} possesses a least $\wI$-energy solution $w_\lambda$ such that $w_\lambda>0$ in some connected component of $\Omega_a^+$, and $\wI(w_\lambda) \leq \wI(u_\lambda) < 0$.
			\item\label{thm:1:r2}
			Let $\lambda \in (\lambda^*, \Lambda^*)$. 
			Then $u_\lambda$ is a local minimum point of $\wI$, and \eqref{eq:P} possesses another nonnegative solution $v_\lambda$ ($v_\lambda \neq u_\lambda, w_\lambda$) such that $v_\lambda>0$ in some connected component of $\Omega_a^+$, $v_\lambda$ is a mountain pass critical point of $\wI$, and $\wI(w_\lambda) \leq \wI(u_\lambda) < \wI(v_\lambda) < 0$.
			\item\label{thm:1:r4}
			Let $\lambda > \Lambda$.
			Then \eqref{eq:P} possesses no nonnegative solution which is positive in $\Omega_a^+$.
		\end{enumerate}
	\end{theorem}
	
	\begin{remark}
		We anticipate that the solutions $u_\lambda$ and $w_\lambda$ obtained in Theorem \ref{thm:1} \ref{thm:1:r1} coincide, at least for  $\lambda$ close to $\lambda^*$.
		Moreover, we expect that $w_\lambda$ and $v_\lambda$ are positive in the whole $\Omega_a^+$.
		See Figure \ref{fig:bifur} for the graphical representation of the results of Theorem \ref{thm:1}.
	\end{remark}

	\begin{figure}[h!]
		\center{
			\includegraphics[width=0.55\linewidth]{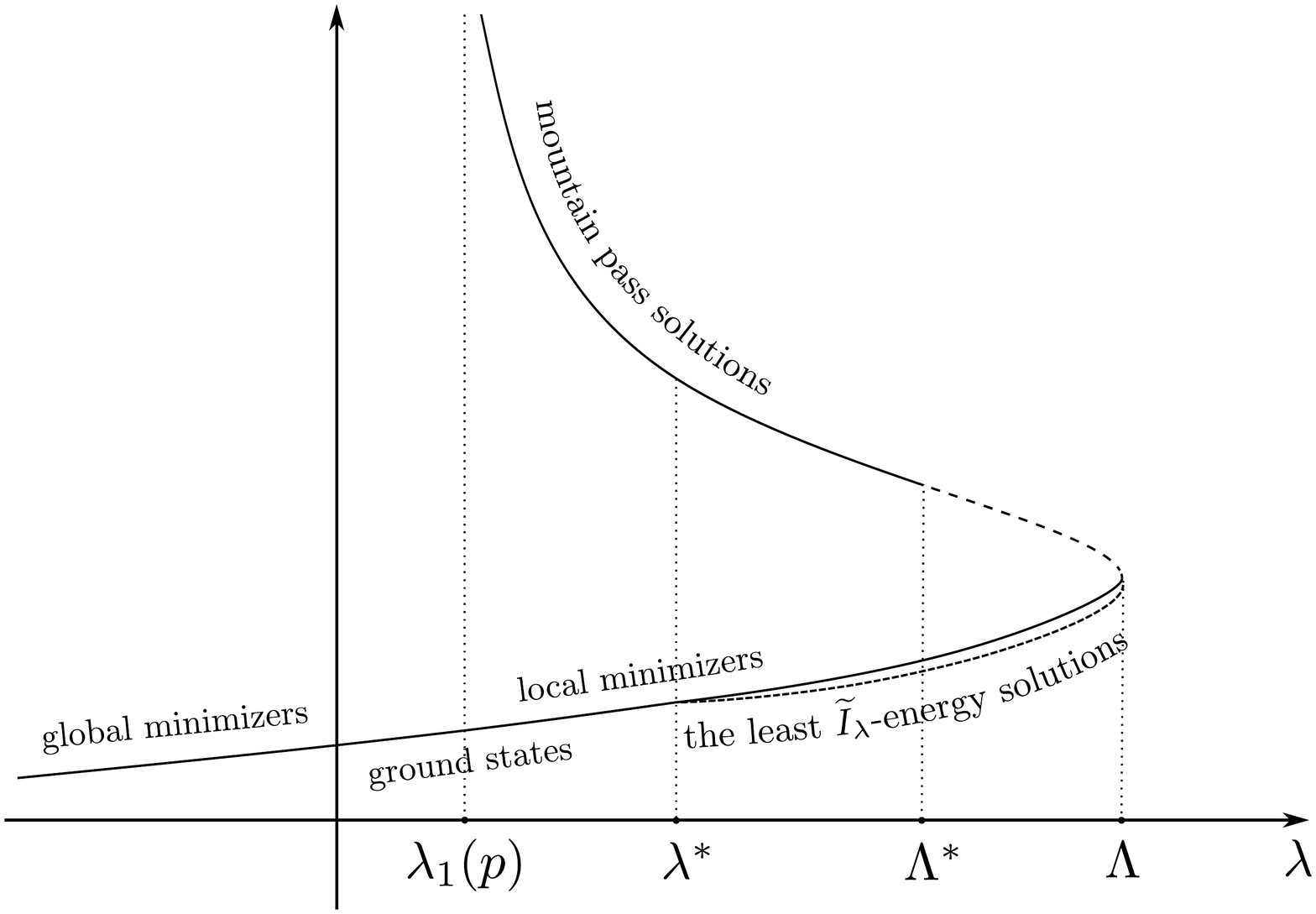}
		}
		\caption{A schematic $L^\infty(\Omega)$-bifurcation diagram provided by Theorems \ref{thm:GS} and \ref{thm:1} under the assumption $\intO a \varphi_p^q \,dx < 0$.}
		\label{fig:bifur}
	\end{figure}
	
	\begin{remark}
		We do not know whether the equality $\Lambda^* = \Lambda$ always takes place, or the inequality $\Lambda^* < \Lambda$ can happen in some regimes. 
		In \cite{Alama1}, in the linear case $p=2$ and under the Neumann boundary conditions, the author obtains the equality under several specific assumptions on the weight $a$, namely, that the set $\text{Int} (\Omega_a^+ \cup \Omega_a^0)$ has a finite number of connected components, each of which is $C^2$-smooth and connected to $\Omega_a^+$, see \cite[(1.5)-(1.7)]{Alama1}. 
		The proofs of \cite{Alama1} rely in a principal way on the strong comparison principle which is known to be a difficult (and, in many essential cases, open) issue in the general nonlinear case $p>1$, see, e.g., \cite{CT}.
		The problem is compounded by the fact that nonnegative solutions of \eqref{eq:P} do not necessarily obey the Hopf maximum principle, which makes it hard to apply approaches known for ``good'' nonlinear cases.
		Thus, the establishment of the equality $\Lambda^* = \Lambda$ under some assumptions for $p>1$ or under significantly weaker assumptions than \cite[(1.5)-(1.7)]{Alama1} for $p=2$, or the construction of examples when $\Lambda^* < \Lambda$, are interesting open problems.
	\end{remark}
	
	Take now any nonnegative $b \in C(\overline{\Omega}) \setminus \{0\}$ and define the weight $a_\mu := a + \mu b$ for $\mu > 0$.
	Consider the boundary value problem analogous to \eqref{eq:P}:
	\begin{equation}\label{eq:Pm}
		\tag{$P_{\lambdaM}^{\mu}$}
		\left\{
		\begin{aligned}
			-\Delta_p u &= \lambda |u|^{p-2}u+ a_\mu(x)|u|^{q-2}u 
			&&\text{in}\ \Omega, \\
			u&=0 &&\text{on}\ \partial \Omega,
		\end{aligned}
		\right.
	\end{equation}
	and denote by $\wIm$ the corresponding truncated energy functional, i.e., \eqref{def:functionals-t} with $a_\mu$ instead of $a$.
	Clearly, if $\intO a \varphi_p^q \,dx=0$, then $\intO a_\mu \varphi_p^q \,dx>0$.
	Moreover, we have $\Omega_{a}^+ \subset \Omega_{a_\mu}^+$.
	
	\begin{theorem}\label{thm:3}
		\marginnote{{\scriptsize Proof on\\ p.\pageref{sec:thm3}}}
		Let $p>2q$, $\partial\Omega$ be connected provided $N \geq 2$, the weight $a$ be such that $\intO a \varphi_p^q \,dx=0$, and $b \in C(\overline{\Omega}) \setminus \{0\}$ be a nonnegative function. 
		Then there exists $\widehat{\mu}>0$  such that for any $\mu \in [0,\widehat{\mu})$ there exist
		$\Lambda^* = \Lambda^*(\mu)$ and $\Lambda = \Lambda(\mu)$ satisfying $\lambda_1(p) <\Lambda^* \leq \Lambda < +\infty$ such that the assertions \ref{thm:1:r1}-\ref{thm:1:r4} of Theorem \ref{thm:1} hold for the problem \eqref{eq:Pm} and the corresponding functional $\wIm$.
		Moreover, for any $\mu \in (0,\widehat{\mu})$ there exists $\epsilon = \epsilon(\mu)>0$ such that for any $\lambda \in (\lambda_1(p)-\epsilon, \lambda_1(p))$ the problem \eqref{eq:Pm} possesses at least three distinct  nonnegative solutions with negative energy - global minimum, local minimum, and mountain pass critical point of $\wIm$. The first two critical points are positive in $\Omega_{a_\mu}^+$, while the third one is positive at least in one connected component of $\Omega_{a_\mu}^+$.
	\end{theorem}
	
	\begin{remark}
		Let us notice that $\inf\{\Lambda^*(\mu): \mu \in [0,\widehat{\mu})\}>\lambda_1(p)$, which follows from the proof of Theorem \ref{thm:3}.
		We refer to Figures \ref{fig:p2q}, \ref{fig:bifur2} for a schematic graphical representation of the results of Theorems \ref{thm:1} and \ref{thm:3}. 
	\end{remark}
	
	\begin{remark}
		In essence, the proofs of Theorems \ref{thm:1} and \ref{thm:3} are based on the observation that the set of nonnegative minimizers of $M(\lambda^*)$ (being nonempty by Theorem \ref{NoGroundStates} \ref{NoGroundStates:1}, \ref{NoGroundStates:2}) has a strict local minimum type geometry which is stable under continuous perturbations of $\wI$.
		In particular, our arguments on the existence of a local minimum point are not restricted solely to $\wIm$, but applicable to \textit{any} continuous perturbation of the functional $\widetilde{I}_{\lambda^*}$. 
	\end{remark}

	\begin{figure}[h!]
		\center{
			\includegraphics[width=0.55\linewidth]{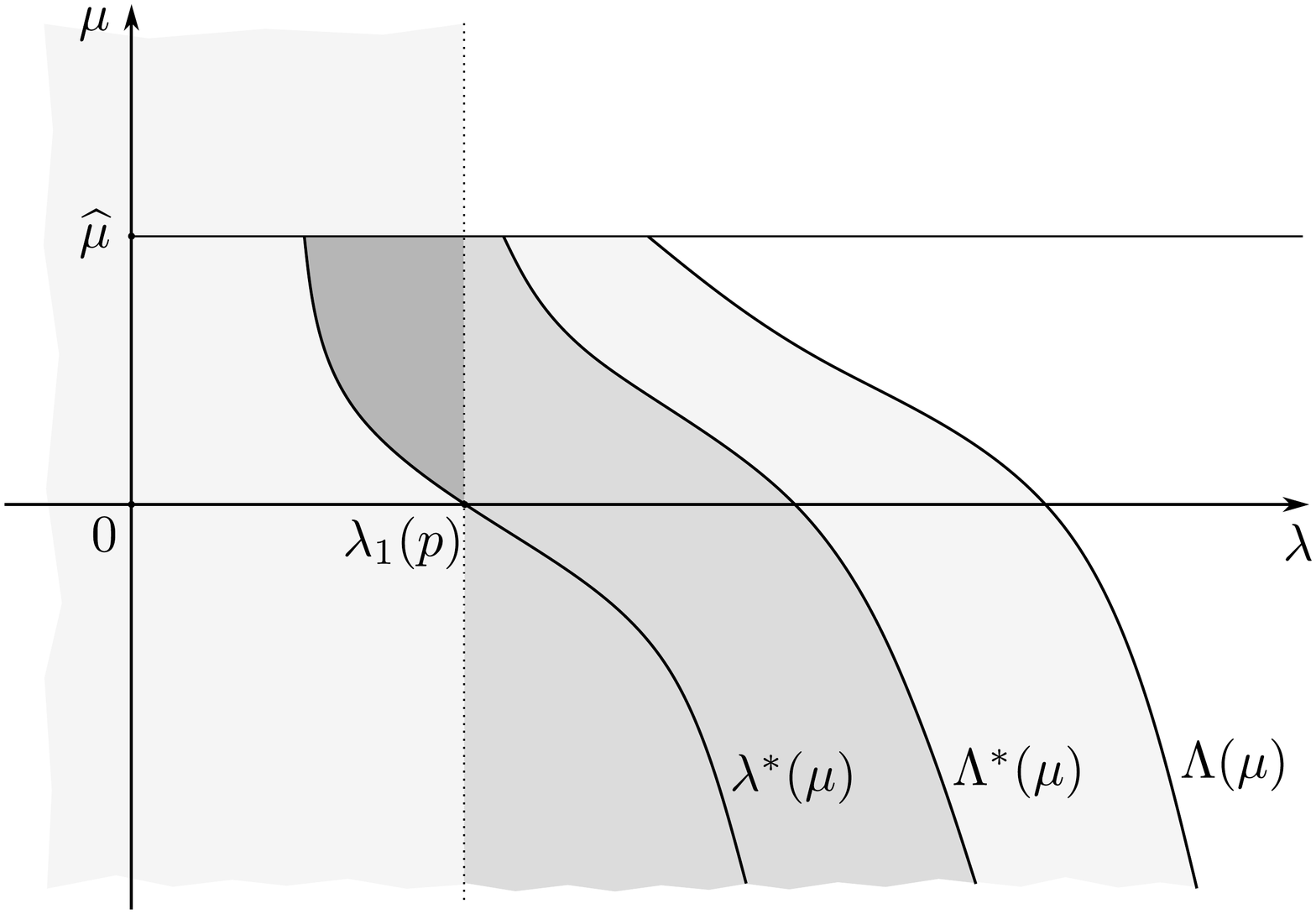}
		}
		\caption{A schematic picture of the regions in the $(\lambda,\mu)$-plane provided by Theorem~\ref{thm:1} (for $\mu \leq 0$) and Theorem \ref{thm:3} (for $\mu \in (0,\widehat{\mu})$).
		Light gray - (at least) one solution. 
		Gray - two solutions.
		Dark gray - three solutions.
	}
		\label{fig:p2q}
	\end{figure}

	\begin{figure}[h!]
		\center{
			\includegraphics[width=0.55\linewidth]{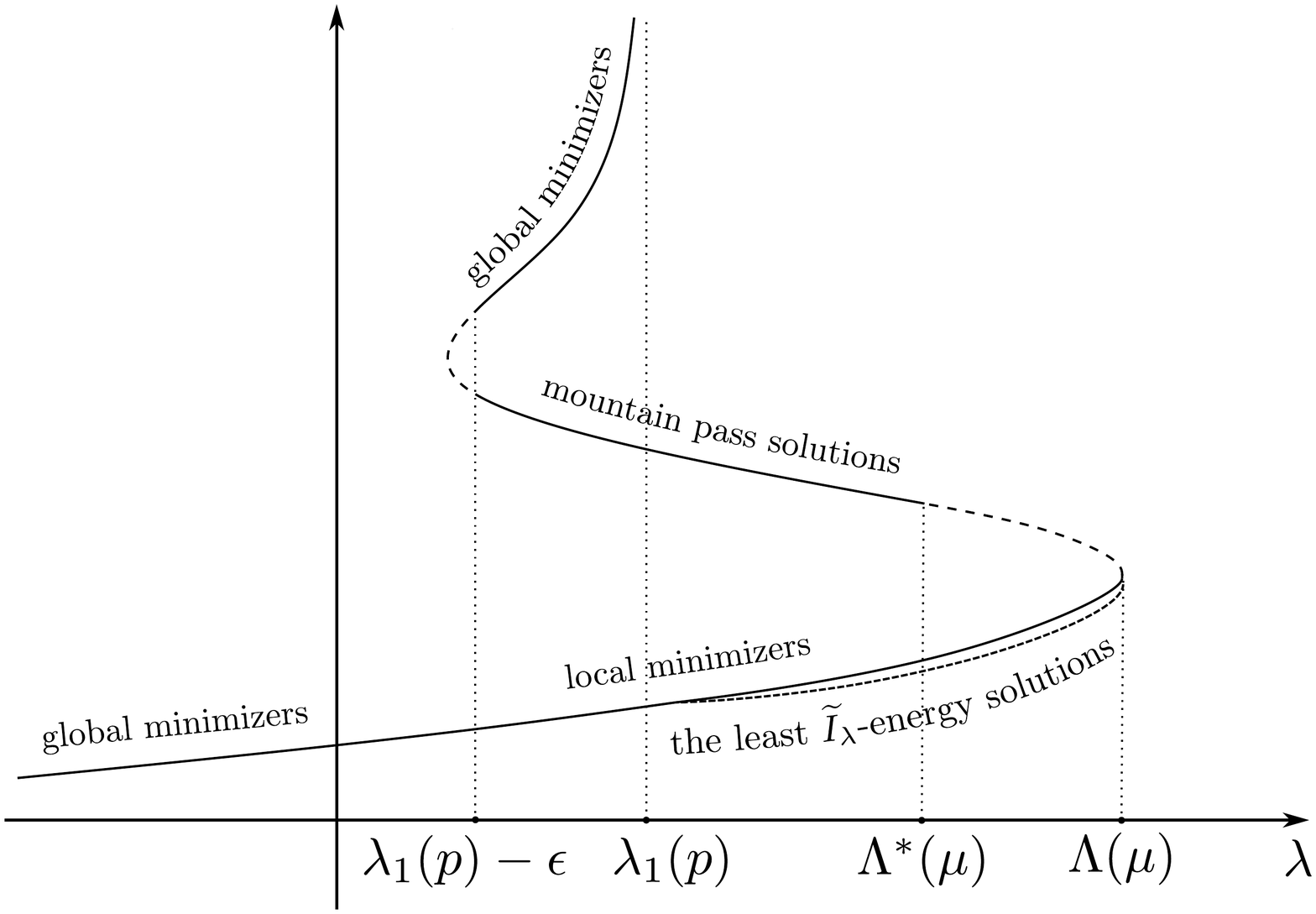}
		}
		\caption{A schematic $L^\infty(\Omega)$-bifurcation diagram provided by Theorem \ref{thm:3} for $\mu \in (0,\widehat{\mu})$.}
		\label{fig:bifur2}
	\end{figure}

	Recently, the literature has been enriched with a few results on the local continuation of the branch of nonnegative solutions to various problems with respect to the parameter beyond a critical value of the type $\lambda^*$ characterizing the limit of applicability of the Nehari manifold method, see, e.g., \cite{BT2021,IS,QS,SM}.
	To the best of our knowledge, the first contribution in this direction was made in \cite{IS} for the problem \eqref{eq:P} in the superhomogeneous regime $q>p$. 
	A similar approach was suggested in \cite{QS} with application to \eqref{eq:P} in the subhomogeneous case $q<p$.
	Unlike \cite{IS,QS}, our arguments do not depend on a particular structure of the Nehari manifold of the perturbed problem, which makes them more universal.

	\medskip
	Our last main result provides information on the nonexistence, in contrast to Theorem~\ref{thm:1}.
	\begin{theorem}\label{NoSol}
	\marginnote{{\scriptsize Proof on\\ p.\pageref{sec:nonexistence}}}
		Assume, in addition to $1<q<p$, that 
		\begin{equation}\label{eq:Picone-0} 
			(q-1) s^p + q s^{p-1} - (p-q) s + (q-p+1) \geq 0  
			~~\text{for all}~~  s \geq 0.
		\end{equation}
		Let, moreover, one of the following assumptions be satisfied:
		\begin{enumerate}[label={\rm(\roman*)}]
			\item\label{NoSol:1} $\intO a\varphi_p^q\,dx= 0$ and $\lambda>\lambda_1(p) \,(=\lambda^*)$.
			\item\label{NoSol:2} $\intO a\varphi_p^q\,dx>0$ and $\lambda\ge \lambda_1(p) \,(=\lambda^*)$.
		\end{enumerate}	
		Then there exists no nonzero nonnegative solution $u$ 
		of \eqref{eq:P} such that $u>0$ in $\Omega_a^+$.
	\end{theorem}

	The proof of Theorem \ref{NoSol} is based on the application of a generalized Picone inequality established by the present authors in \cite[Theorem 1.8]{BT_Picone}.
	A description and some properties of the set of exponents $p$ and $q$ for which \eqref{eq:Picone-0} holds are discussed in \cite[Lemma 1.6 and Remark 1.7]{BT_Picone}, and several sufficient assumptions can be also found therein.
	Let us explicitly emphasize two properties.
	First, if $p > 2q \,(>2)$, then \eqref{eq:Picone-0} is not satisfied, and hence there is no contradiction with the existence results provided by Theorem \ref{thm:1} \ref{thm:1:2}, see Figure \ref{fig:1}.
	Second, if $p=2$, then \eqref{eq:Picone-0} holds for all $q \in (1,2)$, and hence Theorem \ref{thm:1} \ref{thm:1:2} cannot be extended to the case $q<p=2$. 
	Finally, we recall that if $\intO a \varphi_p^q \,dx \geq 0$ and $q>p$, then \eqref{eq:P} has no positive solution for $\lambda > \lambda^*$, see \cite{ilyasov2}.
	That is, the results of Theorem \ref{thm:1} in the case $\intO a \varphi_p^q \,dx = 0$ are of a purely nonlinear and subhomogeneous nature. 
	
	\begin{remark}
		Theorem \ref{NoSol} implies that if, under the imposed assumptions, there exists a nonzero nonnegative solution $u$ of \eqref{eq:P}, then it must exhibit a dead core in $\Omega_a^+$, that is, $u \equiv 0$ in a connected component of $\Omega_a^+$.
	\end{remark}
	
	\begin{remark}
		We notice that the assumption \eqref{eq:Picone-0} can be relaxed to $p\le q+1$ provided any nonzero nonnegative solution $u$ of \eqref{eq:P} with $u>0$ in $\Omega_a^+$ satisfies $\nabla u \nabla \varphi_p \geq 0$ in $\Omega$. 
		The proof of Theorem \ref{NoSol} under these assumptions follows along the same lines, by applying \cite[Theorem 1.8 (ii)]{BT_Picone}.
	\end{remark}
	
	\begin{figure}[h!]
		\center{
		\includegraphics[width=0.5\linewidth]{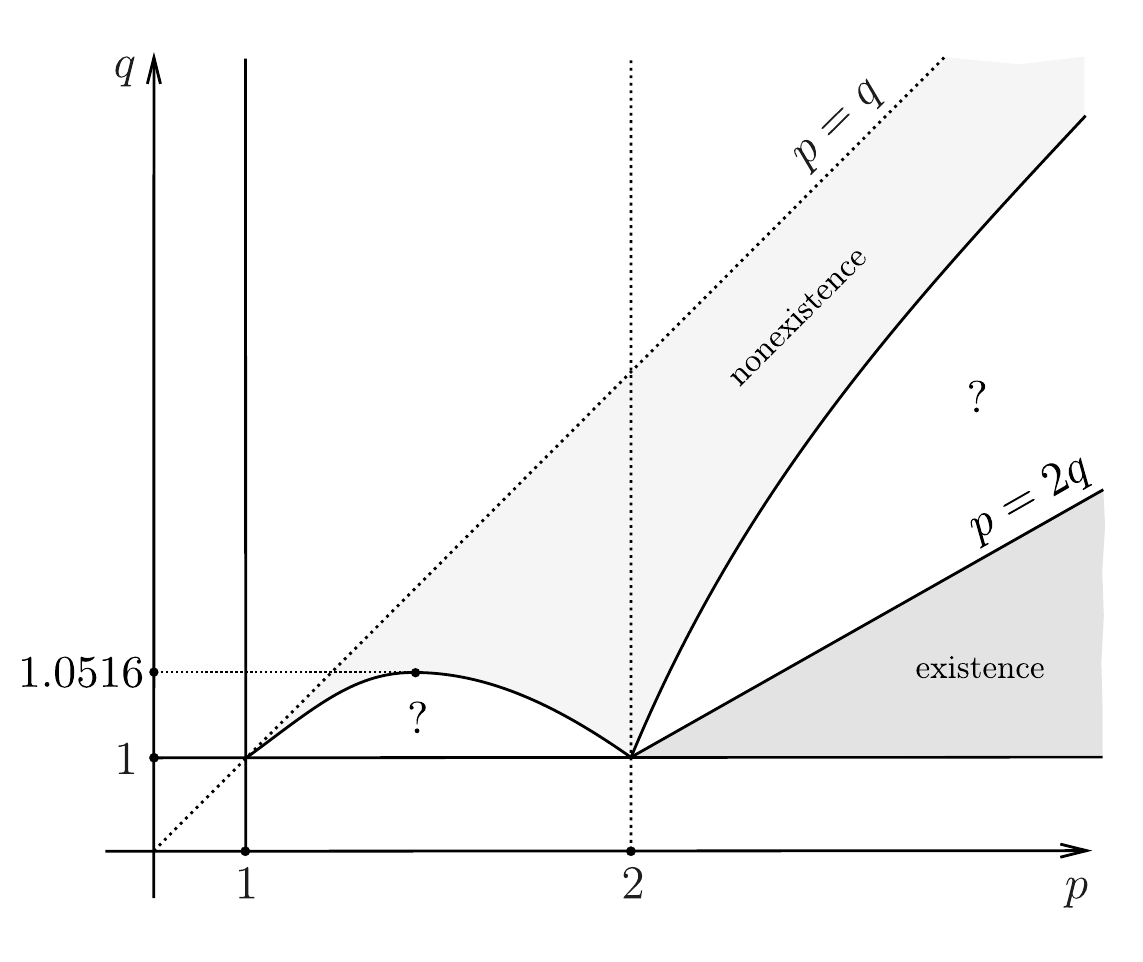}
		}
		\caption{A schematic picture of regions in the $(p,q)$-plane for the validity of Theorem~\ref{thm:1}~\ref{thm:1:2} (gray set) and Theorem~\ref{NoSol}~\ref{NoSol:1} (light gray set)
		in the case $\intO a \varphi_p^q \,dx =0$ and $\lambda>\lambda_1(p)$.
		}
		\label{fig:1}
	\end{figure}

	The rest of the work is organized as follows.
	In Section \ref{sec:auxiliary}, we provide various auxiliary results on the properties of the energy functionals $\I$ and $\wI$.
	Section \ref{sec:proof:NoGroundS} is devoted to the proof of Theorem \ref{NoGroundStates} and to the properties of the set of nonnegative minimizers of $M(\lambda^*)$.
	In Section \ref{sec:qualitative}, we prove Propositions \ref{prop:properties-m}, \ref{prop:PCV}, as well as several other properties of $M$ and $M^-$ and their minimizers.
	In Sections  \ref{sec:thm1} and \ref{sec:thm3}, we present the proofs of Theorems \ref{thm:1} and \ref{thm:3}, respectively.
	Finally, in Section \ref{sec:nonexistence}, we prove of Theorem \ref{NoSol}.

	\section{Auxiliary results}\label{sec:auxiliary}
	
	In this section, we collect several auxiliary results, especially on properties of the energy functionals $\I$ and $\wI$ defined by \eqref{def:functionals} and \eqref{def:functionals-t}, respectively, and on their critical points. 
	Most of the results will be used in the proofs of our main theorems, while several facts also have an independent interest.
	
	\subsection{Special values of parameter} 
	In order to provide finer analysis, we introduce the following critical values of $\lambda$ in addition to $\lambda^*$, and study their relations in brief:
	\begin{align} 
		\label{eq:l+}
		\lambda_\pm^* &:=\inf\left\{\dfrac{\|\nabla u\|_p^p}{\|u\|_p^p}:~
		u\in\W,\ \pm \intO a|u|^q\,dx >0\,\right\},
		\\ 
		\notag
		\lambda_0^* &:=\inf\left\{\dfrac{\|\nabla u\|_p^p}{\|u\|_p^p}:~
		u\in\W\setminus\{0\},\ \intO a|u|^q\,dx=0\,\right\}.
	\end{align}
	
	\begin{proposition}\label{prop:+} 
		The following assertions hold:
		\begin{enumerate}[label={\rm(\roman*)}]
			\item\label{prop:+:1} Let $\intO a\varphi_p^q\,dx>0$. Then 
			\begin{equation}\label{eq:lll}
				\lambda_1(p)=\lambda^*=\lambda_+^*
				<\lambda_0^*=\lambda_-^*.
			\end{equation}
			Moreover, 
			$\lambda^*$ and $\lambda_+^*$ are attained only by $t\varphi_p$ 
			($t \neq 0$), 
			$\lambda_0^*$ is attained, and $\lambda_-^*$ is not attained.
			\item\label{prop:+:2} Let $\intO a\varphi_p^q\,dx=0$. Then 
			$$
			\lambda_1(p)=\lambda^*=\lambda_0^*
			=\lambda_\pm^*.
			$$
			Moreover, 
			$\lambda^*$ and $\lambda_0^*$ are attained only by $t\varphi_p$ 
			($t \neq 0$), 
			while $\lambda_\pm^*$ is not attained.
			
			\item\label{prop:+:3} Let $\intO a\varphi_p^q\,dx<0$. Then 
			$$
			\lambda_1(p)=\lambda_-^*
			<\lambda_0^*=\lambda_+^*=\lambda^*.
			$$
			Moreover, $\lambda_-^*$ 
			is attained only by $t\varphi_p$ 
			($t \neq 0$), 
			$\lambda_0^*$ and $\lambda^*$ are attained, and $\lambda_+^*$ is not attained. 
		\end{enumerate}
	\end{proposition} 
	\begin{proof} 
		We will prove only the assertion \ref{prop:+:1}. The assertions \ref{prop:+:2} and \ref{prop:+:3} can be proved in much the same way.
		Clearly, the first two equalities in \eqref{eq:lll} and the attainability of $\lambda^*$, $\lambda_+^*$ follow directly from the simplicity of $\lambda_1(p)$. 
		The simplicity also yields $\lambda_1(p)<\lambda_0^*, \, \lambda_-^*$.
		The attainability of $\lambda_0^*$ is  evident.
		
		Let us show that $\lambda_-^*$ is not attained and $\lambda_0^*=\lambda_-^*$.
		If we suppose that $\lambda_-^*$ is attained by some $v$, then it is attained also by $|v|$.
		Since $\intO a|v|^q\,dx<0$, we see that $|v|$ is a local minimum point of the Rayleigh quotient, which means that $|v|$ is a nonnegative eigenfunction of the $p$-Laplacian
		corresponding to $\lambda_-^*$. 
		However, the only sign-constant eigenfunction is the first one (see, e.g., \cite{anane1987}).
		Hence, $\lambda_-^*$ is not attained, and the corresponding minimization sequence converges to a function $v$ satisfying $\intO a|v|^q\,dx=0$.
		In particular, we deduce that $\lambda_0^* \leq \lambda_-^*$.	
		
		To complete the proof, let us obtain the converse inequality $\lambda_-^* \leq \lambda_0^*$. 
		Let $u_0$ be a minimizer of $\lambda_0^*$, and hence we have $u_0 \not\equiv 0$ a.e.\ in $\Omega$ and $\intO a |u_0|^q \,dx=0$.
		Suppose first that $u_0$ is not a critical point of the functional $u \mapsto \intO a|u|^q\,dx$. 
		That is, we can find $v\in\W$ such that 
		$\intO a|u_0|^{q-2}u_0v\,dx<0$. 
		So, for sufficiently small $\varepsilon>0$ we get
		$$
		\intO a|u_0+\varepsilon v|^{q}\,dx
		=
		q
		\intO \int_0^\varepsilon \, a|u_0+t v|^{q-2}(u_0+t v)v\,dtdx <0, 
		$$
		whence $u_0+\varepsilon v$ is an admissible function for $\lambda_-^*$.
		Letting $\varepsilon\to 0$, we obtain $\lambda_-^*\le \lambda_0^*$. 
		
		Suppose now that $\intO a|u_0|^{q-2}u_0v\,dx = 0$ for any $v \in \W$. 
		The fundamental lemma of the calculus of variations yields $a|u_0|^{q-2}u_0 \equiv 0$ a.e.\ in $\Omega$, which implies, in its turn, that $u_0 \equiv 0$ a.e.\ in $\Omega_a^\pm$.
		Taking any $v \in C_0^\infty(\Omega) \setminus \{0\}$ with the support in $\Omega_a^-$ and recalling that $\intO a |u_0|^q \,dx=0$, we obtain
		$$
		\intO a|u_0+\varepsilon v|^{q}\,dx=
		\intO a|u_0|^{q}\,dx
		+
		\varepsilon^q
		\intO a|v|^{q}\,dx
		=
		\varepsilon^q
		\int_{\Omega_a^-} a|v|^{q}\,dx 
		< 0
		$$
		for any $\varepsilon>0$. Consequently, $u_0+\varepsilon v$ is admissible for $\lambda_-^*$, and letting $\varepsilon \to 0$ as above, we obtain the desired inequality $\lambda_-^*\le \lambda_0^*$.
	\end{proof} 
	
	\begin{remark}
		In view of the even nature of the Rayleigh quotient, the critical values $\lambda^*$, $\lambda_\pm^*$, $\lambda_0^*$ can be equivalently characterized via the truncated integrals $\intO u_+^p\,dx$ and $\intO a u_+^q\,dx$.
	\end{remark}
	
	\begin{remark}
		Several sufficient assumptions guaranteeing that the minimizer of $\lambda_0^*$ generates a critical point of $\I$ are presented in \cite[Section 3.1]{QS}.
	\end{remark}

	
	\subsection{Fibered functionals}\label{sec:fiber}
	
	The following results on the fibered functionals associated to $\I$ and $\wI$ are standard and can be found, e.g., in \cite{brown} in the linear case $p=2$, or in \cite{drabek-poh,IS} in the superhomogeneous case $q>p$.
	
	Take any $u \in \W$ such that $E_\lambda(u) \cdot \int_\Omega a |u|^q \,dx > 0$. 
	Then there exists a unique critical point  $t_\lambda(u) > 0$  of the function $t \mapsto \I(tu)$ for $t>0$. 
	We easily see that 
	\begin{equation}\label{tu}
		t_\lambda(u)= \left(\frac{\int_\Omega a |u|^q\,dx}{E_\lambda(u)}\right)^{\frac{1}{p-q}} = 
		\frac{\left|\int_\Omega a |u|^q \,dx\right|^{\frac{1}{p-q}}}
		{\left|E_\lambda(u)\right|^{\frac{1}{p-q}}}
	\end{equation}
	and
	\begin{equation}\label{eq:INeh}
		\I(t_\lambda(u)u) 
		= 
		-\frac{p-q}{pq}\,E_\lambda(t_\lambda(u)u)
		=
		-t_\lambda^q(u) \,\frac{p-q}{pq}\, \intO a|u|^q\,dx.
	\end{equation}
	Consequently, if $E_\lambda(u)>0$ and $\int_\Omega a |u|^q \,dx > 0$, then $t_\lambda(u)$ is the global minimum point of the function $t \mapsto \I(tu)$ for $t>0$, and $\I(t_\lambda(u)u)<0$.
	Analogously, if $E_\lambda(u)<0$ and $\int_\Omega a |u|^q \,dx < 0$, then $t_\lambda(u)$ is the global maximum point of the function $t \mapsto \I(tu)$ for $t>0$, and $\I(t_\lambda(u)u)>0$.
	
	Noting that $E_\lambda(u)$ and $\int_\Omega a |u|^q \,dx$ are $p$- and $q$-homogeneous, respectively, we deduce that
	\begin{equation}\label{eq:Poh}
		\J(u) :=  
		\I(t_\lambda(u) u) 
		= 
		-\text{sign}(E_\lambda(u))\, \frac{p-q}{p q}\, 
		\frac{\left|\int_\Omega a |u|^q \,dx\right|^{\frac{p}{p-q}}}{\left|E_\lambda(u)\right|^{\frac{q}{p-q}}}. 
	\end{equation}
	The functional $\J$ is $0$-homogeneous and it is called \textit{fibered} functional. 
	This functional will serve as a convenient tool in the study of attainability of $M(\lambda)$.
	
	Expressions analogous to \eqref{tu} and \eqref{eq:Poh} hold true for the truncated functional $\wI$ defined by \eqref{def:functionals-t}.
	In particular, we denote by $\widetilde{J}_\lambda$ the truncated fibered functional.

	\subsection{Nehari manifold}
	Let us provide an additional information on the Nehari manifold $\N$ defined by \eqref{eq:nehari}.
	If $E_\lambda(u) \cdot \int_\Omega a |u|^q \,dx > 0$ for some $u \in \W$, then $t_\lambda(u)u \in \N$, where $t_\lambda(u)>0$ is given by \eqref{tu}.
	We have
	\begin{equation}\label{eq:EonN1-0} 
		\I(u)
		=
		-\frac{p-q}{pq}\,E_\lambda(u)= 
		-\frac{p-q}{pq}\,\intO a|u|^q\,dx \quad 
		{\rm for~ any}\ u\in\N,
	\end{equation}
	cf.\ \eqref{eq:INeh}.
	Recall the definitions \eqref{eq:A+} and \eqref{eq:A-} of the sets $\mathcal{A}^+$ and $\mathcal{A}^-$, respectively:
	$$
	\A^\pm
	:=\left\{u \in \W:~ 
	\pm \intO a |u|^q\,dx>0\,\right\}.
	$$
	\begin{lemma}\label{lem:NehariEmpty}
		The following assertions hold:
		\begin{enumerate}[label={\rm(\roman*)}]
			\item\label{lem:NehariEmpty:1}
			$\N \cap \mathcal{A}^+\not=\emptyset$ for any $\lambda\in\mathbb{R}$. 
			\item\label{lem:NehariEmpty:2}
			$\N \cap \mathcal{A}^-\not=\emptyset$ if and only if $\lambda>\lambda_-^*$. 
		\end{enumerate}
	\end{lemma}
	\begin{proof} 
		\ref{lem:NehariEmpty:1}
		Since $a \in C(\overline{\Omega})$ and $\Omega_a^+ \neq \emptyset$, for any $x_0 \in \Omega_a^+$ there exists a ball $B(x_0)$ centered at $x_0$ such that $a>0$ in $B(x_0)$. 
		Choose any nonnegative $v\in C_0^\infty(\Omega) \setminus \{0\}$ with the support in $B(x_0)$ and set $v_n(\cdot)=n^{N/p-1} v(n(\cdot-x_0))$. 
		Then $E_\lambda(v_n)>0$ for any sufficiently large $n$, and $\intO a v_n^q\,dx>0$.
		As a result, $t_\lambda(v_n)v_n\in \mathcal{N}_\lambda\cap \mathcal{A}^+$.
		
		\ref{lem:NehariEmpty:2}
		It follows from the definition of $\lambda_-^*$ that if $\lambda \leq \lambda_-^*$, then $E_\lambda(u) \geq 0$ for any $u$ satisfying $\intO a |u|^q \,dx < 0$, and hence $\N \cap \mathcal{A}^- = \emptyset$. 
		On the other hand, if $\lambda>\lambda_-^*$, then there exists $u$ such that $\intO a |u|^q \,dx < 0$ and $\lambda > \|\nabla u\|_p^p/\|u\|_p^p \geq \lambda_-^*$. 
		Consequently, we have $E_\lambda(u)<0$, which yields $t_\lambda(u)u \in \N \cap \mathcal{A}^-$.
	\end{proof} 
	
	\begin{remark}\label{rem:m<0}
		In view of \eqref{eq:Poh}, \eqref{eq:EonN1-0}, and the fact that $\N\cap\mathcal{A}^+\neq\emptyset$ for any $\lambda$, 
		we see that 
		\begin{equation*} 
			M(\lambda)=\inf_{u\in \N} \I(u)=
			\inf_{u\in \N\cap\mathcal{A}^+} \I(u)
			=\inf_{u\in \mathcal{N}_{\lambda}\cap\mathcal{A}^+}\J(u)<0. 
		\end{equation*}
	\end{remark}

	The following result can be proved by standard arguments based on the Lagrange multipliers rule (see, e.g. \cite[Theorem 48.B and Corollary 48.10]{zeidler}).
	\begin{lemma}\label{lemm:crit}
		Let $u \in \N$ be a critical point of $\I$ over $\N$.
		Assume that $E_\lambda(u)\not=0$ (or, equivalently, $\intO a|u|^q\,dx\not=0$). 
		Then $u$ is a critical point of $ \I$ (over $\W$). 
	\end{lemma}

	\begin{remark}
		Analogs of Lemmas \ref{lem:NehariEmpty} and \ref{lemm:crit} are also valid in the truncated case.
	\end{remark}

	\subsection{Properties of \texorpdfstring{$\I$}{I} and \texorpdfstring{$\wI$}{I-tilde}}\label{sec:relation}
	
	In this subsection, we provide several results on the relation between the functionals $\I$ and $\wI$ and on properties of their critical points.
	
	The following result, in combination with Proposition \ref{prop:charact}, asserts that ground states and least $\wI$-energy solutions of \eqref{eq:P} correspond to the same critical level whenever either $M(\lambda)$ or $\widetilde{M}(\lambda)$ is attained.	
	\begin{proposition}\label{prop:wm=m}
		Let $\lambda \in \mathbb{R}$.
		Then ${M}(\lambda)$ is attained if and only if $\widetilde{M}(\lambda)$ is attained.
		Moreover, ${M}(\lambda) = \widetilde{M}(\lambda) < 0$ for any $\lambda \in \mathbb{R}$. 
	\end{proposition}
	\begin{proof}
		Observe that $\widetilde{M}(\lambda) \leq M(\lambda)$ for any $\lambda \in \mathbb{R}$. 
		Indeed, if $u \in \N$, then $|u| \in \N$ and hence $|u| \in \wN$.
		Since $\wI(|u|) = \I(|u|) =  \I(u)$, the inequality $\widetilde{M}(\lambda) \leq M(\lambda)$ follows.
		As a consequence, we have $\widetilde{M}(\lambda) = M(\lambda) = -\infty$ for any $\lambda>\lambda^*$, see Theorem \ref{thm:GS} \ref{thm:GS:2}.
		Moreover, $\widetilde{M}(\lambda), M(\lambda) < 0$ for any $\lambda \in \mathbb{R}$, see Remark \ref{rem:m<0}.
		
		Suppose, by contradiction, that $\widetilde{M}(\lambda) < M(\lambda)$ for some $\lambda \leq \lambda^*$.
		That is, there exists $u \in \wN$ such that $ \wI(u) < M(\lambda) <0$.
		Clearly, we have $u_+ \not\equiv 0$ a.e.\ in $\Omega$.
		We also conclude that $u_- \not\equiv 0$ a.e.\ in $\Omega$, since in the case $u \geq 0$ a.e.\ in $\Omega$ we would get $u \in \N$ and  $\wI(u) =  \I(u) \geq M(\lambda)$.
		The assumptions on $u$ yield
		\begin{equation}\label{eq:int>0}
		\wE(u) =  \int_\Omega |\nabla u_-|^p \,dx + E_\lambda(u_+)  = \intO a u_+^q \, dx > 0.
		\end{equation}
		If $E_\lambda(u_+) > 0$, then $t_\lambda(u_+) u_+ \in \N$ and
		$$
		\I(u_+) 
		\geq
		\min_{t > 0} \I(t u_+)
		=
		\I(t_\lambda(u_+) u_+) \geq M(\lambda),
		$$
		which is impossible since
		$$
		M(\lambda) >  \wI(u) 
		= 
		\frac{1}{p} \int_\Omega |\nabla u_-|^p \,dx 
		+ 
		\I(u_+) 
		> 
		\I(u_+).
		$$	
		Therefore, we have $E_\lambda(u_+) \leq 0$ and hence, recalling that $u_+ \not\equiv 0$ a.e.\ in $\Omega$, we obtain
		\begin{equation}\label{eq:l*+min}
		\frac{\int_\Omega |\nabla u_+|^p \,dx}{\int_\Omega u_+^p \,dx} 
		\leq 
		\lambda 
		\leq 
		\lambda^* 
		=
		\lambda^*_+
		\leq \frac{\int_\Omega |\nabla u_+|^p \,dx}{\int_\Omega u_+^p \,dx}.
		\end{equation}
		Here, the equality is given by Proposition \ref{prop:+}, and the last inequality follows from the definition \eqref{eq:l+} of $\lambda^*_+$ and the fact that $\intO a u_+^q \, dx > 0$, see \eqref{eq:int>0}.
		We conclude from \eqref{eq:l*+min} that  $u_+$ is a minimizer of $\lambda^*_+$.
		However, according to Proposition \ref{prop:+}, $\lambda^*_+$ is either not attained, or attained by $t\varphi_p$ with some $t \neq 0$, which contradicts the fact that $u_- \not\equiv 0$ a.e.\ in $\Omega$.
	\end{proof}
	
	\begin{remark}
		In general, $M(\lambda)$ might possess sign-changing minimizers (since the strong maximum principle is not applicable to nonnegative solutions of \eqref{eq:P}), while the set of minimizers of $\widetilde{M}(\lambda)$ consists solely of nonnegative functions.
		That is, the set of minimizers of $\widetilde{M}(\lambda)$ is contained in that of $M(\lambda)$.
	\end{remark}
	
	\begin{remark}\label{rem:positivity} 
		Let $u$ be a {\it nonnegative} solution of \eqref{eq:P}. 
		If $u$ has a zero point $x_0$ in $\Omega_a^+$, then, 
		according to the strong maximum principle, 
		$u\equiv 0$ in a connected component of $\Omega_a^+$ which contains $x_0$. 
		In other words, if $u\not\equiv 0$ in some connected component of $\Omega_a^+$, 
		then $u>0$ in that connected component. 
		An additional assumption on $u$ guaranteeing that $u> 0$ in the whole open set $\Omega_a^+$ is presented in Lemma \ref{lem:positivity}. 
	\end{remark}

	The proofs of all remaining results of this subsection will be given only for the untruncated functional $\I$.
	The case of $\wI$ can be treated analogously.
	\begin{lemma}\label{lem:positivity}
		Let $u$ be a local minimum point of $ \I$ or of $ \wI$ which is nonnegative in $\Omega_a^+$.
		Then $u>0$ in $\Omega_a^+$.
	\end{lemma}
	\begin{proof}
		Suppose, contrary to our claim, that $u(x_0)=0$ for some $x_0 \in \Omega_a^+$. 
		By the strong maximum principle, we have $u \equiv 0$ in a connected component $A$ of $\Omega_a^+$ containing $x_0$.
		Consider any nonnegative $\varphi \in C_0^\infty(\Omega) \setminus \{0\}$ with the support in $A$. 
		Since $q<p$ and $\intO a \varphi^q \,dx >0$, we have $ \I(t\varphi)<0$ for any sufficiently small $t>0$.
		This yields
		$$
		\I(u) 
		\leq 
		\I(u+t\varphi)
		=
		\I(u) +  \I(t\varphi)
		< \I(u),
		$$
		which contradicts our assumption that $u$ is a local minimum point of $\I$.
	\end{proof}
	
	\begin{remark}
		One could wonder whether it is possible to rid of the nonnegativity assumption in Lemma \ref{lem:positivity}.
		In general, this assumption is vital. In \cite[Section 5]{BDSWPT}, it is shown that there are domains for which a least energy \textit{nodal} solution to the zero Dirichlet problem for the equation $-\Delta u = |u|^{q-2}u$, i.e., \eqref{eq:P} with $a \equiv 1$ and $\lambda=0$, is a point of local minimum.
	\end{remark}
	
	\begin{remark}
		It would be interesting to know whether the result of Lemma \ref{lem:positivity} remains valid for ground states or least $\wI$-energy solutions of \eqref{eq:P}, cf.\ Lemma \ref{lem:negative} below.
	\end{remark}

	Another general results are given in the following two lemmas.
	\begin{lemma}\label{lem:local-negative}
		Let $u$ be a local minimum point of $\I$ (resp.~$ \wI$).
		Then $ \I(u)<0$ (resp.~$ \wI(u)<0$).
	\end{lemma}
	\begin{proof}
		Observe that $u$ is a solution of \eqref{eq:P}, and hence 
		$u \in \N$ and $t_\lambda(u)=1$.
		Suppose first that $ \I(u)>0$. 
		By \eqref{eq:INeh}, $t_\lambda(u)$ is the global maximum point of the function $t \mapsto \I(tu)$, $t>0$, and it is clear that $\I(u)>\I(tu)$ for any $t \neq 1$.
		This contradicts our assumption that 
		$u$ is a local minimizer of $\I$. 
		
		Suppose now that $ \I(u)=0$.
		We get from \eqref{eq:EonN1-0} that $E_\lambda(u)=0=\intO a |u|^q \,dx$. 
		Clearly, $u \not\equiv 0$ in $\Omega$ since $0$ is not a local minimum point of $\I$ in view of the assumption $q<p$ and the fact that
		$\mathcal{A}^+\not=\emptyset$ (see also Lemma \ref{lem:positivity}).
		Moreover, $u$ is not a critical point of the functional $w \mapsto \intO a |w|^q \,dx$. 
		Indeed, if $u$ is such a critical point, then $\intO a |u|^{q-2} u v \,dx =0$ for all $v \in \W$. 
		This yields $a |u|^{q-2} u \equiv 0$ in $\Omega$ and hence $u \equiv 0$ in $\Omega_a^+$.
		However, Lemma \ref{lem:positivity} implies that $u>0$ in $\Omega_a^+$, which is a contradiction. 	
		Thus, we can find $v \in \W$ such that $\intO a |u|^{q-2} u v \,dx < 0$. 
		Since $u$ is a solution of \eqref{eq:P}, we have  $\frac{1}{p}\langle E_\lambda'(u),v\rangle = \intO a |u|^{q-2} u v \,dx < 0$.
		Therefore, we see that $tu$ with $t>1$ is not a critical point of $ \I$:
		\begin{align*}
			\langle  \I'(tu),v\rangle 
			&=
			t^{q-1} \left(\frac{t^{p-q}}{p} \langle E_\lambda'(u),v\rangle - \intO a |u|^{q-2} u v \, dx \right)
			\\
			&=
			t^{q-1} \left(t^{p-q} - 1\right) \intO a |u|^{q-2} u v \,dx
			< 0.
		\end{align*}
		At the same time, we have $ \I(tu)=0$ for any $t>0$ since $E_\lambda(u)=0=\intO a |u|^q \,dx$. 
		Thus, by the mean value theorem, for any $t>1$ and for any sufficiently small $s>0$ there exists $s_0 \in (0,s)$ such that
		$$
		\I(tu + sv) =  \I(tu) + s \langle  \I'(tu+s_0v),v\rangle 
		=
		s \langle  \I'(tu+s_0v),v\rangle
		< 0.
		$$
		Consequently, taking $(t-1)$ and $s$ small enough, we conclude that $tu + sv$ belongs to a small neighborhood of the local minimum point $u$, but $ \I(tu + sv) < 0=\I(u)$. 
		A contradiction.
	\end{proof}
	
	\begin{remark}
		Notice that in Lemmas \ref{lem:positivity} and \ref{lem:local-negative} the local minimum point $u$ of $\I$ is \textit{not} required to be of constant sign in $\Omega$.
	\end{remark}

	\begin{lemma}\label{lem:negative}
		Let $u$ be a nonzero critical point of $ \I$ (resp.\ of $\wI$) such that $u \equiv 0$ in $\Omega_a^+$. 
		Let at least one of the following assumptions be satisfied:
		\begin{enumerate}[label={\rm(\roman*)}]	
			\item\label{lem:negative:1} $\Omega_a^0$ has empty interior.
			\item\label{lem:negative:2} $u$ has a constant sign in $\Omega$. 
			\item\label{lem:negative:3} $p=2$.
			\item\label{lem:negative:4} $N=1$.
		\end{enumerate}
		Then $ \I(u) > 0$ (resp.\ $\wI(u)>0$). 
	\end{lemma}
	\begin{proof}	
		Since $u \equiv 0$ in $\Omega_a^+$, we have $\intO a |u|^q \,dx \leq 0$.
		Thus, $ \I(u) \geq 0$ by \eqref{eq:EonN1-0}.
		Suppose, by contradiction, that $ \I(u)=0$, and hence $\intO a |u|^q \,dx = 0$.
		This yields $\text{supp}\, u \subset\Omega_a^0$. 
		Recalling that  $u \in C^{1,\beta}_0(\overline{\Omega})$ for some $\beta\in(0,1)$ by Remark \ref{rem:reg}, we conclude that $u\equiv 0$ in $\Omega$ under the assumption \ref{lem:negative:1}, which is impossible.
		Since $\text{supp}\, u \subset\Omega_a^0$, we see that $u$  weakly solves the equation $-\Delta_p u = \lambda |u|^{p-2} u$ in $\Omega$, and $u$ is zero in an open subset of $\Omega$. 
		Under the assumption \ref{lem:negative:2}, we have $u \equiv 0$ in $\Omega$ thanks to the strong maximum principle, a contradiction.
		If the assumption \ref{lem:negative:3} holds, then $u$ obeys the unique continuation property (see, e.g., \cite{muller}) and hence $u$ cannot vanish in an open subset of $\Omega$.
		Finally, under the assumption \ref{lem:negative:4}, all eigenfunctions of the $p$-Laplacian have explicit structure and cannot vanish in open intervals as well, see, e.g., \cite{DM}. 
	\end{proof}
	\begin{remark}
		The proof of Lemma \ref{lem:negative} in the case $p=2$ relies of the unique continuation property (UCP).
		Assuming the UCP is true for some $p \neq 2$, the statement of Lemma \ref{lem:negative} can be generalized accordingly.
		However, to the best of our knowledge, the validity of the UCP is not known for sign-changing eigenfunctions of the $p$-Laplacian in higher dimensions $N \geq 2$.
	\end{remark}
	
	The following result complements \cite[Proposition 3.9 (1)]{KQU}.
	\begin{proposition}\label{NoCriticalValue} 
		Assume that $\lambda\le \lambda_-^*$. 
		Then, $\I$ and $\wI$ have no positive critical values. 
	\end{proposition}
	\begin{proof} 
		Let $u$ be a nonzero critical point of $ \I$, and hence $u\in \N$. 
		Since $\N \cap \mathcal{A}^-=\emptyset$ for $\lambda \leq \lambda_-^*$ (see Lemma \ref{lem:NehariEmpty}), 
		we have $\intO a|u|^p \,dx \geq 0$, which yields $\I(u) \le 0$ by \eqref{eq:EonN1-0}. 
		If $u$ is a nonzero critical point of $\wI$, then $u$ is a nonzero nonnegative solution of \eqref{eq:P} (see Section~\ref{sec:mainresults}). 
		That is, $u$ is a nonzero critical point of $\I$, and the conclusion follows as above.
	\end{proof} 
	
	In contrast to Proposition \ref{NoCriticalValue}, the functionals $\I$ and $\wI$ might possess positive critical values for sufficiently large $\lambda$.
	\begin{proposition}\label{prop:nonlipschitz} 
		Assume that $a\equiv \text{const}<0$ in some open ball $B \subset \Omega_a^-$.
		Then there exists a sufficiently large $\bar{\lambda}$ such that \eqref{eq:P} possesses a nonnegative solution with compact support in $B$ and positive energy for any $\lambda > \bar{\lambda}$.
	\end{proposition}
	\begin{proof}
		Consider the Dirichlet problem 
		\begin{equation}\label{eq:Prad}
		\left\{
		\begin{aligned}
			-\Delta_p u &= |u|^{p-2}u - |u|^{q-2}u 
			&&\text{in}\ B_R, \\
			u&=0 &&\text{on}\ \partial B_R,
		\end{aligned}
		\right.
	\end{equation}
		where $B_R \subset \mathbb{R}^N$ is an open ball of some radius $R>0$ centered at the origin.
		It is known from, e.g., \cite[Theorem 1]{DH} (for $N=1$) and \cite[Proposition 2]{FLS} (for $N \geq 2$) that there exists $R>0$ such that \eqref{eq:Prad} has a radially symmetric nonnegative solution $u$ with compact support in $B_R$.
		Moreover, we have $u \in C_0^1(\overline{B_R})$, see Remark \ref{rem:reg}.
		By considering a function $v \in C_0^1(\overline{B})$ defined as $v(x) = A u(C (x-x_0))$ with appropriate constants $A,C,x_0>0$, it can be derived in much the same way as in \cite[Proposition 5.1]{DHI1} that $v$ is a compact support solution of \eqref{eq:P} in the ball $B$ (given in the statement of the proposition)
		for any sufficiently large $\lambda$. 
		Extending $v$ by zero outside of $B$, we obtain a nonnegative solution of \eqref{eq:P} in $\Omega$ (or even in the whole $\mathbb{R}^N$).
		Applying Lemma \ref{lem:negative}, we conclude that $v$ has positive energy.
	\end{proof}
	\begin{remark}
		The assumption of Proposition \ref{prop:nonlipschitz} on the weight $a$ can be weakened to cover some nonconstant weights, see, e.g., \cite[Remark 5.2 and Theorem 5.1]{DHI1} for the linear case $p=2$.
	\end{remark}

	\subsection{Palais--Smale type conditions}
	
	We provide two standard but useful compactness results.
	We denote by $\|\cdot\|_*$ the usual operator norm. 
	\begin{lemma}\label{lem:conv-gs-PS}
		Let $\{\lambda_n\} \subset \mathbb{R}$ converge to some $\lambda \in \mathbb{R}$.
		Let $\{u_n\}$ be a bounded sequence in $\W$ such that $\|I'_{\lambda_n}(u_n)\|_* \to 0$ as $n \to +\infty$.
		Then $\{u_n\}$ has a subsequence strongly convergent in $\W$ to a solution of \eqref{eq:P}. 
	\end{lemma} 
	\begin{proof} 
		Since $\{u_n\}$ is bounded, we may assume that it converges to some $u_0 \in \W$ weakly in $\W$ and strongly in $L^p(\Omega)$, up to a subsequence.
		We obtain from the convergence $\|I'_{\lambda_n}(u_n)\|_* \to 0$ that
		\begin{align*}
			\intO |\nabla u_n|^{p-2}&\nabla u_n\nabla (u_n-u_0)\,dx\\
			=
			&\lambda_n \intO |u_n|^{p-2} u_n (u_n-u_0)\,dx
			+\intO a |u_n|^{q-2} u_n (u_n-u_0)\,dx + o(1) \to 0 
		\end{align*}
		as $n\to +\infty$. 
		In view of the $(S_+)$-property of the $p$-Laplacian (see, e.g., \cite[Theorem 10]{dinica}), $u_n \to u_0$ strongly in $\W$. 
		As a consequence, we easily conclude that $u_0$ is a solution of \eqref{eq:P}. 
	\end{proof}
	
	\begin{lemma}\label{lem:PS} 
		Let $\lambda \neq \lambda_1(p)$. 
		Then 
		$ \wI$ satisfies the Palais--Smale condition. 
	\end{lemma} 
	\begin{proof}
		Let $\{u_n\}$ be any Palais--Smale sequence for $ \wI$, that is, $|\wI(u_n)|$ is bounded for all $n$ and $\| \wI'(u_n)\|_* \to 0$ as $n \to +\infty$. 
		Due to the $(S_+)$-property of the $p$-Laplacian (see, e.g., \cite[Theorem 10]{dinica}), $ \wI$ satisfies the Palais--Smale condition provided $\{u_n\}$ is bounded. 
		Suppose, contrary to our claim, that $\|\nabla u_n\|_p \to +\infty$, up to a subsequence, and consider the normalized functions $v_n:={u_n}/{\|\nabla u_n\|_p}$. 
		Let us show that $v_n \to \varphi_p$ strongly in $\W$.
		Since $\{v_n\}$ is bounded, it converges to some $v_0 \in \W$ weakly in $\W$ and strongly in $L^p(\Omega)$, up to a subsequence.
		We have
		\begin{equation}\label{eq:PS1}
			\left|
			\frac{1}{p}\langle \widetilde{E}_\lambda'(v_n), \xi\rangle
			-
			\frac{1}{\|\nabla u_n\|_p^{p-q}} \intO a (v_n)_+^{q-1} \xi \, dx
			\right|
			=
			\frac{|\langle  \wI'(u_n),\xi \rangle|}{\|\nabla u_n\|_p^{p-1}}
			\leq
			\frac{\|\nabla \xi\|_p\| \wI'(u_n)\|_*}{\|\nabla u_n\|_p^{p-1}}
		\end{equation}
		for any $\xi \in \W$.
		This yields $\langle \widetilde{E}_\lambda'(v_0), \xi\rangle = 0$ for any $\xi \in \W$. 
		Moreover, taking $\xi=v_n$ in \eqref{eq:PS1}, we get $\lim\limits_{n \to +\infty}\langle \widetilde{E}_\lambda'(v_n), v_n \rangle = 0$ and hence $v_0 \not\equiv 0$ in $\Omega$.
		That is, $v_0$ is an eigenfunction of the $p$-Laplacian corresponding to $\lambda$. 	
		Noting that
		$$
		\|\nabla (u_n)_-\|_p^p
		=
		|\langle  \wI^\prime (u_n),(u_n)_-\rangle|
		\leq
		\| \wI'(u_n)\|_*\|\nabla (u_n)_-\|_p,
		$$
		we conclude that $v_0\ge 0$ in $\Omega$. 
		Therefore, $\lambda=\lambda_1(p)$ and $v_0=\varphi_p$, since $\varphi_p$ is the only eigenfunction of the $p$-Laplacian with constant sign in $\Omega$ (see, e.g., \cite{anane1987}).
		This contradicts our assumption $\lambda \neq \lambda_1(p)$.
	\end{proof}

	\subsection{Behavior of sequences} 
	In this subsection, we collect several results on the behavior of functional sequences, which will be used in the proofs of our main theorems.
	
	\begin{lemma}\label{lem:bdd-PS}
		Let $\{\lambda_n\}\subset \mathbb{R}$ 
		and $\{u_n\}\subset \W$ be 
		sequences such that $u_n\ge 0$ a.e.\ in $\Omega$ for all $n \in \mathbb{N}$ and 
		\begin{equation*}\label{eq:bdd-PS-0}  
			\lambda_n\to\lambda \in \mathbb{R}, 
			\quad \|\nabla u_n\|_p \to +\infty,
			\quad 
			\|\In^\prime(u_n)\|_{*} \to 0 
			\quad \text{as}~ n \to +\infty.
		\end{equation*}
		Then $\lambda=\lambda_1(p)$ and 
		the sequence $\{v_n\}$, 
		where $v_n := u_n/\|\nabla u_n\|_p$, has a subsequence strongly convergent in $\W$ 
		to $\varphi_p$. 
	\end{lemma}
	\begin{proof} 
		Since $\{v_n\}$ is bounded in $\W$, we may suppose that $\{v_n\}$ converges to some $v_0 \in \W$ weakly in $\W$ and strongly 
		in $L^p(\Omega)$ and $L^q(\Omega)$, up to a subsequence.
		Consequently, we get
		\begin{align*}
			o(1) &= 
			\bigg< \In^\prime(u_n),\frac{v_n-v_0}{\|\nabla u_n\|_p^{p-1}} 
			\bigg> \\
			&=\intO |\nabla v_n|^{p-2}\nabla v_n\nabla (v_n-v_0)\,dx 
			\\
			&~\quad -\lambda_n\intO |v_n|^{p-2}v_n(v_n-v_0)\,dx 
			-\frac{1}{\|\nabla u_n\|_p^{p-q}} \intO a |v_n|^{q-2}v_n(v_n-v_0)\,dx 
			\\
			&=\intO |\nabla v_n|^{p-1}\nabla v_n\nabla (v_n-v_0)\,dx +o(1). 
		\end{align*}
		Hence, the $(S_+)$-property of the $p$-Laplacian implies that $v_n\to v_0$ strongly in $\W$ (see, e.g., \cite[Theorem 10]{dinica}), 
		whence $\|\nabla v_0\|_p=1$ and $v_0\not\equiv 0$ a.e.\ in $\Omega$. 
		
		Considering now $\langle I'_{\lambda_n}(u_n),\xi/\|\nabla u_n\|_p^{p-1}\rangle$ for any $\xi\in \W$, we have 
		\begin{equation*}
			\intO |\nabla v_n|^{p-2}\nabla v_n\nabla \xi\,dx 
			-\lambda_n \intO |v_n|^{p-2}v_n\xi\,dx 
			-\frac{1}{\|\nabla u_n\|_p^{p-q}}\intO a |v_n|^{q-2}v_n\xi\,dx
			= o(1).
		\end{equation*}
		Letting $n\to+\infty$ and recalling that $\|\nabla v_0\|_p=1$ and 
		$v_n\ge 0$ a.e.\ in $\Omega$ for all $n$, we deduce that 
		$v_0$ is a nonnegative eigenfunction of the $p$-Laplacian.
		This yields $\lambda=\lambda_1(p)$, since any higher eigenfunction must be sign-changing (see, e.g., \cite{anane1987}). 
		Finally, the simplicity of $\lambda_1(p)$ ensures that
		$v_0=\varphi_p$. 
	\end{proof}

	\begin{lemma}\label{lem:conv-gs}
		Let $\{\lambda_n\} \subset \mathbb{R}$ 
		be such that $M(\lambda_n)$ is attained, and $\lambda_n \to \lambda \in \mathbb{R}$ as $n \to +\infty$.
		Let $u_n$ be a minimizer of $M(\lambda_n)$, and let $\{u_n\}$ be bounded in $\W$. 
		If $\liminf\limits_{n\to+\infty} \In(u_n)<0$, then $\{u_n\}$ has a subsequence strongly convergent in $\W$ to a minimizer of $M(\lambda)$.
	\end{lemma} 
	\begin{proof} 
		Applying Lemma \ref{lem:conv-gs-PS}, we see that $\{u_n\}$ converges to a solution $u_0$ of \eqref{eq:P} strongly in $\W$, up to a subsequence. 
		The assumption $\liminf\limits_{n\to+\infty} \In(u_n)<0$ guarantees that $u_0 \not\equiv 0$ in $\Omega$, and hence $u_0 \in \N$.
		Let us prove that $u_0$ is a minimizer of $M(\lambda)$. 
		Fix any $w\in \N\cap \mathcal{A}^+$.
		By the continuity, the strict inequality $E_{\lambda_n}(w)>0$ holds for all sufficiently large $n$, and we obtain 
		$$
		\In (u_n)=M(\lambda_n)\le  \In(t_{\lambda_n}(w) w)
		=\min_{t > 0} \In(t w) \le  \In(w). 
		$$
		Letting $n\to+\infty$, we arrive at $ \I(u_0)\le  \I(w)$, and therefore 
		$u_0$ is a minimizer of $M(\lambda)$, see Remark~\ref{rem:m<0}.
	\end{proof}

	The following auxiliary result will be essential to prove Theorem \ref{NoGroundStates} \ref{NoGroundStates:2} and Lemma~\ref{lem:1}.
	
	\begin{lemma}\label{lem:convergence-to-eigenf}
		Let $p \geq 2q$.
		Assume $\partial \Omega$ to be connected when $N \geq 2$.
		Let $\lambda = \lambda_1(p)$ 
		and $\int_\Omega a \varphi_p^q \,dx = 0$. 
		Let $\{w_n\} \subset \W$ be such that $E_\lambda(w_n)>0$, $\intO a|w_n|^q \,dx >0$ for any $n \in \mathbb{N}$, and $w_n \to \varphi_p$ strongly in $\W$. 
		Then the following assertions hold:
		\begin{enumerate}[label={\rm(\roman*)}]	
			\item\label{lem:convergence-to-eigenf:1}
			If $p=2q$, then
			$$
			-\infty 
			<
			\liminf\limits_{n \to +\infty}  \I(t_\lambda(w_n) w_n)
			\leq 
			\limsup\limits_{n \to +\infty}  \I(t_\lambda(w_n) w_n) 
			\leq 0.
			$$ 
			\item\label{lem:convergence-to-eigenf:2}
			If $p>2q$, then $\lim\limits_{n \to +\infty}  \I(t_\lambda(w_n) w_n) = 0$.
		\end{enumerate}
	\end{lemma}
	\begin{proof}
		Throughout the proof, we always denote by $C$ a positive constant independent of $n \in \mathbb{N}$. 
		We decompose $w_n$ as $w_n = \gamma_n \varphi_p + v_n$, where $\gamma_n \in \mathbb{R}$ and $v_n \in \W$ are chosen in such a way that $\gamma_n = \|\varphi_p\|_2^{-2} \intO w_n \varphi_p \, dx$ and $\intO v_n \varphi_p \, dx = 0$ for all $n$. 
		Notice that $v_n\not=0$ for all $n$ because $\E(w_n)>0=\E(\varphi_p)$. 
		Since $w_n \to \varphi_p$ strongly in $\W$, it is not hard to deduce that 
		\begin{equation}\label{eq:converge1}
			\gamma_n \to 1
			\quad \text{and} \quad
			\|\nabla v_n\|_p \to 0
			\quad \text{as} \quad
			n \to +\infty.
		\end{equation}
		Recalling now that $\partial \Omega$ is connected, we can use the improved Poincar\'e inequality from \cite{takac} (see \cite[Corollary 1.2]{takac}) to provide the lower bound
		\begin{align}
			\label{eq:estimate_for_H}
			E_\lambda(w_n) 
			\geq C\left(|\gamma_n|^{p-2} \intO |v_n|^2 \, dx + \intO |v_n|^p \, dx \right) 
			\geq \dfrac{C}{2}
			\left(
			\intO |v_n|^2 \, dx + \intO |v_n|^p \, dx \right) 
			>0
		\end{align} 
		for all sufficiently large $n$. 
		Let us now estimate $\int_\Omega a |w_n|^q \,dx$ from above. 
		Recalling that $\intO a \varphi_p^q \,dx =0$ and applying the mean value theorem, for each $n$ we can find $\varepsilon_n \in (0,1)$  such that
		\begin{align*}
			\notag
			0 < \int_\Omega a |w_n|^q \,dx
			&= 
			|\gamma_n|^{q} \int_\Omega a \varphi_p^q \,dx 
			+ 
			q \int_\Omega a |\gamma_n \varphi_p+\varepsilon_n v_n|^{q-2}(\gamma_n \varphi_p+\varepsilon_n v_n) v_n \,dx
			\\
			&\leq 
			q \int_\Omega |a| |\gamma_n \varphi_p+\varepsilon_n v_n|^{q-1} |v_n| \,dx.
		\end{align*}
		Since $p \geq 2q > 2(q-1)$ by our assumption, we use the H\"older inequality and the convergences \eqref{eq:converge1} to obtain the following upper bound for all $n$:
		\begin{align}
			\notag
			q\intO |a||\gamma_n \varphi_p + \varepsilon_n v_n|^{q-1}|v_n| \, dx 
			&\leq C\left(\intO |\gamma_n \varphi_p + \varepsilon_n v_n|^{2(q-1)} \, dx\right)^\frac{1}{2} \left(\intO |v_n|^2 \, dx\right)^\frac{1}{2}
			\\
			\label{eq:estimate_for_G}
			&\leq 
			C \left(\intO |v_n|^2 \, dx\right)^\frac{1}{2} 
			\leq 
			C \left(
			\intO |v_n|^2 \, dx
			+
			\intO |v_n|^p \, dx
			\right)^\frac{1}{2}.
		\end{align}
		Combining now \eqref{eq:estimate_for_H} and \eqref{eq:estimate_for_G}, and recalling that $p \geq 2q$, we deduce that
		\begin{align*}
			0 
			&\geq 
			\limsup_{n \to +\infty}  \I(t_\lambda(w_n) w_n)
			\geq
			\liminf_{n \to +\infty}  \I(t_\lambda(w_n) w_n)
			\\
			&=
			\liminf_{n \to +\infty} \J(w_n)
			\geq
			- C \, 
			\limsup_{n \to +\infty} 
			\left(
			\intO |v_n|^2 \, dx
			+
			\intO |v_n|^p \, dx \right)^\frac{p-2q}{2(p-q)} > -\infty.
		\end{align*}
		Moreover, if $p>2q$, then we see that $\lim\limits_{n \to +\infty}  \I(t_\lambda(w_n) w_n) = 0$. 
	\end{proof}

	\section{The least energy at \texorpdfstring{$\lambda^*$}{lambda-*}}\label{sec:proof:NoGroundS}
	
	In this section, we prove Theorem \ref{NoGroundStates} and provide several auxiliary results on the properties of the critical set of $M(\lambda^*)$ which will be crucial for the proof of Theorem \ref{thm:1} given in Section~\ref{sec:thm1} below.

	\subsection{Proof of Theorem \ref{NoGroundStates}}\label{sec:proof:NoGroundS:1}
	
	\begin{proof*}{the assertion \ref{NoGroundStates:1}} 
		As was mentioned in Remark \ref{rem:NoGS}, the proof can be found in \cite{QS}. 
		We provide alternative arguments for the sake of completeness and clarity.
		Let us choose an increasing sequence $\{\lambda_n\} \subset \mathbb{R}$ 
		convergent to $\lambda^*$, and for each $\lambda_n$ 
		we denote by $u_n$ a nonnegative ground state of {\renewcommand{\lambdaM}{\lambda_n}\eqref{eq:P}}, see Theorem \ref{thm:GS} for the existence. 
		Lemma~\ref{lem:bdd-PS} ensures the boundedness of 
		$\{u_n\}$ in $\W$ because 
		$\lambda^*>\lambda_1(p)$ by Proposition \ref{prop:+}. 
		Let us prove that $\liminf\limits_{n\to+\infty} \In(u_n)<0$.
		If this claim is established, then Lemma \ref{lem:conv-gs} implies the assertion \ref{NoGroundStates:1}. 
		Fix any $w\in\mathcal{N}_{\lambda^*}\cap \mathcal{A}^+$. 
		Since $E_{\lambda_n}(w)>0$ for any sufficiently large $n$, 
		we get $t_{\lambda_n}(w)w\in \mathcal{N}_{\lambda_n}\cap \mathcal{A}^+$, and hence
		\begin{equation}\label{eq:NGS-1} 
			\In(u_n)=M(\lambda_n)\le 
			\In(t_{\lambda_n}(w)w)
			=
			\min_{t > 0} \In(tw)
			\leq
			\In(w).
		\end{equation} 
		Recalling that $w\in\mathcal{N}_{\lambda^*}\cap \mathcal{A}^+$ and passing to the limit in \eqref{eq:NGS-1}, we arrive at 
		$$
		\limsup_{n\to+\infty}\In(u_n)=\limsup_{n\to+\infty}M(\lambda_n)\le 
		I_{\lambda^*}(w)<0.
		$$ 
		Applying Lemma \ref{lem:conv-gs}, we finish the proof.
	\end{proof*}

	\begin{proof*}{the assertion \ref{NoGroundStates:2}} 
		Assume first that $p<2q$. 
		We have $\lambda^*=\lambda_1(p)$ by Proposition~\ref{prop:+}, and so $E_{\lambda^*}(\varphi_p)=0$ and
		\begin{equation}\label{eq:noground1}
			\langle E_{\lambda^*}'(\varphi_p), \theta\rangle=0
			\quad\text{for any}~~ \theta \in C_0^{\infty}(\Omega).
		\end{equation}
		Since $\varphi_p>0$ in $\Omega$, we can find $\theta \in C_0^{\infty}(\Omega)$ satisfying $\int_\Omega a \varphi_p^{q-1} \theta \,dx	> 0$.
		By the continuity, there exist $\varepsilon_0>0$ and $C>0$ such that 
		\begin{equation}\label{g<0-2}
			q \int_\Omega a |\varphi_p + \varepsilon \theta|^{q-2} (\varphi_p+\varepsilon \theta) \theta \,dx \geq C > 0
			\quad \text{for all}~
			\varepsilon \in [-\varepsilon_0,\varepsilon_0]. 
		\end{equation}	
		Notice that $\theta\not\in\mathbb{R}\varphi_p$ in view of the assumption $\int_\Omega a \varphi_p^q \,dx = 0$. 
		Therefore, the simplicity of $\lambda_1(p)$ gives $E_{\lambda^*}(\varphi_p+\varepsilon\theta) > 0$ for any $\varepsilon \neq 0$.
		
		Fix any $\varepsilon\in(0,\varepsilon_0]$ and denote  $u_\varepsilon := \varphi_p + \varepsilon \theta$. According to the mean value theorem, there exist $\varepsilon_1\in (0,\varepsilon)$ and $\varepsilon_2\in(0,\varepsilon)$ such that
		\begin{align} 
			\label{eq:prop3.7:1}
			0 < E_{\lambda^*}(u_\varepsilon) &= 
			E_{\lambda^*}(\varphi_p) + 
			\varepsilon\langle E_{\lambda^*}'(\varphi_p+\varepsilon_1\theta), \theta \rangle 
			=\varepsilon \langle E_{\lambda^*}'(\varphi_p+\varepsilon_1\theta), \theta \rangle
		\end{align}
		and, by \eqref{g<0-2}, 
		\begin{equation}
			\label{eq:prop3.7:2}
			\int_\Omega a |u_\varepsilon|^q \,dx = 
			\int_\Omega a \varphi_p^q \,dx 
			+ 
			\varepsilon 
			q \int_\Omega a |\varphi_p + \varepsilon_2 \theta|^{q-2} (\varphi_p+\varepsilon_2 \theta) \theta \,dx
			\geq 
			\varepsilon C > 0. 
		\end{equation}
		Consequently, we have $t_{\lambda^*}(u_\varepsilon)u_\varepsilon\in \mathcal{N}_{\lambda^*} \cap \mathcal{A}^+$. 
		Our aim now is to study the behavior of $I_{\lambda^*}(t_{\lambda^*}(u_\varepsilon)u_\varepsilon)$ as $\varepsilon \to 0$. 
		To this end, we estimate the right-hand side of \eqref{eq:prop3.7:1} from above.
		It is known (see, e.g., \cite[Chapter 12]{lind}) that there exists $C_1>0$ such that 
		\begin{equation}\label{eq:inequality}
			0 \le \langle |x|^{p-2}x-|y|^{p-2}y, x-y\rangle_{\mathbb{R}^N} 
			\le 
			\left\{
			\begin{aligned}
				&C_1 |x-y|^p 					&&{\rm if}\ 1<p\le 2,\\ 
				&C_1 |x-y|^2(|x|+|y|)^{p-2}	&&{\rm if}\ p\ge 2,
			\end{aligned}
			\right.
		\end{equation}
		for all $x,y\in\mathbb{R}^N$.
		Recalling $0<\varepsilon_1<\varepsilon\le \varepsilon_0$ and applying \eqref{eq:noground1} and \eqref{eq:inequality}, 
		we obtain 
		\begin{align*} 
			\langle 
			E_{\lambda^*}'(\varphi_p+\varepsilon_1\theta), \theta 
			\rangle
			&= 
			\langle 
			E_{\lambda^*}'(\varphi_p+\varepsilon_1\theta), \theta 
			\rangle 
			-\langle 
			E_{\lambda^*}'(\varphi_p), \theta 
			\rangle \\
			&=
			\frac{1}{\varepsilon_1}
			\langle E_{\lambda^*}'(\varphi_p+\varepsilon_1\theta)-E_{\lambda^*}^\prime(\varphi_p), 
			(\varphi_p+\varepsilon_1\theta) - \varphi_p 
			\rangle 
			\\
			&\leq 
			\left\{
			\begin{aligned}
				&\frac{C_1 \varepsilon_1^p}{\varepsilon_1}\|\nabla \theta\|_p^p
				\leq C_2 \varepsilon_1^{p-1} &&\text{if}~ 1<p\le 2, 
				\\ 
				&\frac{C_1 \varepsilon_1^2}{\varepsilon_1}
				\intO |\nabla\theta|^2
				(2 |\nabla \varphi_p|+\varepsilon_1|\nabla \theta|)^{p-2}  \, dx
				\le C_2 \varepsilon_1 
				&&\text{if}~ p\ge 2,
			\end{aligned}
			\right.
		\end{align*} 
		where $C_2>0$ does not depend on $\varepsilon_1$.
		Thus, we conclude from \eqref{eq:prop3.7:1} that 
		\begin{equation}\label{eq:upperE}
		E_{\lambda^*}(u_\varepsilon)\leq 
		\left\{
		\begin{aligned}
		&C_2 \varepsilon^{p} &&\text{if}~ 1<p\le 2,\\
		&C_2 \varepsilon^{2} &&\text{if}~ p\ge 2. 
		\end{aligned}
		\right.
		\end{equation}
		Recalling that $t_{\lambda^*}(u_\varepsilon)u_\varepsilon\in \N\cap \mathcal{A}^+$ and using the estimates \eqref{eq:prop3.7:2} and \eqref{eq:upperE}, we arrive at
		\begin{align*} 
			M(\lambda^*) 
			&\leq I_{\lambda^*}(t_{\lambda^*}(u_\varepsilon)u_\varepsilon)\\
			&=J_{\lambda^*}(u_\varepsilon)
			=- \frac{p-q}{pq}\,
			\frac{(\int_\Omega a |u_\varepsilon|^q \,dx)^{\frac{p}{p-q}}}{\left(E_{\lambda^*}(u_\varepsilon)\right)^{\frac{q}{p-q}}} \le 
			\left\{
			\begin{aligned}
				&-C_3 \varepsilon^{\frac{p}{p-q} - \frac{pq}{p-q}}
				&&{\rm if}\ 1<p\le 2, 
				\\[3mm]
				&-C_3 \varepsilon^{\frac{p}{p-q} - \frac{2q}{p-q}}
				&&{\rm if}\ p\ge 2, 
			\end{aligned}
			\right.
		\end{align*}
		where $C_3>0$ is independent of $\varepsilon$. 
		Since $p<2q$, we get $M({\lambda^*}) = -\infty$ by letting $\varepsilon \to 0$.
		
		Assume now that $p \geq 2q$. 
		We start by showing that $M(\lambda^*)>-\infty$.
		Suppose, by contradiction, that there exists a sequence $\{u_n\} \subset \mathcal{N}_{\lambda^*}\cap \mathcal{A}^+$ satisfying $J_{\lambda^*}(u_n) \to -\infty$ as $n \to +\infty$. 
		For each $n$, we define the normalized function $w_n:=u_n/\|\nabla u_n\|_p$. 
		Since the fibered functional $J_{\lambda^*}$ is $0$-homogeneous, we have $J_{\lambda^*}(w_n) \to -\infty$.
		The boundedness of $\{w_n\}$ in $\W$ implies its weak convergence in $\W$ and strong convergence in $L^p(\Omega)$ to some $w_0$, up to a subsequence.
		Hence, recalling that ${\lambda^*}=\lambda_1(p)$, we get $0\le E_{\lambda^*}(w_0)\le \liminf\limits_{n\to+\infty} E_{\lambda^*}(w_n)$.
		On the other hand, the convergence  $J_{\lambda^*}(w_n) \to -\infty$ and the boundedness of $\{w_n\}$ imply that  $\liminf\limits_{n\to+\infty} E_{\lambda^*}(w_n)=0$, see \eqref{eq:Poh}. 
		Consequently, $E_{\lambda^*}(w_0)=0$, and hence $w_n \to w_0 = \varphi_p$ strongly in $\W$.
		Applying Lemma \ref{lem:convergence-to-eigenf}, we deduce that 
		$\{J_{\lambda^*}(w_n)\}$ has to be bounded from below, which gives a contradiction. 
		In conclusion, in view of Remark \ref{rem:m<0}, we have $M(\lambda^*) \in (-\infty, 0)$.
		
		Finally, we prove that $M(\lambda^*)$ is attained provided $p>2q$. 
		Let $\{u_n\}\subset \mathcal{N}_{\lambda^*}\cap \mathcal{A}^+$ 
		be a minimizing sequence for $M(\lambda^*)$. 
		Let us show the boundedness of $\{u_n\}$. 
		Suppose, contrary to our claim, that $\lim\limits_{n\to+\infty}\|\nabla u_n\|_p=+\infty$ along a subsequence.
		Denoting, as above, $w_n:=u_n/\|\nabla u_n\|_p$, we see that 
		$w_n \to w_0 \in \W$ weakly in $\W$ and strongly in $L^p(\Omega)$, up to a subsequence. 
		Since $\{u_n\}\subset \mathcal{N}_{\lambda^*}\cap \mathcal{A}^+$, each $w_n$ satisfies 
		$$
		1-\lambda^*\|w_n\|_p^p=E_{\lambda^*}(w_n)
		=\dfrac{1}{\|\nabla u_n\|_p^{p-q}}\,\intO a |w_n|^q \, dx.
		$$
		Letting $n\to +\infty$ and recalling that $\lambda^*=\lambda_1(p)$, we get $w_0\not\equiv0$ a.e.\ in $\Omega$ and 
		$E_{\lambda^*}(w_0)=0$, 
		whence $\{w_n\}$ converges to $\varphi_p$ strongly in $\W$. 
		According to Lemma~\ref{lem:convergence-to-eigenf}, we derive 
		the following contradiction:
		$$
		0=\lim_{n\to +\infty}I_{\lambda^*}(t_{\lambda^*}(w_n)w_n)
		=\lim_{n\to +\infty}I_{\lambda^*}(u_n)=M(\lambda^*)<0, 
		$$
		where we used the equality $t_{\lambda^*}(w_n)w_n=u_n$ by $u_n\in\mathcal{N}_{\lambda^*}$, $n \in \mathbb{N}$.
		That is, $\{u_n\}$ is bounded in $\W$, and hence $\{u_n\}$ converges to some $u_0 \in \W$ weakly $\W$ and strongly in $L^p(\Omega)$, up to a subsequence.
		Since $E_{\lambda^*}(u_0)\ge 0$ in view of $\lambda^*=\lambda_1(p)$, 
		the inequality $0>M(\lambda^*)\ge I_{\lambda^*}(u_0)$ 
		leads to $\intO a|u_0|^q\,dx>0$, and hence 
		$u_0\not\in\mathbb{R}\varphi_p$ because of the assumption 
		$\intO a\varphi_p^q\,dx=0$. 
		Hence, the simplicity of 
		$\lambda^*=\lambda_1(p)$ guarantees $E_{\lambda^*}(u_0)>0$, and we get
		$$
		M(\lambda^*)=\liminf_{n\to+\infty}I_{\lambda^*}(u_n)
		\ge I_{\lambda^*}(u_0)\ge I_{\lambda^*}(t_{\lambda^*}(u_0)u_0) 
		\ge M(\lambda^*), 
		$$
		which means that $u_0$ is a minimizer of $M(\lambda^*)$.
	\end{proof*}

	\begin{proof*}{the assertion \ref{NoGroundStates:3}} 
		Since $\intO a\varphi_p^q\,dx>0$, we have $\lambda^*=\lambda_1(p)$ by Proposition \ref{prop:+}, and hence $E_{\lambda^*}(\varphi_p)=0$. 
		Fix any $\theta \in \W \setminus \mathbb{R}\varphi_p$ and consider $w_\varepsilon := \varphi_p + \varepsilon \theta$ for $\varepsilon > 0$. 
		The simplicity of $\lambda_1(p)$ ensures that  $E_{\lambda^*}(w_\varepsilon) > 0$.
		Thus, we have $t_{\lambda^*}(w_\varepsilon)w_\varepsilon \in \mathcal{N}_{\lambda^*} \cap \A^+$ for any sufficiently small $\varepsilon>0$, and 
		\begin{align*} 
			M(\lambda^*) \leq  
			I_{\lambda^*}(t_{\lambda^*}(w_\varepsilon)w_\varepsilon)
			=J_{\lambda^*}(w_\varepsilon) 
			=-\, \frac{p-q}{p q}\, \frac{(\int_\Omega a |w_\varepsilon|^q \,dx)^{\frac{p}{p-q}}}{E_{\lambda^*}(w_\varepsilon)^{\frac{q}{p-q}}}
			\to -\infty 
		\end{align*}
		as $\varepsilon\to 0$, since 
		$\intO a|w_\varepsilon|^q\,dx\to \intO a\varphi_p^q\,dx>0$ and 
		$E_{\lambda^*}(w_\varepsilon) \to E_{\lambda^*}(\varphi_p)=0$. 
	\end{proof*}

	\subsection{Properties of the critical set of \texorpdfstring{$\widetilde{M}(\lambda^*)$}{M-tilde(lambda-*)}}\label{sec:K*}
	Throughout this section, we always assume that either of the following two assumptions is satisfied:
	\begin{enumerate}[label={\rm(\Roman*)}]
		\item\label{assump:1} $\intO a\varphi_p^q\,dx< 0$.
		\item\label{assump:2} $\intO a\varphi_p^q\,dx= 0$, $p>2q$, and $\partial \Omega$ is connected provided $N\ge 2$. 
	\end{enumerate}
	These assumptions coincide with the assumptions \ref{thm:1:1} and \ref{thm:1:2} of Theorem \ref{thm:1}, respectively.
	In both cases, Theorem \ref{NoGroundStates} asserts that $M(\lambda^*)$ is attained.
	Recall that $M(\lambda^*)=\widetilde{M}(\lambda^*)$, see Proposition \ref{prop:wm=m}. 
	
	In order to prove Theorem \ref{thm:1}, we need to establish several key properties of the set of minimizers of $\widetilde{M}(\lambda^*)$ (or, equivalently, the set of \textit{nonnegative} minimizers of $M(\lambda^*)$):
	\begin{equation}\label{eq:K}
		K^* 
		:=
		\{
		u \in \widetilde{\mathcal{N}}_{\lambda^*}:~ \widetilde{I}_{\lambda^*}(u)=\widetilde{M}(\lambda^*)
		\}.
	\end{equation}
	In particular, any $u \in K^*$ is a critical point of $\widetilde{I}_{\lambda^*}$. 
	\begin{lemma}\label{lem:1}
		$K^*$ is a compact set. 
		Moreover, 
		for any $u \in K^*$ there holds
		\begin{equation}\label{eq:lem1}
			\widetilde{E}_{\lambda^*}(u) 
			=
			\intO a u_+^q \,dx 
			=
			-\widetilde{M}(\lambda^*) \frac{pq}{p-q} > 0.
		\end{equation}
	\end{lemma}
	\begin{proof}
		If $\lambda^* > \lambda_1(p)$ (which is the case of  the assumption \ref{assump:1}), then the compactness of $K^*$ follows from Lemma \ref{lem:PS}.
		Suppose that $\lambda^* = \lambda_1(p)$, that is, we are under the assumption~\ref{assump:2}.
		As in Lemma \ref{lem:PS}, in view of the $(S_+)$-property of the $p$-Laplacian, it is sufficient to establish the boundedness of an arbitrary sequence $\{u_n\} \subset K^*$.
		Let $\{u_n\}$ be such sequence. 
		Suppose, by contradiction, that $\|\nabla u_n\|_p \to +\infty$ as $n \to +\infty$, up to a subsequence.
		Denoting $w_n := u_n/\|\nabla u_n\|_p$, we deduce from Lemma \ref{lem:bdd-PS} that 
		$w_n \to \varphi_p$ strongly in $\W$.
		Moreover, we evidently have $t_{\lambda^*}(w_n)w_n = u_n \in \mathcal{N}_{\lambda^*} \cap \mathcal{A}^+$.
		Applying Lemma \ref{lem:convergence-to-eigenf}, we get the following contradiction:
		$$
		0>\widetilde{M}(\lambda^*)
		=
		\widetilde{I}_{\lambda^*}(u_n)
		=
		\widetilde{I}_{\lambda^*}(t_{\lambda^*}(w_n)w_n) \to 0
		\quad
		\text{as}~ n \to +\infty.
		$$
		Finally, the equality \eqref{eq:lem1} directly follows from \eqref{eq:EonN1-0}.
	\end{proof}

	Let us now consider the closed $\delta$-neighborhood of $K^*$ with some $\delta>0$:
	\begin{equation}\label{eq:K*delta}
		K^*_\delta 
		:=
		\{
		u \in \W:~
		\text{dist}(u,K^*) \leq \delta
		\},
	\end{equation}
	where $\text{dist}(u,K^*) := \inf\{\|\nabla (u-v)\|_p: v \in K^*\}$.
	We notice that the infimum is achieved in view of the compactness of $K^*$.
	Hereinafter, by $B_\delta(u)$ we denote an open ball in $\W$ of radius $\delta$ centered at $u$.
	
	\begin{lemma}\label{lem:2}
		$K^*_\delta$ is weakly sequentially compact (and bounded). 
		Moreover, if $\delta>0$ is small enough, then there exists $C_1>0$ such that 
		\begin{equation}\label{eq:lem2}
			\widetilde{E}_{\lambda^*}(u) \geq C_1
			\quad \text{and} \quad
			\intO a u_+^q \,dx 
			\geq C_1
			\quad \text{for any} \quad  u \in K^*_\delta.
		\end{equation}
	\end{lemma}
	\begin{proof}
		The boundedness simply follows from the compactness of $K^*$.
		Hence, any $\{u_n\} \subset K^*_\delta$ converges weakly in $\W$ to some $u_0 \in \W$, up to a subsequence.
		In view of the definition of $K^*_\delta$, for any $n \in \mathbb{N}$ there exists $v_n \in K^*$ such that $\|\nabla (u_n-v_n)\|_p \leq \delta$.
		Since $K^*$ is compact, $\{v_n\}$ converges to some $v_0 \in K^*$ strongly in $\W$, up to a subsequence.
		Thus, for any $\varepsilon>0$ and any sufficiently large $n$, $v_n \in B_\varepsilon(v_0)$. 
		Consequently, we have $u_n \in {B_{\varepsilon+\delta}(v_0)}$. 
		Since $\overline{B_{\varepsilon+\delta}(v_0)}$ is weakly closed, we conclude that $u_0 \in \overline{B_{\varepsilon+\delta}(v_0)}$. 
		Recalling that $\varepsilon>0$ was arbitrary, we deduce that $u_0 \in \overline{B_{\delta}(v_0)}$, and hence $u_0 \in K^*_\delta$. 
		
		Let us prove the second part of the lemma.
		By the continuity 
		and in view of Lemma \ref{lem:1}, for every $v \in K^*$ there exists $\varepsilon(v) >0$ such that 
		\begin{equation}\label{eq:lem2:2}
			\widetilde{E}_{\lambda^*}(u) \geq 
			-\frac{\widetilde{M}(\lambda^*)}{2} \frac{pq}{p-q}
			\quad \text{and} \quad
			\intO a u_+^q \,dx 
			\geq 
			-\frac{\widetilde{M}(\lambda^*)}{2} \frac{pq}{p-q}
		\end{equation}
		for any $u \in \overline{B_{\varepsilon(v)}(v)}$.
		Since $K^*$ is compact, there exist $v_1,\dots,v_k \in K^*$ such that $K^* \subset \cup_{j=1}^k B_{\varepsilon({v_j})/2}(v_j)$.
		Taking any $0<\delta<\min\{\varepsilon({v_1}),\dots,\varepsilon({v_k})\}/2$,
		we conclude that $K^*_\delta \subset \cup_{j=1}^k \overline{B_{\varepsilon(v_j)}(v_j)}$  since 
		$\varepsilon({v_j})/2 +\delta\le \varepsilon({v_j})$ for each $j$. 
		That is, \eqref{eq:lem2:2} remains valid for any $u \in K^*_\delta$.
	\end{proof}
	
	Finally, we consider the boundary of $K^*_\delta$:
	\begin{equation}\label{eq:K*delta-par}
		\partial K^*_\delta 
		:=
		\{
		u \in \W:~
		\text{dist}(u,K^*) = \delta
		\}.
	\end{equation}
	The main property of $\partial K^*_\delta$ is given in the following lemma which says that the elements of $\partial K^*_\delta$ have strictly higher energy than the elements of $K^*$. 
	\begin{lemma}\label{lem:3}
		Let $\delta>0$ be small enough to satisfy \eqref{eq:lem2}.
		Then
		$$
		\widetilde{M}(\lambda^*)<\inf\{\widetilde{I}_{\lambda^*}(u):~u\in \partial K^*_\delta\,\}. 
		$$
	\end{lemma}
	\begin{proof}
		Let $\delta>0$ be as required. 
		We have
		\begin{equation}\label{eq:lem3:1}
			\widetilde{M}(\lambda^*) 
			\leq 
			\widetilde{I}_{\lambda^*}(u)
			~~\text{for any}~
			u \in K^*_\delta.
		\end{equation}
		Indeed, in view of \eqref{eq:lem2}, 	
		for any $u\in K^*_\delta$ there exists $t>0$ such that $tu \in \widetilde{\mathcal{N}}_{\lambda^*}$ and $\widetilde{I}_{\lambda^*}(tu)<0$ (cf.\ Section \ref{sec:fiber}), and hence
		\begin{equation}\label{eq:chain1}
		\widetilde{M}(\lambda^*) \leq \widetilde{I}_{\lambda^*}(tu)
		=
		\min_{s>0} \widetilde{I}_{\lambda^*}(su)
		\leq 
		\widetilde{I}_{\lambda^*}(u).
		\end{equation}
		
		Suppose now, by contradiction to our main claim, that there exists a sequence $\{u_n\} \subset \partial K^*_\delta$ such that $\widetilde{I}_{\lambda^*}(u_n) \searrow \widetilde{M}(\lambda^*)$. Since $K^*_\delta$ is weakly sequentially compact by Lemma \ref{lem:2}, $\{u_n\}$ converges to some $u_0\in K^*_\delta$ weakly in $\W$ and strongly in $L^p(\Omega)$, up to a subsequence.
		If $\|\nabla u_0\|_p < \liminf\limits_{n\to +\infty} \|\nabla u_n\|_p$, i.e., there is no strong convergence in $\W$, then we get
		$$
		\widetilde{I}_{\lambda^*}(u_0)
		<
		\liminf_{n\to +\infty} 
		\widetilde{I}_{\lambda^*}(u_n) = \widetilde{M}(\lambda^*),
		$$
		a contradiction to \eqref{eq:lem3:1}.
		Therefore, $\|\nabla u_0\|_p = \liminf\limits_{n\to +\infty} \|\nabla u_n\|_p$, which means that $u_n \to u_0$ strongly in $\W$, and hence $u_0 \in \partial K^*_\delta \,(\subset K^*_\delta)$. 
		In view of \eqref{eq:chain1}, we obtain the chain
		$$
		\widetilde{M}(\lambda^*) =\lim_{n\to +\infty} \widetilde{I}_{\lambda^*}(u_n)= \widetilde{I}_{\lambda^*}(u_0) \geq \min_{s>0} \widetilde{I}_{\lambda^*}(su_0)
		\geq 
		\widetilde{M}(\lambda^*),
		$$
		which implies that $u_0 \in \widetilde{\mathcal{N}}_{\lambda^*}$ and $u_0$ is a minimizer of $\widetilde{M}(\lambda^*)$.
		That is, $u_0 \in K^*$.
		Thus, we have $u_0 \in K^*$ and $u_0 \in \partial K^*_\delta$, simultaneously.
		But this is impossible since $K^* \cap \partial K^*_\delta = \emptyset$.
	\end{proof}
	
	\begin{remark}
		Any minimizer of ${M}(\lambda^*)$ is a local minimum point of $\I$ (see Proposition \ref{prop:charact}), but we do not know a priori whether these minimizers are \textit{strict} local minimum points.
		To put this less formally, Lemma \ref{lem:3} asserts that a neighborhood of the set of all nonnegative minimizers of $M(\lambda^*)$ has a strict local minimum type structure.
	\end{remark}

	\section{Qualitative properties of \texorpdfstring{$M$}{M} and \texorpdfstring{$M^-$}{M-}}\label{sec:qualitative}
	
	In this section, we establish several properties of the levels $M$ and $M^-$ and the corresponding minimizers.
	In particular, we prove Propositions \ref{prop:charact}, \ref{prop:properties-m}, and \ref{prop:PCV}.
 
	\begin{proof*}{Proposition~\ref{prop:charact}}\label{page:prop:charact}
		Let $u$ be a minimizer of $M(\lambda)$. 
		Since $M(\lambda)<0$ by Remark \ref{rem:m<0},
		we deduce from Lemma \ref{lemm:crit} that $u$ is a solution of \eqref{eq:P}.
		If $u$ is not a ground state of \eqref{eq:P}, then there exists a solution $w$ such that $\I(w) < \I(u)$, which contradicts the definition of $M(\lambda)$ since $w \in \N$. 
		Let us show that $u$ is a local minimum point of $\I$. 
		Since $0<E_\lambda(u)=\intO a|u|^q\,dx$, we can find a sufficiently small $r>0$ such that 
		$E_\lambda(v)>0$ and $\intO a|v|^q\,dx>0$ for any
		$v\in B_r(u)$, i.e., $\|\nabla (u-v)\|_p<r$. 
		Thus, for each $v\in B_r(u)$ we have
		$t_{\lambda}(v)v\in \mathcal{N}_\lambda\cap \mathcal{A}^+$, and therefore
		$$
		\I(u)=\inf \{\I(w):\, w \in \mathcal{N}_\lambda\} 
		\le  \I(t_{\lambda}(v)v)
		=\min_{t > 0} \I(tv)\le  \I(v), 
		$$
		which means that $u$ is a local minimizer of $\I$.
		
		Let us prove that either $u>0$ or $u<0$ in $\Omega_a^+$.
		Clearly, if $u$ is a minimizer of $M(\lambda)$, then so is $v:=|u|$. 
		Hence $v$ is a nonnegative local minimum point of $\I$, and we deduce from Lemma~\ref{lem:positivity} that $v>0$ in $\Omega_a^+$.
		Consequently, $u$ cannot change the sign in $\Omega_a^+$.
		
		Let us show that the minimizer $u$ is a global minimum point of $\I$ whenever $\lambda \leq \lambda_1(p)$. 
		The case $\lambda < \lambda_1(p)$ is simple, see the discussion at the beginning of Section \ref{sec:subcritical}.
		Let $\lambda=\lambda_1(p)$ and suppose, by contradiction, that there exists $w$ such that $\I(w) < \I(u) \,(<0)$.
		Noting that $E_\lambda(w) \geq 0$ by the definition of $\lambda_1(p)$, we obtain $\intO a |w|^q \,dx > 0$.
		If $E_\lambda(w) > 0$, then $t_\lambda(w)w \in \N \cap \mathcal{A}^+$ and $\I(t_\lambda(w)w) \leq \I(w)$, and hence we get a contradiction to the definition of $M(\lambda)$.
		If $E_\lambda(w)=0$, then $w = t\varphi_p$ for some $t \neq 0$. 
		That is, $\intO a \varphi_p^q \,dx > 0$. 
		But in this case Theorem \ref{NoGroundStates} \ref{NoGroundStates:3} yields $M(\lambda)=-\infty$, which is impossible since $u$ is a minimizer of $M(\lambda)$ by the assumption.
	\end{proof*}

	\begin{proof*}{Proposition~\ref{prop:properties-m}}\label{page:proof:properties-m}
		\ref{prop:properties-m:1} 
		We start by showing that $M$ is (strictly) decreasing on $(-\infty,\lambda^*]$. 
		Recall that $M(\lambda)$ is attained for any $\lambda<\lambda^*$, see Theorem \ref{thm:GS}.
		Let $u_\lambda$ be 
		a corresponding  minimizer, that is, 
		$u_\lambda\in\mathcal{N}_\lambda\cap \mathcal{A}^+$ and $M(\lambda)=\I(u_\lambda)$. 
		
		Taking any $\lambda<\mu<\lambda^*$, we conclude from the definition \eqref{eq:lambda+*} of $\lambda^*$ that $E_\lambda(u_\lambda)>E_\mu(u_\lambda)>0$.
		Therefore, $t_\mu(u_\lambda)u_\lambda\in\mathcal{N}_\mu\cap \mathcal{A}^+$, and we obtain the monotonicity:
		\begin{align*} 
			M(\lambda)
			=\I (u_\lambda)=J_\lambda(u_\lambda) 
			&=
			-\,\frac{p-q}{p q}\, \frac{\left(\int_\Omega a |u_\lambda|^q \,dx\right)^{\frac{p}{p-q}}}{\left(E_\lambda(u_\lambda)\right)^{\frac{q}{p-q}}}
			\\ 
			&>-\,\frac{p-q}{p q}\, \frac{\left(\int_\Omega a |u_\lambda|^q \,dx\right)^{\frac{p}{p-q}}}{\left(E_\mu(u_\lambda)\right)^{\frac{q}{p-q}}}
			=J_\mu(u_\lambda)=I_\mu(t_\mu(u_\lambda)u_\lambda)\ge M(\mu).
		\end{align*} 
		Let us show now that $M(\lambda)> M(\lambda^*)$ for any $\lambda<\lambda^*$.
		In view of the definition of $\lambda^*$, we have $E_{\lambda^*}(u_\lambda) \geq 0$.
		If $E_{\lambda^*}(u_\lambda)>0$, then the same arguments as above yield the desired monotonicity.
		Assume that $E_{\lambda^*}(u_\lambda)=0$.
		Recalling that $\intO a |u_\lambda|^q \,dx > 0$, we conclude that $u_\lambda$ is a minimizer of $\lambda^*$, and hence Proposition~\ref{prop:+} implies that $\lambda^*=\lambda_1(p)$ and $u_\lambda = t\varphi_p$ for some $t \neq 0$. 
		Therefore, Theorem \ref{NoGroundStates} \ref{NoGroundStates:3} gives $M(\lambda^*)=-\infty$, and the monotonicity of $M(\lambda)$ follows.
		
		Finally, recalling that $M(\lambda)=-\infty$ for any $\lambda>\lambda^*$ (see Theorem \ref{thm:GS}), we deduce that $M$ is nonincreasing on $\mathbb{R}$, which completes the proof.
		
		\ref{prop:properties-m:3}
		The proof follows from Proposition \ref{prop:behavior-gs} \ref{prop:behavior-gs:1} below.
		
		\ref{prop:properties-m:4}
		The proof follows from Proposition \ref{prop:behavior-gs} \ref{prop:behavior-gs:2}, \ref{prop:behavior-gs:3} below.
	\end{proof*}

	Let us now prove the following general facts.
	\begin{proposition}\label{prop:behavior-gs} 
		Let $\{\lambda_n\} \subset \mathbb{R}$ be a convergent sequence 
		such that $\lambda_n<\lambda^*$, 
		and set $\lambda:=\lim\limits_{n\to+\infty}\lambda_n$. 
		Let $u_n$ be a ground state of {\renewcommand{\lambdaM}{\lambda_n}\eqref{eq:P}}. 
		Then the following assertions are satisfied:
		\begin{enumerate}[label={\rm(\roman*)}]
			\item\label{prop:behavior-gs:1} 
			Let $\lambda<\lambda^*$. 
			Then $\{u_n\}$ is bounded in $\W$ and has a subsequence strongly convergent in $\W$ to a ground state of \eqref{eq:P} as $n \to +\infty$.
			\item\label{prop:behavior-gs:2} 
			Let $\lambda=\lambda^*$ and $M(\lambda^*)=-\infty$. 
			Then $\lim\limits_{n\to+\infty}\|\nabla u_n\|_p = +\infty$, $\lim\limits_{n\to+\infty}\In(u_n) = -\infty$,  $|u_n|/\|\nabla u_n\|_p \to \varphi_p$ strongly in $\W$ as $n \to +\infty$, and $\lambda^*=\lambda_1(p)$.
			\item\label{prop:behavior-gs:3} 
			Let $\lambda=\lambda^*$ and $M(\lambda^*)>-\infty$. 
			Then 
			$\lim\limits_{n\to+\infty}\In(u_n)=M(\lambda^*)$. 
		\end{enumerate}
	\end{proposition}
	\begin{proof}
		Since $\lambda_n<\lambda^*$, the ground state $u_n$ is a minimizer of $M(\lambda_n)$, and hence $|u_n|$ is also a ground state.
		Therefore, each $|u_n|$ is a solution of {\renewcommand{\lambdaM}{\lambda_n}\eqref{eq:P}}
		satisfying
		\begin{equation}\label{eq:gs:0}
			E_{\lambda_n}(|u_n|)=\intO a |u_n|^q\,dx>0. 
		\end{equation}
		Taking $\mu$ such that
		$\mu<\lambda =\lim\limits_{n\to+\infty}\lambda_n$ and recalling that $M$ is decreasing by 
		Proposition~\ref{prop:properties-m}~\ref{prop:properties-m:1}, we observe that 
		\begin{equation}\label{eq:gs:1} 
			\limsup_{n\to+\infty} I_{\lambda_n}(u_n)=
			\limsup_{n\to+\infty} M(\lambda_n) 
			\leq
			M(\mu)<0.
		\end{equation}

		\ref{prop:behavior-gs:1} 
		Let $\lambda<\lambda^*$.
		We claim that $\{u_n\}$ is bounded in $\W$.
		Suppose, by contradiction, that $\|\nabla u_n\|_p \to +\infty$.
		Clearly, we also have $\|\nabla |u_n|\|_p \to +\infty$.
		If $\lambda \not=\lambda_1(p)$, then the contradiction follows from Lemma~\ref{lem:bdd-PS} applied to $\{|u_n|\}$. 
		In the case $\lambda=\lambda_1(p)$, Lemma \ref{lem:bdd-PS} ensures that the sequence of $v_n:=|u_n|/\|\nabla u_n\|_p$ converges to $\varphi_p$ strongly in $\W$, up to a subsequence.
		Since $\lim\limits_{n\to+\infty}\intO a v_n^q\,dx\ge 0$ by \eqref{eq:gs:0}, we have 
		$\intO a \varphi_p^q\,dx \ge 0$, and hence $\lambda^*=\lambda_1(p)$ by Proposition~\ref{prop:+}.
		Thus, $\lambda < \lambda^*=\lambda_1(p)$, which contradicts our assumption $\lambda = \lambda_1(p)$.
		Consequently, $\{u_n\}$ is bounded in $\W$, and Lemma \ref{lem:conv-gs} in combination with \eqref{eq:gs:1} guarantees that $\{u_n\}$ converges to a ground state of \eqref{eq:P} strongly in $\W$, up to a subsequence. 
		
		\ref{prop:behavior-gs:2} 
		If $\{u_n\}$ has a bounded subsequence, then, as above, 
		Lemma \ref{lem:conv-gs} in combination with \eqref{eq:gs:1} justifies the existence of a ground state of \eqref{eq:P}, and hence $M(\lambda^*)>-\infty$, 
		which contradicts our assumption $M(\lambda^*)=-\infty$. 
		Therefore, we have $\lim\limits_{n\to+\infty}\|\nabla u_n\|_p=+\infty$, and Lemma~\ref{lem:bdd-PS} implies that 
		the sequence of $|u_n|/\|\nabla u_n\|_p$ converges to
		$\varphi_p$ strongly in $\W$, up to a subsequence.
		Finally, let us show that $\lim\limits_{n\to+\infty}\In(u_n)=-\infty$. 
		Fix any $R>0$. 
		Since $M(\lambda^*)=-\infty$, we can choose $v\in \mathcal{N}_{\lambda^*}\cap\mathcal{A}^+$ such that $I_{\lambda^*}(v)\le -R$.  
		Thus, for any sufficiently large $n$ we have 
		$t_{\lambda_n}(v) v\in\mathcal{N}_{\lambda_n} \cap \mathcal{A}^+$, and so 
		$$
		\In (u_n)=M(\lambda_n) \le \In(t_{\lambda_n}(v) v)\le \In(v).
		$$
		This yields $\limsup\limits_{n\to+\infty}\In (u_n)\le I_{\lambda^*}(v)\le -R$. 
		Since $R>0$ is arbitrary, we conclude that $\lim\limits_{n\to+\infty}\In(u_n)=-\infty$.

		\ref{prop:behavior-gs:3} 
		We know from Proposition \ref{prop:properties-m} \ref{prop:properties-m:1} and Remark \ref{rem:m<0} that $0>\In(u_n) \geq M(\lambda^*)$ for any $n \in \mathbb{N}$.
		Suppose, contrary to our claim, that there exists $C>0$ such that 
		$\In(u_n) \geq M(\lambda^*)+C$ for any $n$.
		Since $M(\lambda^*)>-\infty$, we can find $u \in \mathcal{N}_{\lambda^*} \cap \A^+$ satisfying
		\begin{equation}\label{eq:In-bound}
			0> \In(u_n) \geq M(\lambda^*)+C > I_{\lambda^*}(u) \geq M(\lambda^*).
		\end{equation}
		Due to our assumptions $\lambda_n<\lambda^*$ and $u \in \mathcal{N}_{\lambda^*} \cap \A^+$, we have $E_{\lambda_n}(u)>E_{\lambda^*}(u)>0$ for any $n$, and hence $t_{\lambda_n}(u)u \in \mathcal{N}_{\lambda_n} \cap \mathcal{A}^+$. 
		This yields $I_{\lambda_n}(t_{\lambda_n}(u)u) \geq M(\lambda_n) = I_{\lambda_n}(u_n)$. 
		On the other hand, by the continuity, we have $I_{\lambda_n}(t_{\lambda_n}(u)u) \to I_{\lambda^*}(u)$ as $n \to +\infty$.
		This contradicts \eqref{eq:In-bound} for all sufficiently large $n$.
	\end{proof}

	\begin{proof*}{Proposition~\ref{prop:PCV}}\label{page:proof:PCV}
		For convenience, we prove our assertions in a nondirect order.
		
		\ref{prop:PCV:4}
		Recall that $\lambda^*=\lambda_0^*$ and $\lambda_0^*$ is attained, see Proposition \ref{prop:+}.
		Let $u_0$ be a minimizer of $\lambda_0^*$.
		In particular, $u_0$ satisfies
		$E_{\lambda^*}(u_0)=0=\intO a|u_0|^q\,dx$.
		Fixing any $\lambda>\lambda^*$, we have $E_\lambda(u_0)<0$. 
		Assume first that there exists $v\in\W$ such that $\intO a|u_0|^{q-2}u_0v\,dx<0$,  
		and set $u_\varepsilon:=u_0+\varepsilon v$ for $\varepsilon>0$. 
		Then, for any sufficiently small $\varepsilon>0$ we have 
		$$
		E_\lambda (u_\varepsilon)<0
		\quad 
		\text{and}
		\quad 
		\intO a|u_\varepsilon|^{q} \,dx
		=
		q
		\intO \,\int_0^\varepsilon a|u_s|^{q-2}u_sv \,ds\,dx <0.
		$$
		Consequently, $t_\lambda(u_\varepsilon)u_\varepsilon \in \mathcal{N}_\lambda\cap\mathcal{A}^-$ 
		for such $\varepsilon>0$, and  we get 
		\begin{equation}\label{eq:m-} 
			M^-(\lambda) = \inf_{u\in \mathcal{N}_\lambda\cap\mathcal{A}^-} \I (u) 
			\leq \I(t_\lambda(u_\varepsilon)u_\varepsilon)
			=
			\J(u_\varepsilon) 
			=\frac{p-q}{p q}\, \frac{|\int_\Omega a |u_\varepsilon|^q \,dx|^{\frac{p}{p-q}}}{|E_\lambda(u_\varepsilon)|^{\frac{q}{p-q}}}
			\to 0 
		\end{equation}
		as $\varepsilon\to 0$ because $E_\lambda(u_0)<0$ and $\int_\Omega a |u_0|^q \,dx=0$. 
		
		Assume now that $\intO a|u_0|^{q-2}u_0v\,dx=0$ for any $v\in\W$.
		This implies that $a|u_0|^{q-2}u_0 \equiv 0$ a.e.\ in $\Omega$, and hence $u_0 \equiv 0$ a.e.\ in $\Omega_a^\pm$.
		Take any $v \in C_0^\infty(\Omega) \setminus \{0\}$ with the support in $\Omega_a^-$ and set $u_\varepsilon:=u_0+\varepsilon v$.
		Recalling that $\intO a |u_0|^q \,dx=0$, we obtain
		$$
		\intO a|u_\varepsilon|^{q}\,dx=
		\intO a|u_0|^{q}\,dx
		+
		\varepsilon^q
		\intO a|v|^{q}\,dx
		=
		\varepsilon^q
		\intO a|v|^{q}\,dx 
		< 0
		$$
		for any $\varepsilon>0$. 
		Moreover, since $\lambda>\lambda^*$, we have $E_\lambda (u_\varepsilon)<0$ for any sufficiently small $\varepsilon>0$ by the continuity.
		Arguing as in \eqref{eq:m-} above, we obtain the desired result.  
		
		\ref{prop:PCV:1}
		Let $\lambda_1(p)<\lambda<\mu\leq\lambda^*$ and let $u_\lambda$ be a minimizer of $M^-(\lambda)$ which exists by Theorem~\ref{prop:m-minus}.
		Noting that 
		$$
		E_\mu(u_\lambda)<E_\lambda(u_\lambda)=\intO a|u_\lambda|^q\,dx<0, 
		$$
		we get  $t_\mu(u_\lambda)u_\lambda\in\mathcal{N}_\mu\cap \mathcal{A}^-$, and hence
		\begin{align*} 
			M^-(\lambda)
			=\I (u_\lambda)=J_\lambda(u_\lambda) 
			&=
			\frac{p-q}{p q}\, \frac{\left|\int_\Omega a |u_\lambda|^q \,dx\right|^{\frac{p}{p-q}}}{\left|E_\lambda(u_\lambda)\right|^{\frac{q}{p-q}}}
			\\ 
			&>\frac{p-q}{p q}\, \frac{\left|\,\int_\Omega a |u_\lambda|^q \,dx\,\right|^{\frac{p}{p-q}}}{\left|E_\mu(u_\lambda)\right|^{\frac{q}{p-q}}}
			=J_\mu(u_\lambda)=I_\mu(t_\mu(u_\lambda)u_\lambda)\ge M^-(\mu).
		\end{align*} 
		That is, $M^-$ is decreasing on $(\lambda_1(p),\lambda^*]$.
		The fact that $M^-$ is nonincreasing on $(\lambda_1(p),+\infty)$ follows from the assertion \ref{prop:PCV:4}.

		\ref{prop:PCV:2}
		Let $u_\lambda\ge 0$ be a minimizer of $M^-(\lambda)$ 
		with $\lambda \in (\lambda_1(p),\lambda^*)$. 
		It is shown in \cite[Theorem 4.1 (2)]{KQU} that 
		$C_\lambda u_\lambda \to C\varphi_p$ 
		in $C^1(\overline{\Omega})$ as $\lambda\to \lambda_1(p)+0$, 
		where $C_\lambda,C>0$ are suitable normalization constants, 
		and $C$ is independent of $\lambda$. 
		Therefore, we have 
		$$
		E_\lambda(C_\lambda u_\lambda)\to 
		E_{\lambda_1(p)}(C\varphi_p)=0\quad {\rm and}\quad 
		\intO a|C_\lambda u_\lambda|^q\,dx\to 
		C^q\intO a\varphi_p^q\,dx<0, 
		$$
		as $\lambda\to \lambda_1(p)+0$, and hence 
		\begin{align*} 
			M^-(\lambda)=\I(u_\lambda)=J_\lambda(u_\lambda)=
			J_\lambda(C_\lambda u_\lambda)=
			\frac{p-q}{p q}\, \frac{|\int_\Omega a |C_\lambda u_\lambda|^q \,dx|^{\frac{p}{p-q}}}{|E_\lambda(C_\lambda u_\lambda)|^{\frac{q}{p-q}}}\to +\infty 
		\end{align*}
		as $\lambda\to \lambda_1(p)+0$.

		\ref{prop:PCV:5}
		Let us show that $M^-(\lambda) \to M^-(\lambda^*)$ as $\lambda \to \lambda^*-0$.
		In view of the monotonicity stated in the assertion \ref{prop:PCV:1}, we suppose, by contradiction, that there exists a sequence $\{\lambda_n\} \subset (\lambda_1(p),\lambda^*)$ such that $\lim\limits_{n\to +\infty}\lambda_n=\lambda^*$ and $M^-(\lambda^*) < \liminf\limits_{n \to +\infty}M^-(\lambda_n)$.
		Chose any $u \in \mathcal{N}_{\lambda^*} \cap \A^-$ such that 
		\begin{equation}\label{eq:mminus1}
			M^-(\lambda^*) \leq I_{\lambda^*}(u) < \liminf\limits_{n \to +\infty}M^-(\lambda_n).
		\end{equation}
		By the continuity, we have $E_{\lambda_n}(u) \to E_{\lambda^*}(u) < 0$ as $n \to +\infty$. 
		This yields $t_{\lambda_n}(u)u \in \mathcal{N}_{\lambda_n} \cap \A^-$ for any sufficiently large $n$, and $t_{\lambda_n}(u) \to 1$, and therefore
		$$
		M^-(\lambda_n) \leq I_{\lambda_n}(t_{\lambda_n}(u)u) \to I_{\lambda^*}(u)
		~~\text{as}~ n \to +\infty,
		$$
		which contradicts \eqref{eq:mminus1}.

		\ref{prop:PCV:3}
		The continuity of $M^-(\lambda)$ follows from Proposition \ref{prop:behavior:m^-} below.
	\end{proof*}

	\begin{remark}\label{rem:positive-least-energy_2}
		The assertion \ref{prop:PCV:4} of Proposition \ref{prop:PCV} coincides with that of \cite[Lemma~2.9~(2)]{QS}. 
		This lemma is more general in nature, but its application to the problem \eqref{eq:P} requires an additional assumption on $a$ (see \cite[Section 3.1]{QS}).
		Arguing in much the same way as in the proof of the assertion \ref{prop:PCV:4} above, one can show that $M^-(\lambda)=0$ for any $\lambda>\lambda_0^* \,(\ge \lambda^*)$, regardless the sign of $\intO a \varphi_p^q \,dx$.
	\end{remark}

	\begin{proposition}\label{prop:behavior:m^-} 
		Assume that $\intO a\varphi_p^q\,dx< 0$. 
		Let $\{\lambda_n\} \subset \mathbb{R}$ be a convergent sequence 
		such that $\lim\limits_{n\to+\infty}\lambda_n =: \lambda \in (\lambda_1(p),\lambda^*)$. 
		Let $u_n$ be a minimizer of $M^-(\lambda_n)$. 
		Then $\{u_n\}$ is bounded in $\W$ and it has a 
		subsequence strongly convergent in $\W$
		to a minimizer $u_0\in\N\cap\mathcal{A}^-$ of $M^-(\lambda)$. 
	\end{proposition}
	\begin{proof}
		Notice that $|u_n|$ shares with $u_n$ the property of being a minimizer of $M^-(\lambda_n)$. 
		In view of the assumption $\lambda>\lambda_1(p)$, we apply Lemma~\ref{lem:bdd-PS} to deduce the boundedness of $\{|u_n|\}$ and hence of $\{u_n\}$ in $\W$. 
		Thus, by Lemma \ref{lem:conv-gs-PS}, $\{u_n\}$ converges to a solution $u_0$ of \eqref{eq:P} strongly in $\W$,  up to a subsequence.
		We have $u_0\not\equiv 0$ in $\Omega$. 
		Indeed, taking $\mu \in (\lambda, \lambda^*)$, 
		Theorem~\ref{prop:m-minus} and Proposition  \ref{prop:PCV} \ref{prop:PCV:1} imply that 
		$M^-(\lambda_n)>M^-(\mu)>0$ for any sufficiently large $n$. 
		Therefore, 
		we get
		\begin{equation*}\label{eq:behavior:m^-:1}
			\I(u_0)=\lim_{n\to+\infty}I_{\lambda_n}(u_n)=
			\lim_{n\to+\infty}M^-(\lambda_n)\ge M^-(\mu)>0, 
		\end{equation*}
		whence $u_0\in\N\cap\mathcal{A}^-$ and, in particular, $u_0$ is nonzero.
		Finally, in order to prove that $u_0$ is a minimizer of $M^-(\lambda)$, let us show that 
		$\I(u_0)\le \I(w)$ for all $w\in\N\cap\mathcal{A}^-$. 
		Fix any $w\in\N\cap\mathcal{A}^-$. 
		Then 
		$t_{\lambda_n}(w)w\in \mathcal{N}_{\lambda_n} \cap \mathcal{A}^-$ 
		for all sufficiently large $n$, and we obtain 
		\begin{equation}\label{eq:behavior:m^-:4}
			I_{\lambda_n}(u_n)=M^-(\lambda_n)\le 
			I_{\lambda_n}(t_{\lambda_n}(w)w). 
		\end{equation} 
		Noting that $t_{\lambda_n}(w)\to 1$ by $w\in\N$, and passing to the limit as $n\to+\infty$ in 
		\eqref{eq:behavior:m^-:4}, we get our assertion. 
	\end{proof}

	
	\section{Existence after \texorpdfstring{$\lambda^*$}{lambda-*}. Proof of Theorem \ref{thm:1}}\label{sec:thm1}
	
	In this section, we prove Theorem \ref{thm:1}.
	Throughout the section, \textit{we always assume that either the assumption \ref{thm:1:1} or \ref{thm:1:2} of Theorem \ref{thm:1} is satisfied}.
	
	\subsection{First solution}
	We prove Theorem \ref{thm:1} \ref{thm:1:r1}  and the first part of Theorem \ref{thm:1} \ref{thm:1:r2} on the existence of a local minimum point.
	
	\subsubsection{Local minimizer}\label{sec:local}
	Our arguments will rely on the definitions and results from Section \ref{sec:K*}.
	Take a sufficiently small $\delta>0$ as required in Lemmas \ref{lem:2} and \ref{lem:3}.
	Let us show that 
	there exists $\widehat{\lambda}>\lambda^*$ such that for any $\lambda \in [\lambda^*,\widehat{\lambda})$ we have
	\begin{equation}\label{eq:proof:1}
		\inf \{\widetilde{I}_\lambda(u):~
		u \in K^*\}<
		\inf \{\widetilde{I}_\lambda(u):~
		u \in \partial K^*_\delta\},
	\end{equation}
	where the sets $K^*$, $K^*_\delta$, $\partial K^*_\delta$ are defined by \eqref{eq:K}, \eqref{eq:K*delta}, \eqref{eq:K*delta-par}, respectively. We emphasize that these sets do not depend on $\lambda$.
	By writing 
	$$
	\widetilde{I}_{\lambda^*}(u)
	= 
	\wI(u)
	+
	\frac{\lambda-\lambda^*}{p} \intO u_+^p \,dx,
	$$
	we obtain the following uniform estimate with respect to $u \in K^*_\delta$ provided $\lambda\ge \lambda^*$:
	$$
	0
	\le 
	\widetilde{I}_{\lambda^*}(u)
	-
	\wI(u)
	\leq 
	\frac{\lambda-\lambda^*}{p}\max\{\|v_+\|_p^p:\,v\in K^*_\delta\,\}.
	$$
	Noting that $K^*_\delta$ is compact in $L^p(\Omega)$ by Lemma \ref{lem:2}, we deduce the inequality \eqref{eq:proof:1} from Lemma \ref{lem:3}.

	Let us now consider a minimization problem
	$$
	M_0(\lambda) := \inf \{\widetilde{I}_\lambda(u):~
	u \in K^*_\delta\}
	$$
	for $\lambda \in (\lambda^*,\widehat{\lambda})$.
	Since $K^*_\delta$ is weakly sequentially compact (see Lemma \ref{lem:2}), any minimizing sequence of $M_0(\lambda)$ has a weakly convergent subsequence and 
	its weak limit $u_\lambda$ belongs to $K^*_\delta$.
	In view of the inequality \eqref{eq:proof:1} and the weak lower semicontinuity of $\widetilde{I}_\lambda$, $u_\lambda$ stays in the interior of $K^*_\delta$ and hence $M_0(\lambda)$ is attained by $u_\lambda$.
	Consequently, $u_\lambda$ is a local minimum point of $\widetilde{I}_\lambda$ and a nonnegative solution of \eqref{eq:P}.
	Clearly, $u_\lambda$ is nonzero and $\wI (u_\lambda)<0$ according to Lemma~\ref{lem:local-negative}. 
	Moreover, $u_\lambda$ is positive in $\Omega_a^+$ by Lemma \ref{lem:positivity}.
	
	\medskip
	We define a critical value
	\begin{equation}\label{eq:Lambda*}
		\Lambda^* 
		:= \sup\,
		\left\{ 
		\begin{array}{l|l} 
			\bar{\lambda}>\lambda^* & 
			\text{for any}~ \lambda \in (\lambda^*,\bar{\lambda}) 
			~\text{there exist}~ u \in \W 
			~\text{and neighborhood} \\ 
			& K ~\text{of}~ u ~\text{such that}~ 
			\wI(u)=\inf_{K}\wI<\inf_{\partial K}\wI \, 
		\end{array} 
		\right\}. 
	\end{equation} 
	From the above arguments, we have $\Lambda^* > \lambda^*$.
	Notice that $u$ in the definition of $\Lambda^*$ is a local minimum point of $\wI$ which may not be a strict local minimum point. 
	We conclude, as above, that $u$ is a nonnegative solution of \eqref{eq:P} such that 
	$\wI(u)<0$ and $u>0$ in $\Omega_a^+$.

	\subsubsection{Least \texorpdfstring{$\wI$}{I-tilde}-energy solution}\label{sec:global}
	
	Let us define a critical value
	\begin{equation}\label{eq:Lambda}
		\Lambda 
		:=
		\sup\{\lambda:~
		\eqref{eq:P} ~\text{possesses a nonnegative solution}~ u ~\text{such that}~ u>0 ~\text{in}~ \Omega_a^+
		\}.
	\end{equation}
	We know from Section \ref{sec:local} that $\Lambda \geq \Lambda^* > \lambda^*$.
	Moreover, it will be proved in Section \ref{sec:nonexistence1} that $\Lambda$ is finite. 
	
	We start by showing that on the whole interval $(\lambda^*,\Lambda)$, 
	\eqref{eq:P} has a nonnegative solution which is positive in $\Omega_a^+$. 
	The interval $(\lambda^*,\Lambda^*)$ is covered by Section~\ref{sec:local}.
	Thus, if $\Lambda=\Lambda^*$, then we are done.
	Assume that $\Lambda>\Lambda^*$ and take any $\lambda \in [\Lambda^*, \Lambda)$. 
	By the definition of $\Lambda$, there exists $\bar{\lambda} \in (\lambda, \Lambda]$ for which {\renewcommand{\lambdaM}{\bar{\lambda}}\eqref{eq:P}} possesses a nonnegative solution $\bar{u}$ such that $\bar{u}>0$ in $\Omega_a^+$. 
	Clearly, 
	$\bar{u}$ is a \textit{supersolution} of \eqref{eq:P}, i.e.,
	\begin{equation*}
		\intO |\nabla \bar{u}|^{p-2}\nabla \bar{u}\nabla\varphi \,dx 
		\geq \lambda \intO \bar{u}^{p-1}\varphi\,dx
		+\intO a \bar{u}^{q-1}\varphi\,dx 
	\end{equation*}
	for any nonnegative $\varphi \in \W$. 
	Moreover, recall that $\bar{u} \in C^{1,\beta}(\overline{\Omega})$ for some $\beta \in (0,1)$, see Remark \ref{rem:reg}.
	We take $\underline{u} = 0$ as a subsolution of \eqref{eq:P}. 
	It is not hard to verify that the set
	$$
	\mathcal{S}
	:=
	\{
	u \in \W:~ 0 \leq u \leq \bar{u} ~\text{a.e.\ in}~ \Omega
	\}
	$$
	is closed in $\W$ and convex, and hence $\mathcal{S}$ is weakly closed.
	Therefore, we deduce from \cite[Theorem~1.2]{struwe} that 
	some $u_\lambda\in\mathcal{S}$ delivers a minimum value of 
	$\I$ (and hence of $\wI$) over $\mathcal{S}$. 
	Then, arguing exactly as in the proof of \cite[Theorem 2.4]{struwe}, we conclude that $u_\lambda$ is a (nonnegative) solution of \eqref{eq:P}.
	
	Since $u_\lambda$ is a minimizer of $\I$ over $\mathcal{S}$, it is easily seen that that $u_\lambda$ is nonzero. 
	Indeed, consider any nonnegative $\varphi \in C_0^\infty(\Omega) \setminus \{0\}$ with the support in $\Omega_a^+$. 
	In view of the assumption $q<p$ and the inequality $\intO a \varphi^q \,dx >0$, we deduce that $\I(t\varphi)<0$ for any sufficiently small $t>0$.
	Recalling that $\bar{u}>0$ in $\Omega_a^+$ and ${\rm supp}\,\varphi\subset \Omega_a^+$, 
	we have $t \varphi \in \mathcal{S}$ provided $t>0$ is small enough.
	Thus, $\min\{\I(w):  w \in \mathcal{S}\}<0$, which yields $\I(u_\lambda)<0$ and $u_\lambda \not\equiv 0$ in $\Omega$. 
	
	Let us show that $u_\lambda>0$ in $\Omega_a^+$.
	We will argue similarly to the proof of Lemma \ref{lem:positivity}.
	Suppose, contrary to our claim, that $u_\lambda(x_0)=0$ for some $x_0 \in \Omega_a^+$. 
	By the strong maximum principle, we have $u_\lambda \equiv 0$ in a connected component $A$ of $\Omega_a^+$ containing $x_0$.
	Consider any nonnegative $\varphi \in C_0^\infty(\Omega) \setminus \{0\}$ with the support in $A$. 
	As above, since $q<p$ and $\intO a \varphi^q \,dx >0$, we have $ \I(t\varphi)<0$ for any sufficiently small $t>0$.
	Taking $t>0$ smaller, if necessary, we also get $u_\lambda+t\varphi \in \mathcal{S}$.
	Therefore, we arrive at the following contradiction:
	$$
	\I(u_\lambda) = \min\{\I(\omega):  \omega \in \mathcal{S}\} 
	\leq 
	\I(u_\lambda+t\varphi)
	=
	\I(u_\lambda) +  \I(t\varphi)
	< \I(u_\lambda).
	$$
	
	Finally, we prove that for any $\lambda \in (\lambda^*, \Lambda)$, \eqref{eq:P} possesses a least $\wI$-energy solution. 
	Let $\{u_n\}$ be a sequence of nonnegative solutions to \eqref{eq:P} such that $\{\wI(u_n)\}$ converges to the infimum of $\wI$ among all nonnegative solutions. 
	This infimum is negative in view of the inequality
	$\wI(u_\lambda)<0$.
	Since $\lambda>\lambda^*\geq \lambda_1(p)$, we apply Lemma \ref{lem:bdd-PS} to deduce that the sequence $\{u_n\}$ is bounded, and hence, in accordance with Lemma \ref{lem:conv-gs-PS}, $\{u_n\}$ converges to a least $\wI$-energy 
	solution $w_\lambda$ of \eqref{eq:P} strongly in $\W$, up to a subsequence.
	Applying Lemma~\ref{lem:negative} \ref{lem:negative:2} and noting Remark~\ref{rem:positivity}, 
	we deduce that $w_\lambda>0$ at least in some connected component of $\Omega_a^+$.

	\subsection{Mountain pass solution}\label{sec:secondsol}
	
	In this section, we prove the existence of another solution $v_\lambda$ of \eqref{eq:P} for any $\lambda \in (\lambda^*, \Lambda^*)$, as stated in Theorem \ref{thm:1} \ref{thm:1:r2}, where $\Lambda^*$ is defined by \eqref{eq:Lambda*}.
	First, we establish a slightly more general (but less precise) result.
	\begin{theorem}\label{thm:2}
		Let $\lambda > \lambda^*$.
		Let $u \in \W$ be a local minimum point of $\wI$. 
		Then there exists another critical point $v$ of $ \wI$, $v \neq u$, $\wI(u) \leq \wI(v)<0$, and $v>0$ in some connected component of $\Omega_a^+$.
	\end{theorem}
	\begin{proof} 
		Let $u$ be a local minimum point of $ \wI$. 
		Then $\wI(u)<0$ by Lemma \ref{lem:local-negative}.
		Let us take any $\omega\in\W$ satisfying $\wI(\omega)<\wI(u)$.
		Such $\omega$ exists since 
		$\widetilde{M}(\lambda) =  M(\lambda)= -\infty$ by Proposition~\ref{prop:wm=m}, Theorem~\ref{thm:GS}, and the assumption $\lambda>\lambda^*$. 
		By standard arguments, 
		we define the following mountain pass value: 
		\begin{equation*}
			c(\omega) := \inf_{\gamma\in\Gamma(\omega)} \max_{s\in[0,1]}  \wI(\gamma(s)), 
		\end{equation*} 	
		where
		\begin{equation}\label{eq:Gamma}
			\Gamma(\omega) 
			:=
			\left\{
			\gamma\in C([0,1],\W):~ \gamma(0) = u,~ 
			\gamma(1)=\omega
			\right\}. 
		\end{equation} 
		Since $\wI$ satisfies the Palais--Smale condition according to Lemma \ref{lem:PS}, 
		\cite[Theorem 1]{PS} implies that $c(\omega)$ is the critical level of $\wI$, and there exists a critical point $v=v(\omega)$ such that $c(\omega)=\wI(v)$ and $v \neq u$. 
		On the one hand, if $c(\omega)=\wI(u)$, then \cite[Theorem 1]{PS} ensures that 
		$v$ is also a local minimum point of $\wI$. 
		Consequently, we have
		$c(\omega)=\wI(v)<0$ and $v>0$ in $\Omega_a^+$ 
		by Lemma~\ref{lem:positivity}, and we are done.
		On the other hand, if $c(\omega)>\wI(u)$, then $v$ is a mountain pass solution 
		of \eqref{eq:P}. 
		If we know that $c(\omega)<0$, then 
		$v \not\equiv 0$ in $\Omega_a^+$, since otherwise Lemma~\ref{lem:negative}~\ref{lem:negative:2} 
		gives a contradiction to $\wI(v)=c(\omega)<0$. 
		That is, $v>0$  at least in one connected component of $\Omega_a^+$, see Remark~\ref{rem:positivity}. 
		Thus, in order to conclude our proof, it suffices to show that
		$c(\omega)<0$ for some function $\omega$ by  constructing a ``good'' path in $\Gamma(\omega)$.
		
		Recalling that ${M}(\lambda)= -\infty$ and 
		passing to the absolute value, 
		we have the existence of $\omega \in \wN$ such that $\omega \geq 0$ a.e.\ in $\Omega$ and $\wI(\omega)<\wI(u)$. 
		Consider the path
		$$
		\xi(s) 
		= 
		\left((1-s)u^q +s \omega^q\right)^{1/q}
		\quad \text{for}~ s\in [0,1]. 
		$$
		Since $u,\omega \geq 0$ a.e.\ in $\Omega$, we see that
		$$
		\xi(s) \geq 0 ~~\text{a.e.\ in}~~ \Omega,  
		\quad \text{and so}\quad \wI(\xi(s))=\I (\xi(s)) ~~\text{for any}~ s 
		\in [0,1]. 
		$$
		Moreover, by $\omega \in \widetilde{\mathcal{N}}_\lambda$ and $ \wI(\omega)<0$ we get $\intO a \omega^q \,dx > 0$, and hence
		\begin{equation}\label{eq:positiveint}
			\intO a (\xi(s))^q \,dx
			=
			(1-s) \intO a u^q \,dx
			+
			s \intO a \omega^q \,dx > 0
			~~\text{for any}~ s \in [0,1].
		\end{equation}
		
		If $\widetilde{E}_\lambda(\xi(s))>0$ for all $s\in [0,1]$, then 
		we readily see that $ \wI({t}_\lambda(\xi(s))\xi(s))<0$ for all $s \in [0,1]$, and hence $t_\lambda(\xi(s\cdot))\xi(\cdot)$ is the desired path 
		belonging to $\Gamma(\omega)$. 
		Recalling that $\widetilde{E}_\lambda(\xi(0))>0$ and $\widetilde{E}_\lambda(\xi(1))>0$, we suppose now that there exists $s_0 \in (0,1)$ satisfying $\widetilde{E}_\lambda(\xi(s_0))=0$. 
		Without loss of generality, we set 
		$$
		s_0 
		:= 
		\inf\{s\in (0,1):~ \widetilde{E}_\lambda(\xi(s))\le 0\}
		=
		\inf\{s\in (0,1):~ E_\lambda(\xi(s))\le 0\},
		$$ 
		whence $\widetilde{E}_\lambda (\xi(s))>0$ for all $s\in (0,s_0)$.
		In view of \eqref{eq:positiveint}, we deduce that $\I(t_\lambda(\xi(s))\xi(s))<0$  for all $s\in (0,s_0)$, and 
		$$
		\I(t_\lambda(\xi(s))\xi(s)) 
		= J_\lambda(\xi(s)) 
		=
		-\frac{p-q}{p q}\, 
		\frac{\left(\int_\Omega a u^q \,dx\right)^{\frac{p}{p-q}}}{(E_\lambda(u))^{\frac{q}{p-q}}}
		\to -\infty ~~\text{as}~~ s \nearrow s_0. 
		$$
		Thus, we can find $s_1 \in (0, s_0)$ such that 
		$\wI(t_\lambda(\xi(s_1))\xi(s_1))=\I(t_\lambda(\xi(s_1))\xi(s_1)) < \wI(u)$.
		Taking $\widetilde{\omega} := t_\lambda(\xi(s_1))\xi(s_1)$ and 
		considering the path $\eta(s) = t_\lambda(\xi(s_1 s))\xi(s_1 s)$ for $s\in[0,1]$, we see that $\eta\in \Gamma(\widetilde{\omega})$ 
		and $c(\widetilde{\omega})\le \max\limits_{s\in[0,1]}\wI(\eta(s))<0$, 
		which completes the proof. 
	\end{proof} 
	
	Now we are ready to prove the existence of the mountain pass solution of \eqref{eq:P} stated in Theorem~\ref{thm:1}~\ref{thm:1:r2}. 
	Let $\lambda \in (\lambda^*, \Lambda^*)$ and let $u$ be a local minimum point of $ \wI$ provided by the definition \eqref{eq:Lambda*} of $\Lambda^*$.
	Recall that $\wI(u)<0$ by Lemma \ref{lem:local-negative}.
	In view of the definition of $\Lambda^*$, 
	there exists a neighborhood $K$ of $u$ with the following property:
	\begin{equation}\label{eq:MP-2} 
		\inf\{\wI(u^\prime):~ u^\prime\in\partial K\}>
		\inf\{\wI(u^\prime):~ u^\prime\in K\}=\wI (u). 
	\end{equation}
	Thus, any $\omega \in \W$ such that $\wI(\omega)<\wI(u)$ must satisfy $\omega \not\in \overline{K}$. 
	Consequently, any path belonging to $\Gamma(\omega)$ (see \eqref{eq:Gamma}) intersects $\partial K$, and \eqref{eq:MP-2} yields $c(\omega)>\wI(u)$. 
	Arguing now exactly as in the second part of the proof of Theorem \ref{thm:2}, we obtain a mountain pass solution $v_\lambda$ such that 
	$0>c(\omega)(\text{or}\ c(\widetilde{\omega}))=\wI(v_\lambda)>\wI(u)$, and $v_\lambda>0$ in some connected component of $\Omega_a^+$. 
	This completes the proof of Theorem \ref{thm:1} \ref{thm:1:r2}.

	\subsection{Boundedness of \texorpdfstring{$\Lambda$}{Lambda}. Proof of Theorem \ref{thm:1} \ref{thm:1:r4}}\label{sec:nonexistence1}
	
	As a consequence of the definition \eqref{eq:Lambda} of $\Lambda$, for any $\lambda>\Lambda$ there is no nonnegative solution to \eqref{eq:P} which is positive in $\Omega_a^+$. 
	To make this statement nontrivial, we have to show that $\Lambda<+\infty$.
	In fact, we provide a slightly more general result, whose proof is rather standard anyway.
	Let us define 
	\begin{equation*}
		\Lambda_1 
		:=
		\sup\{\lambda:~
		\eqref{eq:P} ~\text{possesses a solution}~ u ~\text{such that}~ u>0 ~\text{in}~ \Omega_a^+
		\}.
	\end{equation*}
	That is, we do not require the solution $u$ to be nonnegative.
	Evidently, we have $\Lambda \leq \Lambda_1$.
	
	\begin{proposition}\label{prop:bounded}
		Assume that 
		$$
		\lambda > \lambda_1(p; \Omega_a^+) 
		:=
		\inf\left\{
		\frac{\int_{\Omega_a^+} |\nabla \varphi|^{p} \,dx}{\int_{\Omega_a^+}
			\varphi^p \, dx}:~ \varphi \in C_0^\infty(\Omega_a^+) \setminus \{0\}, ~ \varphi \geq 0
		\right\}.
		$$
		Then \eqref{eq:P} does not possess a solution $u$ satisfying $u>0$ in $\Omega_a^+$.
		In particular, $\Lambda \leq \Lambda_1 \leq \lambda_1(p;\Omega_a^+) < +\infty$.
	\end{proposition}
	\begin{proof}
		Let $u\in \W(\Omega)$ be a solution of \eqref{eq:P} with some $\lambda \in \mathbb{R}$ such that $u>0$ in $\Omega_a^+$.
		Recall that $u\in C^1_0(\overline{\Omega})$, see Remark \ref{rem:reg}. 
		Let us take any nonnegative $\varphi \in C_0^\infty(\Omega) \setminus \{0\}$ 
		satisfying ${\rm supp}\,\varphi\subset \Omega_a^+$. 
		Then there exists $c>0$ such that $u(x) \geq c$ for any $x \in \text{supp}\, \varphi$.
		Therefore, 
		$\frac{\varphi}{u} \in L^\infty(\Omega)$, and so, by the regularity,  $\frac{\varphi^p}{u^{p-1}} \in C_0^1(\Omega_a^+)$ and we can use it as a test function for \eqref{eq:P}. 
		Applying the standard Picone inequality \cite[Theorem 1.1]{Alleg}, we get
		\begin{align*}
			\int_{\Omega_a^+} |\nabla \varphi|^{p} \,dx
			&\geq
			\int_{\Omega_a^+} |\nabla u|^{p-2} \nabla u\nabla \left(\frac{\varphi^p}{u^{p-1}}\right) dx\\
			&=\lambda \int_{\Omega_a^+} 
			\varphi^p \, dx 
			+ 
			\int_{\Omega_a^+} a(x) u^{q-p}
			\varphi^p \,dx
			\geq
			\lambda \int_{\Omega_a^+} \varphi^p \, dx.
		\end{align*}
		Consequently, $\lambda$ must be bounded from above as follows:
		$$
		\lambda \leq \frac{\int_{\Omega_a^+} |\nabla \varphi|^{p} \,dx}{\int_{\Omega_a^+}
			\varphi^p \, dx} < +\infty.
		$$
		Minimizing over all such $\varphi$, we conclude that
		$\lambda \leq \lambda_1(p;\Omega_a^+)$.
	\end{proof}
	
	\begin{remark}
		The same argument as in the proof of Proposition \ref{prop:bounded} shows that
		\eqref{eq:P} {\it does not} possess a solution $u$ satisfying $u>0$ in a 
		connected component $A$ of $\Omega_a^+$ 
		provided $\lambda>\lambda_1(p;A)$. 
		In particular, if $\{A_i\}$ is the set of all connected components of $\Omega_a^+$ and $\max_i \lambda_1(p;A_i)<+\infty$, then \eqref{eq:P} does not possess a solution $u$ satisfying $0 \not\equiv u \geq 0$ in $\Omega_a^+$ whenever $\lambda > \max_i \lambda_1(p;A_i)$.
	\end{remark}

	\section{Proof of Theorem \ref{thm:3}}\label{sec:thm3}

	Let us note that $\lambda_1(p)=\lambda^*$ since we impose the assumption $\intO a\varphi_p^q\,dx=0$, see Proposition~\ref{prop:+}. 
	Recall that 
	$\wI^\mu$ is the truncated functional defined by 
	\eqref{def:functionals-t} with the weight function 
	$a_\mu=a+\mu b$ replacing $a$. 
	
	The proof of the first part of Theorem \ref{thm:3} goes along the same lines as the proof of Theorem \ref{thm:1}. 
	Indeed, since $\wIm$ is a continuous perturbation of $\widetilde{I}_{\lambda^*} (\equiv \widetilde{I}_{\lambda^*}^0)$, we have the analog of the key inequality \eqref{eq:proof:1} for $\wIm$. 
	Namely, there exist $\widehat{\mu}>0$ and $\varepsilon>0$ such that for any $\mu \in (-\widehat{\mu},\widehat{\mu})$ and any $\lambda \in [\lambda_1(p)-\varepsilon,\lambda_1(p)+\varepsilon]$ we have
	\begin{equation}\label{eq:proof:1m}
		\inf \{\wIm(u):~
		u \in K^*\}<
		\inf \{\wIm(u):~
		u \in \partial K^*_\delta\}.
	\end{equation}
	Here, $K^*$, $K^*_\delta$, and $\partial K^*_\delta$ are the sets defined in Section \ref{sec:K*}, and these sets are independent of $\mu$ and $\lambda$.
	The case $\mu \leq 0$ implies $\intO a_\mu \varphi_p^q \,dx \leq 0$, which is covered by Theorem \ref{thm:1}.
	That is why we interested only in $\mu \in (0,\widehat{\mu})$, for which there holds $\intO a_\mu \varphi_p^q \,dx > 0$.
	
	Fixing $\mu \in (0,\widehat{\mu})$ and considering a minimization problem
	$$
	M_\mu(\lambda) := \inf \{\wIm(u):~
	u \in K^*_\delta\}
	$$
	for $\lambda \in [\lambda_1(p)-\varepsilon,\lambda_1(p)+\varepsilon]$, we deduce as in Section \ref{sec:local} that $M_\mu(\lambda)$ is attained by a critical point $u_\lambda \in K^*_\delta$ of $\wIm$ which is a local minimum point, $\wIm(u_\lambda)<0$ by Lemma \ref{lem:local-negative}, and $u_\lambda$ is positive in $\Omega_a^+$ in view of Lemma \ref{lem:positivity}.
	Arguing now exactly as in Sections \ref{sec:global}, \ref{sec:secondsol}, and \ref{sec:nonexistence1}, we obtain
	the analogs of the assertions \ref{thm:1:r1}, \ref{thm:1:r2}, \ref{thm:1:r4} of Theorem \ref{thm:1}  for the problem \eqref{eq:Pm} and the corresponding functional $\wIm$.
	
	Let us prove the second part of Theorem \ref{thm:3} on the existence of three solutions in a left neighborhood of $\lambda_1(p)$.
	Since $u_\lambda\in K^*_\delta$, $K^*_\delta$ is compact in $L^p(\Omega)$ and 
	$L^q(\Omega)$ (see Lemma \ref{lem:2}), and $K^*_\delta$ is independent of $\lambda$ and $\mu$, 
	we obtain the following uniform lower bound on 
	$\wIm(u_\lambda)$:
	\begin{equation*} 
		\wIm(u_\lambda)
		\ge 
		-\frac{\lambda_1(p)}{p}\max\{\|v_+\|_p^p:\,v\in K^*_\delta\,\}
		-\frac{\|a\|_\infty+\widehat{\mu} \|b\|_\infty}{q} \, 
		\max\{\|v_+\|_q^q:\,v\in K^*_\delta\,\} 
	\end{equation*} 
	for any 
	$\lambda\in[\lambda_1(p)-\varepsilon,\lambda_1(p)]$.
	At the same time, in view of the inequality $\intO a_\mu\varphi_p^q\,dx>0$, the global minimum point $w_\lambda$ of $\wIm$ given by Theorem \ref{thm:GS} \ref{thm:GS:1} for $\lambda<\lambda_1(p)$ satisfies $\wIm(w_\lambda) \to -\infty$ as $\lambda \to \lambda_1(p)-0$, 
	see Proposition~\ref{prop:properties-m}~\ref{prop:properties-m:4} 
	in combination with  Theorem~\ref{NoGroundStates} \ref{NoGroundStates:3}. 
	Thus, we conclude that there exists a sufficiently small $\epsilon>0$ such that
	\begin{equation}\label{eq:iwiu}
		\wI^\mu(w_\lambda)<\wI^\mu(u_\lambda)<0
		\quad\text{for any}~ 
		\lambda \in (\lambda_1(p)-\epsilon,\lambda_1(p)).
	\end{equation}
	In particular, $w_\lambda$ is different from $u_\lambda$ for such $\lambda$.
	Finally, using \eqref{eq:iwiu} and arguing as in Section~\ref{sec:secondsol} (see also \cite[Corollary 1]{PS}), we establish the existence of the third critical point $v_\lambda$ of $\wIm$ for any $\lambda \in (\lambda_1(p)-\epsilon,\lambda_1(p))$. 
	Thanks to the inequality \eqref{eq:proof:1m}, this critical point has the mountain pass type and satisfies $\wIm(w_\lambda) < \wIm(u_\lambda) < \wIm(v_\lambda)<0$.
	Therefore, we obtain three different nonnegative solutions of \eqref{eq:P}.
	The proof is complete.

	\section{Nonexistence after \texorpdfstring{$\lambda^*$}{lambda-*}. Proof of Theorem~\ref{NoSol}}\label{sec:nonexistence}
	
	Let $u\in \W$ be a nonnegative solution of \eqref{eq:P} such that $u>0$ in $\Omega_a^+$.
	Recall that $u\in C^1_0(\overline{\Omega})$ by Remark~\ref{rem:reg}. 
	Fix any $\varepsilon>0$. Then 
	$\frac{\varphi_p}{u+\varepsilon} \in L^\infty(\Omega)$, and so 
	we can choose $\frac{\varphi_p^q}{(u+\varepsilon)^{q-1}}$ as a test function for \eqref{eq:P}. 
	Applying the classical Picone inequality \cite[Theorem 1.1]{Alleg} and the generalized Picone inequality \cite[Theorem 1.8]{BT_Picone}, we get
	\begin{align*}
		\lambda \intO \left(\dfrac{u}{u+\varepsilon}\right)^{q-1}\,
		u^{p-q} \varphi_p^q \, dx 
		&+ \intO a\left(\dfrac{u}{u+\varepsilon}\right)^{q-1}\varphi_p^q \,dx 
		\\ 
		=
		\intO |\nabla u|^{p-2} \nabla u\nabla \left(\frac{\varphi_p^q}{(u+\varepsilon)^{q-1}}\right) dx 	
		&\leq
		\intO |\nabla \varphi_p|^{p-2} \nabla \varphi_p \nabla \left(\frac{\varphi_p^{q-p+1}}{(u+\varepsilon)^{q-p}}\right) dx
		\\
		&=
		\lambda_1(p) \intO \varphi_p^q\,(u+\varepsilon)^{p-q} \, dx, 
	\end{align*}
	where the last equality follows since $\varphi_p$ is the first eigenfunction of the $p$-Laplacian.
	Letting $\varepsilon\to 0$, we obtain 
	\begin{equation}\label{eq:NoSol:2}
		(\lambda_1(p)-\lambda)\intO u^{p-q}\varphi_p^q\,dx 
		\ge \int_{\{x\in\Omega:\, u(x)>0\}}\,a\varphi_p^q\,dx. 
	\end{equation} 
	Recalling that $u>0$ in $\Omega_a^+$ and $\varphi_p>0$ in $\Omega$, we get
	\begin{equation}\label{eq:NoSol:3}
		\int_{\{x\in\Omega:\, u(x)>0\}}\,a\varphi_p^q\,dx 
		\ge \intO a\varphi_p^q\,dx. 
	\end{equation} 
	Combining \eqref{eq:NoSol:2} and \eqref{eq:NoSol:3}, we arrive at 
	\begin{equation*}\label{eq:NoSol:4} 
		(\lambda_1(p)-\lambda)\intO u^{p-q}\varphi_p^q\,dx 
		\ge \intO a\varphi_p^q\,dx. 
	\end{equation*} 
	This leads to either
	$\lambda<\lambda_1(p)$ or 
	$\lambda\le \lambda_1(p)$, provided
	$\intO a\varphi_p^q\,dx> 0$ or $\intO a\varphi_p^q\,dx= 0$, 
	respectively, which proves the theorem.

	\bigskip
	\noindent
	\textbf{Acknowledgments.}
	V.~Bobkov was supported in the framework of implementation of the development program of Volga Region Mathematical Center (agreement no.~075-02-2021-1393).
	M.~Tanaka was supported by JSPS KAKENHI Grant Number JP 19K03591.



\addcontentsline{toc}{section}{\refname}
	\small
	\begin{thebibliography}{99}
		
		
		\bibitem{Alama1}
		Alama, S. (1999). Semilinear elliptic equations with sublinear indefinite nonlinearities. Advances in Differential Equations, 4(6), 813-842.	
		\url{http://projecteuclid.org/euclid.ade/1366030748}
		
		\bibitem{alama-del}
		Alama, S., \& Del Pino, M. (1996). Solutions of elliptic equations with indefinite nonlinearities via Morse theory and linking. Annales de l'Institut Henri Poincar\'e C, Analyse non lin\'eaire 13(1), 95-115. 
		\href{http://dx.doi.org/10.1016/S0294-1449(16)30098-1}{DOI:10.1016/S0294-1449(16)30098-1}
		
		\bibitem{alama-tar}
		Alama, S., \& Tarantello, G. (1993). On semilinear elliptic equations with indefinite nonlinearities. Calculus of Variations and Partial Differential Equations, 1(4), 439-475.	
		\href{http://dx.doi.org/10.1007/BF01206962}{DOI:10.1007/BF01206962}
		
		
		\bibitem{Alleg}
		Allegretto, W., \& Huang, Y. (1998). A Picone's identity for the $p$-Laplacian and applications. Nonlinear Analysis: Theory, Methods \& Applications, 32(7), 819-830. 
		\href{http://dx.doi.org/10.1016/S0362-546X(97)00530-0}{DOI:10.1016/S0362-546X(97)00530-0}
		
		\bibitem{anane1987}
		Anane, A. (1987). 
		Simplicit\'e et isolation de la premiere valeur propre du $p$-laplacien avec poids.
		Comptes Rendus de l'Acad\'emie des Sciences-Series I-Mathematics, 
		305(16) (1987), 725--728.
		\url{http://gallica.bnf.fr/ark:/12148/bpt6k57447681/f27}	
		
		\bibitem{BDO}
		Balabane, M., Dolbeault, J., \& Ounaies, H. (2003). Nodal solutions for a sublinear elliptic equation. Nonlinear Analysis: Theory, Methods \& Applications, 52(1), 219-237.
		\href{http://dx.doi.org/10.1016/S0362-546X(02)00104-9}{DOI:10.1016/S0362-546X(02)00104-9}
		
		\bibitem{BPT1}
		Bandle, C., Pozio, M. A., \& Tesei, A. (1987). The asymptotic behavior of the solutions of degenerate parabolic equations. Transactions of the American Mathematical Society, 303(2), 487-501.
		\href{http://dx.doi.org/10.1090/S0002-9947-1987-0902780-3}{DOI:10.1090/S0002-9947-1987-0902780-3}
		
		\bibitem{berest}
		Berestycki, H., Capuzzo-Dolcetta, I., \& Nirenberg, L. (1995). Variational methods for indefinite superlinear homogeneous elliptic problems. Nonlinear Differential Equations and Applications NoDEA, 2(4), 553-572.
		\href{http://dx.doi.org/10.1007/BF01210623}{DOI:10.1007/BF01210623}
		
		\bibitem{BobkovTanaka2017}
		Bobkov, V., \& Tanaka, M. (2018). Remarks on minimizers for $(p, q)$-Laplace equations with two parameters. 
		Communications on Pure and Applied Analysis, 17(3), 1219-1253.
		\href{http://dx.doi.org/10.3934/cpaa.2018059}{DOI:10.3934/cpaa.2018059}
		
		\bibitem{BT_Picone} 
		Bobkov, V., \& Tanaka, M. (2020). Generalized Picone inequalities and their applications to $(p,q)$-Laplace equations. Open Mathematics, 18(1), 1030-1044.
		\href{http://dx.doi.org/10.1515/math-2020-0065}{DOI:10.1515/math-2020-0065}
		
		\bibitem{BT2021}
		Bobkov, V., \& Tanaka, M. (2021). Multiplicity of positive solutions for $(p,q)$-Laplace equations with two parameters. Communications in Contemporary Mathematics, 2150008.
		\href{http://dx.doi.org/10.1142/S0219199721500085}{DOI:10.1142/S0219199721500085}
		
		\bibitem{BDSWPT}
		Bonheure, D., Santos, E. M. D., Parini, E., Tavares, H., \& Weth, T. (2020). Nodal Solutions for sublinear-type problems with Dirichlet boundary conditions. 
		International Mathematics Research Notices, rnaa233.
		\href{http://dx.doi.org/10.1093/imrn/rnaa233}{DOI:10.1093/imrn/rnaa233}
		
		
		\bibitem{brasco}
		Brasco, L., \& Franzina, G. (2019). An overview on constrained critical points of Dirichlet integrals. Rendiconti del Seminario Matematico, Universit\`a e Politecnico di Torino, 78(2), 7-50. 
		\url{http://www.seminariomatematico.polito.it/rendiconti/78-2.html}
		
		
		\bibitem{brown}
		Brown, K. J. (2004). The Nehari manifold for a semilinear elliptic equation involving a sublinear term. 
		Calculus of Variations and Partial Differential Equations, 22(4), 483-494.
		\href{http://dx.doi.org/10.1007/s00526-004-0289-2}{DOI:10.1007/s00526-004-0289-2}
		

		\bibitem{CT}
		Cuesta, M., \& Tak\'a\v{c}, P. (2000). A strong comparison principle for positive solutions of degenerate elliptic equations. Differential and Integral Equations, 13(4-6), 721-746.
		\url{http://projecteuclid.org/euclid.die/1356061247}	
		
		\bibitem{diaz}
		D\'iaz, J. I. (1985). Nonlinear partial differential equations and free boundaries. 
		Vol.\ 1: Elliptic Equations. 
		Pitman Advanced Publishing Program, Boston-London-Melbourne. 
		
		
		\bibitem{DH}
		D\'iaz, J. I., \& Hern\'andez, J. (1999). Global bifurcation and continua of nonnegative solutions for a quasilinear elliptic problem. Comptes Rendus de l'Acad\'emie des Sciences-Series I-Mathematics, 329(7), 587-592.
		\href{http://dx.doi.org/10.1016/S0764-4442(00)80006-3}{DOI:10.1016/S0764-4442(00)80006-3}
		
		
		\bibitem{DHI1}
		D\'iaz, J. I., Hern\'andez, J., \& Il'yasov, Y. (2015). On the existence of positive solutions and solutions with compact support for a spectral nonlinear elliptic problem with strong absorption. Nonlinear Analysis: Theory, Methods \& Applications, 119, 484-500.
		\href{http://dx.doi.org/10.1016/j.na.2014.11.019}{DOI:10.1016/j.na.2014.11.019}
		
		
		\bibitem{dinica}
		Dinca, G., Jebelean, P., \& Mawhin, J. (2001). Variational and topological methods for Dirichlet problems with $p$-Laplacian. Portugaliae Mathematica, 58(3), 339.
		\url{https://eudml.org/doc/49321}
		
		\bibitem{DM}
		Dr\'abek, P., \& Man\'asevich, R. (1999). On the closed solution to some nonhomogeneous eigenvalue problems with $p$-Laplacian. Differential and Integral Equations, 12(6), 773-788.
		\url{http://projecteuclid.org/euclid.die/1367241475}	
		
		\bibitem{drabek-poh}
		Dr\'abek, P., \& Pohozaev, S. I. (1997). Positive solutions for the $p$-Laplacian: application of the fibrering method. Proceedings of the Royal Society of Edinburgh Section A: Mathematics, 127(4), 703-726.
		\href{http://dx.doi.org/10.1017/S0308210500023787}{DOI:10.1017/S0308210500023787}
		
		
		\bibitem{takac}
		Fleckinger-Pell\'e J., \& Tak\'a\v{c}, P.
		An improved Poincar\'e inequality and the $p$-Laplacian at resonance for $p>2$.
		Advances in Differential Equations, 7(8), 951-971.
		\url{http://projecteuclid.org/euclid.ade/1356651685}
		
		
		\bibitem{FLS}
		Franchi, B., Lanconelli, E., \& Serrin, J. (1996). Existence and uniqueness of nonnegative solutions of quasilinear equations in $R^n$. Advances in Mathematics, 118(2), 177-243.
		\href{http://dx.doi.org/10.1006/aima.1996.0021}{DOI:10.1006/aima.1996.0021}
		
		
		\bibitem{ilyasov2}
		Il'yasov, Y. (2001). On positive solutions of indefinite elliptic equations. Comptes Rendus de l'Acad\'emie des Sciences-Series I-Mathematics, 333(6), 533-538.
		\href{http://dx.doi.org/10.1016/S0764-4442(01)01924-3}{DOI:10.1016/S0764-4442(01)01924-3}
		
		\bibitem{ilyasov1}
		Il'yasov, Y. S. (2002). Non-local investigation of bifurcations of solutions of non-linear elliptic equations. Izvestiya: Mathematics, 66(6), 1103-1130.
		\href{http://dx.doi.org/10.1070/IM2002v066n06ABEH000408}{DOI:10.1070/IM2002v066n06ABEH000408}
		
		\bibitem{IS}
		Ilyasov, Y., \& Silva, K. (2018). On branches of positive solutions for $p$-Laplacian problems at the extreme value of the Nehari manifold method. Proceedings of the American Mathematical Society, 146(7), 2925-2935.
		\href{http://dx.doi.org/10.1090/proc/13972}{DOI:10.1090/proc/13972}
		
		
		
		\bibitem{kaji1}
		Kajikiya, R. (2016). Symmetric mountain pass lemma and sublinear elliptic equations. Journal of Differential Equations, 260(3), 2587-2610.
		\href{http://dx.doi.org/10.1016/j.jde.2015.10.016}{DOI:10.1016/j.jde.2015.10.016}
		
		\bibitem{KQU2}
		Kaufmann, U., Quoirin, H. R., \& Umezu, K. (2020). A curve of positive solutions for an indefinite sublinear Dirichlet problem. 
		Discrete \& Continuous Dynamical Systems, 40(2), 617-645.
		\href{http://dx.doi.org/10.3934/dcds.2020063}{DOI:10.3934/dcds.2020063}		
		
		\bibitem{KQU1}
		Kaufmann, U., Quoirin, H. R., \& Umezu, K. (2020). 
		Past and recent contributions to indefinite sublinear elliptic problems. 
		Rendiconti dell'Istituto di Matematica dell'Universit\'a di Trieste, 52, 217-241.
		\href{http://dx.doi.org/10.13137/2464-8728/30913}{DOI:10.13137/2464-8728/30913}
		
		
		\bibitem{KQU} 
		Kaufmann, U., \& Quoirin, H. R., \& Umezu, K. (2021). Uniqueness and positivity issues 
		in a quasilinear indefinite problem. 
		Calculus of Variations and Partial Differential Equations, 60(5), 187.
		\href{http://dx.doi.org/10.1007/s00526-021-02057-8}{DOI:10.1007/s00526-021-02057-8}
		
		
		\bibitem{Lieberman}
		Lieberman, G. M. (1988). Boundary regularity for solutions of degenerate elliptic equations. Nonlinear Analysis: Theory, Methods \& Applications, 12(11), 1203-1219.
		\href{http://dx.doi.org/10.1016/0362-546X(88)90053-3}{DOI:10.1016/0362-546X(88)90053-3}
		
		
		\bibitem{lind}
		Lindqvist, P. (2019). Notes on the stationary $p$-Laplace equation. Springer.
		\href{http://dx.doi.org/10.1007/978-3-030-14501-9}{DOI:10.1007/978-3-030-14501-9}
		
		\bibitem{MMT}
		Miyajima, S., Motreanu, D., \& Tanaka, M. (2012). Multiple existence results of solutions for the Neumann problems via super-and sub-solutions. Journal of Functional Analysis, 262(4), 1921-1953.
		\href{http://dx.doi.org/10.1016/j.jfa.2011.11.028}{DOI:10.1016/j.jfa.2011.11.028}
		
		
		\bibitem{moroz}
		Moroz, V. (2003). On the Morse critical groups for indefinite sublinear elliptic problems. Nonlinear Analysis: Theory, Methods \& Applications, 52(5), 1441-1453.
		\href{http://dx.doi.org/10.1016/S0362-546X(02)00174-8}{DOI:10.1016/S0362-546X(02)00174-8}
		
		
		
		\bibitem{muller}
		M\"uller, C. (1954). On the behavior of the solutions of the differential equation $\Delta U= F(x,U)$ in the neighborhood of a point. Communications on Pure and Applied Mathematics, 7(3), 505-515.
		\href{http://dx.doi.org/10.1002/cpa.3160070304}{DOI:10.1002/cpa.3160070304}
		
		
		
		\bibitem{ouyang}
		Ouyang, T. (1991). On the Positive Solutions of Semilinear Equations $\Delta u+ \lambda u+ h u^p= 0$ on Compact Manifolds. Part II. Indiana University Mathematics Journal, 40(3), 1083-1141.
		\url{https://www.jstor.org/stable/24896322}
		
		\bibitem{PS}
		Pucci, P., \& Serrin, J. (1985). A mountain pass theorem. Journal of Differential Equations, 60(1), 142-149.
		\href{http://dx.doi.org/10.1016/0022-0396(85)90125-1}{DOI:10.1016/0022-0396(85)90125-1}
		
		\bibitem{PS2}
		Pucci, P., \& Serrin, J. (2004). The strong maximum principle revisited. Journal of Differential Equations, 196(1), 1-66.
		\href{http://dx.doi.org/10.1016/j.jde.2003.05.001}{DOI:10.1016/j.jde.2003.05.001}
		
		\bibitem{QS}
		Quoirin, H. R., \& Silva, K. (2021). On the Nehari set for a class of functionals depending on a parameter and having two homogeneous terms.
		\href{https://arxiv.org/abs/2107.00777}{arXiv:2107.00777}
		
		
		\bibitem{SM}
		Silva, K., \& Macedo, A. (2018). Local minimizers over the Nehari manifold for a class of concave-convex problems with sign changing nonlinearity. Journal of Differential Equations, 265(5), 1894-1921.
		\href{http://dx.doi.org/10.1016/j.jde.2018.04.018}{DOI:10.1016/j.jde.2018.04.018}
		
		
		\bibitem{struwe}
		Struwe, M. (2000). Variational methods (Vol. 991). Springer-Verlag.
		\href{http://dx.doi.org/10.1007/978-3-540-74013-1}{DOI:10.1007/978-3-540-74013-1}
		
		
		\bibitem{zeidler}
		Zeidler, E. (1985). 
		Nonlinear Functional Analysis and its Application III:
		Variational Methods and Optimization. 
		Springer-Verlag.
		\href{http://dx.doi.org/10.1007/978-1-4612-5020-3}{DOI:10.1007/978-1-4612-5020-3}
		
	\end{thebibliography}
\end{document}